\documentclass[11pt]{amsart}
\usepackage{setspace}
\usepackage[backref=page]{hyperref}
\usepackage[top=2.7cm,left=2.8cm,right=2.8cm,bottom=2.6cm]{geometry}
\usepackage[usenames,dvipsnames]{xcolor}
\hypersetup{colorlinks=true,citecolor=NavyBlue,linkcolor=Brown,urlcolor=Orange}

\usepackage{amssymb, comment, paralist, xspace, graphicx, url, amscd, euscript, mathrsfs,stmaryrd,epic,eepic,color}
\usepackage{wrapfig} 
\usepackage{tikz-cd}
\usepackage[all]{xy}
\usepackage{amsthm}
\usepackage{amsmath}
\usepackage{amsfonts}
\usepackage{ dsfont }

\usepackage{enumerate}
\usepackage{xcolor}
\usepackage{tikz}
\usetikzlibrary{shapes,arrows,decorations.pathreplacing,backgrounds,positioning,fit,matrix}
\tikzstyle{block} = [rectangle, draw, fill=blue!20, 
    text width=5em, text centered, rounded corners, minimum height=4em]
\usepackage{float}

\usepackage{bbm}
\usepackage{dsfont}
\usepackage{cleveref}

\SelectTips{cm}{}

\calclayout

1

\allowdisplaybreaks[1]

\newenvironment{enumeratea}
{\begin{enumerate}[\upshape (a)]}
{\end{enumerate}}

\newenvironment{enumeratei}
{\begin{enumerate}[\upshape (i)]}
{\end{enumerate}}

\newenvironment{enumerate1}
{\begin{enumerate}[\upshape (1)]}
{\end{enumerate}}

\newtheorem*{namedtheorem}{\theoremname}
\newcommand{\theoremname}{testing}

\newtheorem{theorem}{Theorem}[section]
\newtheorem{proposition}[theorem]{Proposition}
\newtheorem{proposition-definition}[theorem]{Proposition-Definition}
\newtheorem{corollary}[theorem]{Corollary}
\newtheorem{lemma}[theorem]{Lemma}
\newtheorem{conjecture}[theorem]{Conjecture}

\theoremstyle{definition}
\newtheorem{definition}[theorem]{Definition}

\newtheorem*{question*}{Question}
\newtheorem{condition}[theorem]{Condition}
\newtheorem{notation}[theorem]{Notation}
\newtheorem{example}[theorem]{Example}
\newtheorem{remark}[theorem]{Remark}

\newtheorem{assumption}[theorem]{Assumption}
\theoremstyle{remark}

\newcommand\cA{\mathcal{A}}

\newcommand\cF{\mathcal{F}}

\newcommand\cI{\mathcal{I}}

\newcommand\cK{\mathcal{K}}
\newcommand\cL{\mathcal{L}}
\newcommand\cM{\mathcal{M}}

\newcommand\cO{\mathcal{O}}
\newcommand\cP{\mathcal{P}}

\newcommand\cS{\mathcal{S}}

\newcommand\cX{\mathcal{X}}
\newcommand\cY{\mathcal{Y}}
\newcommand\cZ{\mathcal{Z}}

\newcommand\uX{\underline{X}}

\renewcommand\AA{\mathbb{A}}

\newcommand\CC{\mathbb{C}}
\newcommand\DD{\mathbb{D}}

\newcommand\FF{\mathbb{F}}
\newcommand\GG{\mathbb{G}}

\newcommand\LL{\mathbb{L}}

\newcommand\NN{\mathbb{N}}

\newcommand\PP{\mathbb{P}}
\newcommand\QQ{\mathbb{Q}}
\newcommand\RR{\mathbb{R}}

\newcommand\ZZ{\mathbb{Z}}

\newcommand\bR{\mathbf{R}}

\newcommand\rA{\mathrm{A}}

\newcommand\rC{\mathrm{C}}
\newcommand\rD{\mathrm{D}}
\newcommand\rE{\mathrm{E}}

\newcommand\rH{\mathrm{H}}
\newcommand\rI{\mathrm{I}}

\newcommand\rK{\mathrm{K}}
\newcommand\rL{\mathrm{L}}

\newcommand\rN{\mathrm{N}}
\newcommand\rO{\mathrm{O}}
\newcommand\rP{\mathrm{P}}

\newcommand\rR{\mathrm{R}}
\newcommand\rS{\mathrm{S}}
\newcommand\rT{\mathrm{T}}
\newcommand\rU{\mathrm{U}}

\newcommand\rmd{\mathrm{d}}

\newcommand\rms{\mathrm{s}}
\newcommand\rmt{\mathrm{t}}

\newcommand\frh{\mathfrak{h}}

\renewcommand\frm{\mathfrak{m}}

\newcommand\Cl{{\mathrm{Cl}}}
\newcommand\Fil{{\mathrm{Fil}}}
\newcommand\Def{{\mathrm{Def}}}
\newcommand\rank{{\rm rank}}

\newcommand\NS{{\rm NS}}

\newcommand\Hom{{\rm Hom}}

\newcommand\cris{{\rm cris}}

\newcommand\disc{{\rm disc}}

\newcommand\srS{\mathscr{S}}
\newcommand{\spin}{\mathrm{spin}}
\newcommand{\CSpin}{\mathrm{CSpin}}
\newcommand\sfK{\mathsf{K}}
\newcommand{\Kum}{\mathrm{Kum}}
\newcommand\arr{\ifinner\to\else\longrightarrow\fi}

\newcommand\im{\operatorname{im}}

\def\displaytimes_#1{\mathrel{\mathop{\times}\limits_{#1}}}

\def\displayotimes_#1{\mathrel{\mathop{\bigotimes}\limits_{#1}}}

\newcommand\End{\operatorname{End}}

\newcommand\Frac{\operatorname{Frac}}

\newcommand\spec{\operatorname{Spec}}
\newcommand\Spec{\operatorname{Spec}}

\newcommand\Spf{\operatorname{Spf}}
\newcommand\codim{\operatorname{codim}}

\newcommand{\Pic}{\operatorname{Pic}}

\newcommand\rk{\operatorname{rk}}

\newcommand\id{\mathrm{id}}

\newdir{ >}{{}*!/-5pt/@{>}}

\newcommand\doublelong[2]{\mathbin{\xymatrix{{}\ar@<3pt>[r]^{#1}
\ar@<-3pt>[r]_{#2}&}}}

\newlength{\ignora}

\newcommand\Sym{\operatorname{Sym}}

\newcommand\bir{\sim_\mathrm{bir}}

\DeclareMathOperator \ch{ch}

\newcommand{\dR}{\mathrm{dR}}
\newcommand{\td}{\mathrm{td}}

\newcommand{\CH}{{\rm CH}}
\newcommand{\CHM}{{\rm CHM}}
\renewcommand{\ch}{{\rm ch}}
\newcommand{\h}{{\mathfrak h}}

\newcommand{\Gal}{{\rm Gal}}
\newcommand{\et}{{\rm\acute{e}t}}
\newcommand{\Br}{{\rm Br}}

\newcommand\uq{\underline{q}}

\renewcommand{\setminus}{\smallsetminus}
\renewcommand{\tilde}{\widetilde}

\begin{document}

\title[Supersingular Tate conjecture]{Supersingular Tate conjecture for irreducible symplectic varieties of known types: I}

\author{Lie Fu}
\address{Institut de Recherche Mathématique Avancée (IRMA), Universit\'e de Strasbourg, France
$\&$
Institut Universitaire de France (IUF)}
\email{lie.fu@math.unistra.fr}

\author{Xuanlin Huang}
\address{Shanghai Center for Mathematical Science\\
Fudan University\\
2005 Songhu Road\\
20438 Shanghai, China}
\email{xlhuang24@m.fudan.edu.cn}

\author{Zhiyuan Li}
\address{Shanghai Center for Mathematical Science\\
Fudan University\\
2005 Songhu Road\\
20438 Shanghai, China}
\email{zhiyuan\_li@fudan.edu.cn}


\thanks{L.F.\ is supported by the University of Strasbourg Institute for Advanced Study (USIAS), by the Agence Nationale de la Recherche (ANR) under project DAG-Arts (ANR-24-CE40-4098), and by the International Emerging Actions (IEA) project of CNRS, and by Institut Universitaire de France (IUF). Zhiyuan Li is supported by National Science Fund for  General Program (11771086),   Key Program (11731004) and the ``Dawn" Program (17SG01) of Shanghai Education Commission}

\begin{abstract}
We prove that for a supersingular irreducible symplectic variety admitting suitable lifting to characteristic zero has Tate Chow motive if it is of deformation type $\rK3^{[n]}$, OG6 (with Artin invariant $\neq$ 4), or OG10 (with Artin invariant $\neq$ 12), and has supersingular abelian Chow motive if it is of $\Kum^n$-type (with Artin invariant $\neq$ 3). In particular, any product of those irreducible symplectic varieties satisfies the supersingular Tate conjecture for the whole $\ell$-adic and crystalline cohomology rings. 
\end{abstract}

\maketitle

\setcounter{tocdepth}{1}

\tableofcontents

\newpage
\section{Introduction}
\label{sec:introduction}
We study in this paper algebraic cycles of supersingular irreducible symplectic varieties, which are higher-dimensional generalizations of supersingular K3 surfaces. 

\subsection{Supersingular K3 surfaces}
K3 surfaces play a distinguished role in the study of algebraic surfaces. 
In positive characteristic, they share some fundamental similarities to complex K3 surfaces, but also reveal interesting new features that have no analogue in characteristic zero. An important such instance is the presence of \textit{supersingular} K3 surfaces, first constructed by Tate \cite{Tate65}. Let us briefly recall the notion of supersingularity. As it is a geometric property, we fix in this paper an algebraically closed field $k$ of characteristic $p>0$.
Following Artin \cite{Artin_K3}, a K3 surface $S$ defined over $k$ is called \textit{supersingular} if its formal Brauer group, which is a smooth 1-dimensional formal group by \cite{AM77}, is of infinite height:
\[\widehat{\Br}(S)\simeq \widehat{\mathbb{G}}_{\mathrm{a}},\]
or equivalently, the $F$-isocrystal $\rH^2_{\cris}(S/W(k))[1/p]$ is isoclinic of slope 1.
On the other hand, motivated by examples of unirational K3 surfaces, Shioda \cite{Shioda_ss} introduced an a priori different notion of supersingularity by requiring the maximality of the (geometric) Picard rank: $$\rho(S)=22,$$
or equivalently, all classes in the second ($\ell$-adic or crystalline) cohomology of $S$ are algebraic.
It is easy to see that Shioda's supersingularity implies Artin's supersingularity. The converse is precisely the \textit{supersingular Tate conjecture} for K3 surfaces, proposed by Artin \cite{Artin_K3} and established as the following theorem in a series of works by Artin \cite{Artin_K3}, Maulik \cite{Ma14}, Charles ($p\geq 5$) \cite{Ch13, Ch16}, Madapusi Pera ($p\geq 3$) \cite{Pe15}, and Kim--Madapusi Pera ($p=2$) \cite{KM16}. 

\begin{theorem}[Supersingular Tate conjecture]
\label{thm:TateConjectureK3}
 Let $S$ be a K3 surface defined over an algebraically closed field $k$ of characteristic $p>0$. Assume that $S$ is Artin supersingular, then $S$ is Shioda supersingular, that is, the $\ell$-adic first Chern class map 
 \[c_1\colon \NS(S)\otimes \QQ_{\ell}\to \rH^2_{\et}(S, \QQ_{\ell})\]
 is surjective for any $\ell\neq p$, 
 and the crystalline first Chern class map 
 \[
 c_1\colon \NS(S)\otimes W(k)[1/p]\to \rH^2_{\cris}(S/W(k))[1/p]
 \]
 is surjective.
\end{theorem}

Thanks to \Cref{thm:TateConjectureK3}, we will not distinguish the two notions of supersingularity for K3 surfaces in the rest of the paper.

\begin{remark}
    The Tate conjecture is more commonly formulated over a finitely generated base field and claims that the image of the ($\ell$-adic or crystalline) cycle class map is the subspace of ($\ell$-adic or crystalline) Tate classes. 
    It is easy to deduce this more traditional formulation from \Cref{thm:TateConjectureK3}.
    Moreover, combined with Lieblich--Maulik--Snowden \cite{LMS14}, this implies that there are only finitely many isomorphism classes of K3 surfaces over a finite field of characteristic $p>2$. 
\end{remark}

For the purpose of this paper, the following stronger result on motives plays a more fundamental role. It is the supersingular analogue of the celebrated Bloch conjecture (see Remark \ref{rmk:BlochBeilinson} below). For $p\geq 5$, it is a consequence of \Cref{thm:TateConjectureK3} and Fakhruddin's result \cite{Fak02}. See \cite[Section 4.5, Proof of Theorem 1.3]{SSIV} for the general case based on a result of Bragg--Lieblich \cite{BL18}.

\begin{theorem}\label{thm:SSK3TateMotive}
    Let $S$ be a supersingular K3 surface defined over an algebraically closed field $k$. Then the rational Chow motive of $S$ is of Tate type:
    \[\h(S)\simeq \QQ\oplus \QQ(-1)^{\oplus 22}\oplus \QQ(-2) \quad \quad \text{ in } \CHM(k)_{\QQ} .\]
\end{theorem}
Consequently, the Chow motive of a product of supersingular K3 surfaces is again of Tate type, hence the $\ell$-adic and crystalline cycle class maps for such a product variety are isomorphisms.

\subsection{Irreducible symplectic varieties}
In complex geometry, or more generally when $\operatorname{char}(k)=0$,  (smooth) \textit{irreducible symplectic varieties}, also known as \textit{compact hyper-K\"ahler manifolds}, are the natural higher-dimensional analogues of K3 surfaces. By the Beauville--Bogomolov decomposition theorem \cite{Beauville_1983}, irreducible symplectic varieties form one type of building blocks for smooth projective varieties with vanishing first Chern class. So far, all the irreducible symplectic varieties that have been discovered are deformation equivalent to one of the following types: Hilbert schemes of points on K3 surfaces, generalized Kummer varieties, O'Grady's six and ten dimensional examples. We call them the \textit{known} deformation types.

When the base field $k$ is algebraically closed of positive characteristic, an \textit{irreducible symplectic variety} is by definition a smooth projective $k$-variety $X$ with trivial $\pi_1^{\et}(X)$ such that $\rH^0(X, \Omega_{X/k}^2)$ is generated by a nowhere degenerate closed algebraic 2-form; see \Cref{Def: HK variety}. In particular, it is of even dimension and has trivial canonical bundle. 

Our goal is to extend \Cref{thm:TateConjectureK3} and \Cref{thm:SSK3TateMotive} to higher-dimensions. More precisely, an irreducible symplectic $k$-variety $X$ is called \textit{supersingular}\footnote{In Fu--Li \cite{SSIV}, this is called \textit{$2^{nd}$-Artin-supersingular}, to distinguish it from its counterparts for all degrees and from the analogue of Shioda's notion of supersingularity.} if $\widehat{\Br}(X)\simeq \widehat{\mathbb{G}}_{\mathrm{a}}$.
Our previous work~\cite{SSIV} initiated a systematic study of 
supersingular irreducible symplectic varieties, where we compared several notions of supersingularity and proposed the following conjecture:
\begin{conjecture}[Fu--Li]
\label{conj:SSConjectures}
    Let $X$ be an irreducible symplectic variety defined over an algebraically closed field $k$ of characteristic $p>0$. If $X$ is supersingular, then 
    \begin{itemize}
        \item (Supersingular Tate conjecture) The $\ell$-adic and crystalline cycle class maps 
         \[\CH^*(X)\otimes \QQ_{\ell}\to \rH^{2*}_{\et}(X, \QQ_{\ell}),\]
         \[\CH^*(X)\otimes W(k)[1/p]\to \rH^{2*}_{\cris}(X/W(k))[1/p]\]
    are surjective. 
        \item (Supersingular Bloch--Beilinson conjecture) The Chow motive of $X$ belongs to the thick tensor subcategory of $\CHM(k)_\QQ$ generated by supersingular abelian varieties. If moreover all the odd Betti numbers of $X$ vanish, then the Chow motive of $X$ is of Tate type. 
    \end{itemize}
\end{conjecture}
Note that by cohomological realization of motives, the supersingular Bloch--Beilinson conjecture implies the supersingular Tate conjecture.
 
\begin{remark}[Bloch--Beilinson]
\label{rmk:BlochBeilinson}
    Some comments on the general context of the Bloch--Beilinson conjecture are in order. The conjecture originates from Bloch's conjecture for 0-cycles on surfaces with vanishing geometric genus \cite{Bloch-LecturesOnCycles} and from Beilinson's vision of mixed motives \cite{Beilinson-HeightPairing}. The overall philosophy is to control algebraic cycles or motives (with rational equivalence) by cohomology. 
    The Bloch--Beilinson conjecture in the arithmetic setting states that for varieties defined over a number field, the Abel--Jacobi map is injective on the rational Chow groups. Over finite fields (or their algebraic closure), Beilinson conjectured that rational equivalence and  numerical equivalence coincide; in particular, the rational $\ell$-adic cycle class maps are injective. Over complex numbers, as formulated by Jannsen \cite{Jannsen-MotivicSheaves} and Murre \cite{Murre-FiltrationOnChow-I, Murre-FiltrationOnChow-II}, the geometric Bloch--Beilinson conjecture predicts that for smooth projective varieties, there exists a canonical finite exhaustive descending filtration $F^{\bullet}$ on the rational Chow groups such that the graded associated pieces depend only on the homological motive of certain degree. In terms of filtration, the aforementioned conjectures says that $F^2=0$ over number fields and $F^1=0$ over finite fields.
    Now for a smooth projective variety defined over an algebraically closed field of characteristic $p>0$, the condition of being supersingular amounts to saying that the even cohomology groups consist of Tate classes and the odd cohomology groups behave like the first cohomology of a supersingular abelian variety. The conjectural filtration then should satisfy $F^2=0$ (even $F^1=0$ if there is no odd cohomology). This is the motivation for the first and the third author to propose Conjecture \ref{conj:SSConjectures} in \cite{SSIV}.  
\end{remark}

 Various pieces of evidence for Conjecture \ref{conj:SSConjectures} are provided in our work \cite{SSIV} and \cite{SSOG6}, mostly for moduli spaces of sheaves on supersingular K3 surfaces or abelian surfaces. The present paper addresses supersingular
irreducible symplectic varieties of known deformation types without an a priori modular description.

A key technique in our study consists of lifting the symplectic variety to characteristic zero. Inspired by the work of Yang \cite{YangISV}, we introduce the following notion (see Section \ref{subsec:Mazur--Ogus} for the details):

\begin{definition}\label{def: variety of known type}
    Let $X$ be an irreducible symplectic variety defined over an algebraically closed field $k$ of characteristic $p>0$. We say that $X$ is \emph{excellent}, if there is a smooth projective lifting $\cX$ of $X$ to a finite extension DVR of $W(k)$, such that the geometric generic fiber $\cX_{\bar{\eta}}$ is an irreducible symplectic variety with Hodge numbers equal to those of $X$, and such that $\cX$ carries a lifting of a primitive polarization on $X$. \\
   In this situation, we say that $X$ is \textit{of the deformation type of} the geometric generic fiber $\mathcal{X}_{\bar{\eta}}$.
\end{definition}

\begin{remark}\label{remark: remark on IS known type
}
    Let $X$ be a smooth projective variety defined over an algebraically closed field $k$ of characteristic $p>0$ satisfying the condition of being excellent in Definition~\ref{def: variety of known type}. 
    If $p>\dim(X)$ and $X$ satisfies a coprime-to-$p$ \cref{condition: discriminant prime to p}, then $X$ is automatically irreducible symplectic; see \Cref{theorem: Mazur--Ogus implies IS}. Also note that the condition of primitive polarization is necessary, see 
    \Cref{remark: primitivity is necessary}.
\end{remark}

\begin{remark}
    Given an excellent irreducible symplectic variety $X$, its deformation type is not a priori uniquely determined by $X$. However, when $p>\dim(X)$, it is intrinsically determined by $X$ for the known deformation types ($\mathrm{K}3^{[n]}$-type, $\mathrm{Kum}^n$-type, $\mathrm{OG}6$-type and $\mathrm{OG}10$-type); see \Cref{prop: deformation type independent on lifting}.
\end{remark}

Our main results are the following theorems establishing Conjecture \ref{conj:SSConjectures} for excellent irreducible symplectic varieties of all \textit{known} deformation types, except in a few extremal cases of maximal Artin invariants:

\begin{theorem}
\label{thm:intro:SSMotive}
    Let $X$ be an excellent irreducible symplectic variety defined over an algebraically closed field $k$ of characteristic $p>\dim(X)$. Assume that $X$ is supersingular and we are in one of the following cases:
    \begin{enumeratei}
        \item $X$ is of $\mathrm{K}3^{[n]}$-type;
        \item $X$ is of $\mathrm{Kum}^n$-type with Artin invariant $\sigma_0\neq 3$;
         \item $X$ is of $\mathrm{OG}10$-type with Artin invariant $\sigma_0\neq 12$;
        \item $X$ is of $\mathrm{OG}6$-type with Artin invariant $\sigma_0\neq 4$;
    \end{enumeratei}
    Then the Chow motive of $X$ is of Tate type in Cases $(i), (iii), (iv)$, and is of supersingular abelian type in Case (ii). 
\end{theorem}

Applying the realization functors, we deduce the following consequence of \Cref{thm:intro:SSMotive}:
\begin{theorem}[Supersingular Tate conjecture]
Let $Y$ be a product of supersingular irreducible symplectic varieties as in \Cref{thm:intro:SSMotive}, then all the $\ell$-adic and crystalline cycle class maps 
         \[\CH^*(Y)\otimes \QQ_{\ell}\to \rH^{2*}_{\et}(Y, \QQ_{\ell}),\]
         \[\CH^*(Y)\otimes W(k)[1/p]\to \rH^{2*}_{\cris}(Y/W(k))[1/p]\]
    are surjective. 
\end{theorem}

Our proof is based on choosing (twisted) modular lifting of these irreducible symplectic varieties. This lifting result is worthy to be summarized as the following theorem:
\begin{theorem}
     Let $X$ be an excellent irreducible symplectic variety defined over an algebraically closed field $k$ of characteristic $p>\dim(X)$. Assume that $X$ is supersingular and we are in one of the following cases:
    \begin{enumeratei}
        \item $X$ is of $\mathrm{K}3^{[n]}$-type or $\mathrm{Kum}^n$-type;
         \item $X$ is of $\mathrm{OG}10$-type with Artin invariant $\sigma_0\neq 12$;
    \end{enumeratei}
    Then $X$ admits a (twisted) modular lifting.  
\end{theorem}

\begin{remark}
\label{rmk:NoModularLiftingOG10}
    For an OG10-type variety $X$ with Artin invariant $\sigma_0=12$, it cannot admit a (twisted) modular lifting. Indeed, suppose it admits a (twisted) modular lifting. The exceptional class on the generic fiber gives a class of norm $-6$ and divisibility $3$. Taking specialization, this contradicts to \Cref{prop: OG10 big Artin invariant}. 
\end{remark}

\begin{remark}
   The three excluded maximal Artin invariant cases in \Cref{thm:intro:SSMotive} are due to limitation of our techniques: 
   For the $\Kum^n$-type with $\sigma_0=3$, we can only have a twisted modular lifting. But we cannot control the motive in the twisted case.
   For the OG10-type with $\sigma_0=12$, we cannot have a (twisted) modular lifting, as explained in Remark \ref{rmk:NoModularLiftingOG10}. For the OG6-type with $\sigma_0=4$, we cannot choose a lifting that admits a Mongardi--Rapagnetta--Sacc\`a double cover. We will treat the remaining cases in a subsequent paper.
\end{remark}

\subsection{Idea of the proof}
    Given an irreducible symplectic variety $X$, we construct an $\ell$-adic and crystalline Beauville--Bogomolov form on its second cohomology, and show that the latter is equipped with the structure of a $K3$-crystal. After studying the deformation property, we prove the supersingular Tate conjecture for divisor classes in $X$, using a period map to an integral Spin Shimura variety.  Then we classify the supersingular N\'eron--Severi lattices and identify a small-rank saturated sublattice that can be lifted with a polarization. By choosing a lifting together with suitable line bundles, we can use a numerical birational criterion to realize the characteristic-zero fiber as a moduli space of (twisted) sheaves on a $K3$ or abelian surface. We show that these $K3$ or abelian surfaces have supersingular good reductions. In characteristic zero, the Chow motive of the generic fiber is controlled by the motive of the underlying surface, and we obtain the main theorem by specialization to characteristic $p$. This gives the conclusion for $K3^{[n]}$-type, $\Kum^n$-type, and OG10-type. For the OG6-type, we replace the direct modular realization by the Mongardi--Rapagnetta--Sacc\`a double cover and control the Chow motive by the motive of the associated K3 surface and abelian fourfold.

\subsection{Organization of the paper}

\begin{itemize}
    \item Section~\ref{sec:IS-varieties} introduces irreducible symplectic varieties in
positive characteristic, the known deformation types, and Artin
supersingularity.
\item  In Section~\ref{sec:BB-form} we construct the integral
Beauville--Bogomolov forms and the associated K3 crystals.  
\item  Section~\ref{sec:deformation}
studies lattice-polarized deformations, proves the supersingular Tate
conjecture for divisors, and establishes the lifting theorem. 
\item Section~\ref{sec:NS-classification}
classifies the N\'eron--Severi lattices in terms of the Artin invariant.
\item Section~\ref{sec:bridgeland} reviews Bridgeland moduli spaces, numerical
birational criteria, and motivic results.
\item   In Section~\ref{sec:three-types}
we prove the main theorem for the $K3^{[n]}$-type, the $\Kum^n$-type, and
the $\mathrm{OG}10$-type.  Section~\ref{sec:og6} treats the $\mathrm{OG}6$-type.
\item The appendix studies the torsion of cohomology groups of irreducible symplectic varieties of known deformation types. 
\end{itemize}

\subsection{History and AI disclosure}
This project began around early 2024. The first case completed was the \(K3^{[n]}\)-type case, which originated as the undergraduate thesis problem of the second author. In early 2025, we completed the proof of the entire paper, mainly treating the cases where the irreducible symplectic variety of known type is modular or twisted modular. After that, our main focus shifted to the non-modular cases, which will appear in the next paper. The third author presented this work at various workshops and conferences in 2024 and 2025 such as the \textit{Sino-French conference on Algebraic Geometry} in AMSS (Beijing) and the \textit{International conference on Geometry, Algebra and Mathematical Physics} in New Uzbekistan University, etc.

\noindent
\begin{minipage}[t]{0.13\textwidth}
  \vspace{0pt}
  \centering
  \includegraphics[width=\linewidth]{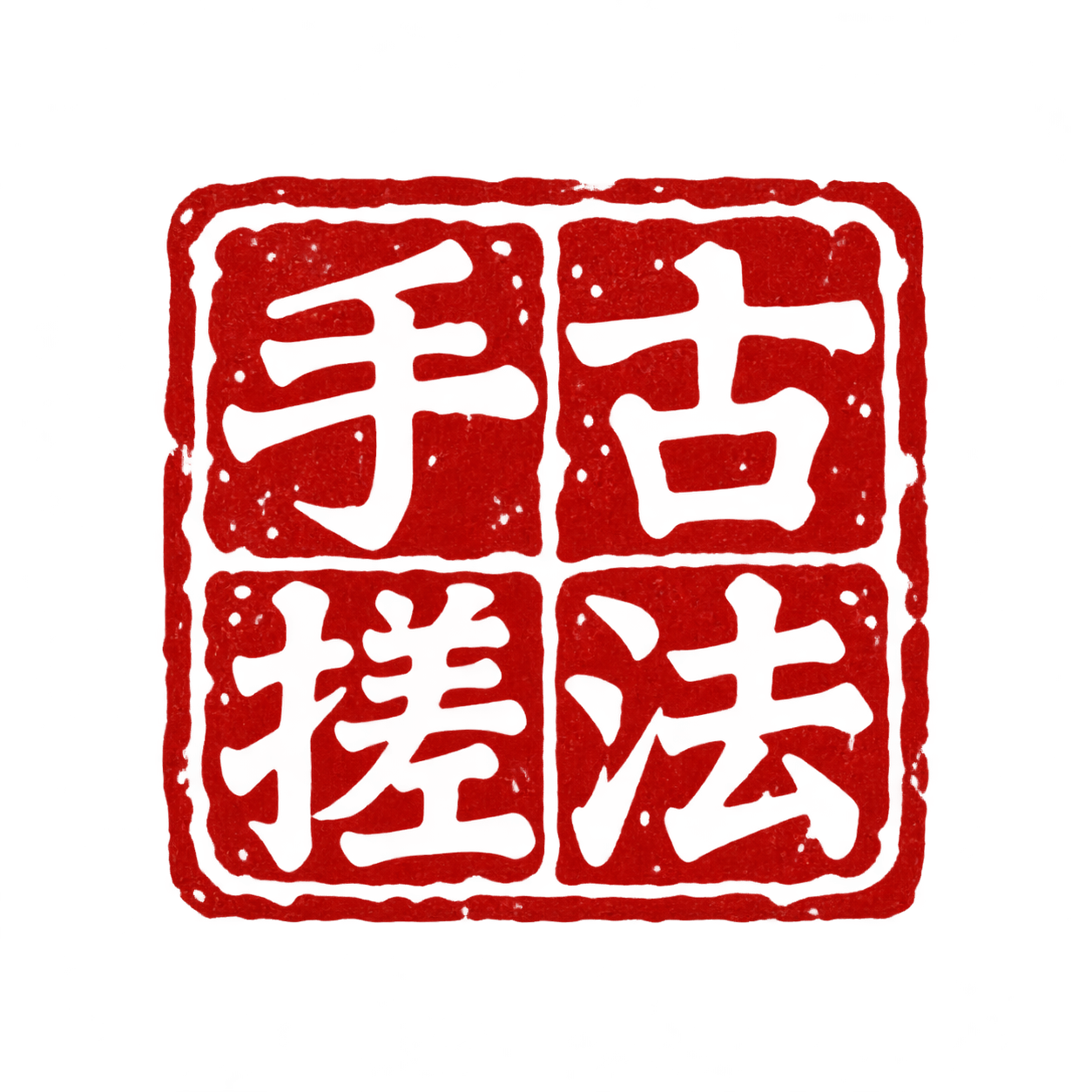}
\end{minipage}\hspace{1em}%
\begin{minipage}[t]{\dimexpr\textwidth-0.15\textwidth-1em\relax}
  \vspace{0pt}
  All mathematical proofs and ideas in this paper were developed without the use of LLM. LLM was used only to check typos, correct the grammatical errors, locate references of some widely known results, and generate the image on the left. The authors assume full responsibility for all mathematical content.
\end{minipage}
\par\medskip

\subsection*{Acknowledgment:}
The authors want to thank Ziquan Yang, Ruxuan Zhang, and Haitao Zou for very helpful discussions.

\subsection*{Notation and conventions}
\begin{itemize}
    \item For a perfect field $k$ of characteristic $p$, we write $W$ for the Witt ring of $k$ and $K$ for the fraction field of $W$, and $\sigma$ for the lifting of the Frobenius endomorphism $x\to x^p$ on $k$.
    \item We write $\widehat{\ZZ}$ for the profinite completion of $\ZZ$ and $\widehat{\ZZ}^p$ for its prime-to-$p$ component. We write $\AA_f$ for $\widehat{\ZZ}\otimes\QQ$ and $\AA_f^p$ for its prime-to-$p$ component.
    \item Let $R$ be an integral domain. By a lattice $V$ over $R$, we mean a finite free $R$-module $V$ together with a quadratic form $q$ on $V$. We also denote by $q(\cdot,\cdot)$ the associated symmetric bilinear form. For a vector $v\in V$, we write $v^2$ for $q(v)$.
\end{itemize}

\section{Irreducible symplectic  varieties}\label{sec:IS-varieties}

In this section, we collect some basic notions about irreducible symplectic varieties. We fix a field $k$ throughout this section.

\subsection{Irreducible symplectic varieties}
In complex geometry, an \emph{irreducible symplectic manifold} is a simply connected compact Kähler manifold \(X\) such that \(\rH^0(X,\Omega_X^2)\) is generated by a nowhere degenerate holomorphic \(2\)-form, which is automatically closed. By Beauville--Bogomolov theorem \cite{Beauville_1983}, abelian varieties, strict Calabi--Yau varieties and irreducible symplectic varieties form the building blocks of compact Kähler manifolds with vanishing first Chern class. 

In positive characteristics, there is currently no universally accepted definition of irreducible symplectic varieties. We use the following definition in this paper.

\begin{definition}\label{Def: HK variety}
A smooth projective variety  $X$ over  $k$ is called \textit{irreducible symplectic} if  
\begin{enumeratei}
    \item  $\rH^0(X,\Omega^2_{X/k})$ is spanned by a nowhere degenerate closed $2$-form;
    \item $\pi_1^{\et}(X_{\bar{k}})$ is trivial.
\end{enumeratei}
\end{definition}
 We refer to \cite[Section 3]{SSIV} for examples of irreducible symplectic varieties.

\subsection{Mazur--Ogus lifting}
\label{subsec:Mazur--Ogus}
Let $k$ be an algebraically closed field of characteristic $p>0$. Following \cite{Joshi} and \cite{Antieau-Bragg}, we say that a smooth projective $k$-variety $X$ is \textit{Mazur--Ogus} if the Hodge-to-de Rham spectral sequence degenerates at $E_1$ and all crystalline cohomology groups are torsion free. 

We introduce the following notion of liftings:

\begin{definition}[Mazur--Ogus lifting]
\label{def:MazurOgus-lifting}
    Let $X$ be a smooth projective $k$-variety. Let $V$ be a finite extension DVR of $W(k)$. A \textit{Mazur--Ogus} lifting of $X$ is a smooth projective family $\pi\colon \cX\to \Spec V$ with special fiber isomorphic to $X$, such that $\mathbf{R}^j\pi_*\Omega^i_{\cX/V}$ is a free $V$-module for all $i, j\in\mathbb{N}$.
\end{definition}

\begin{remark}
    The above condition on the freeness of the higher direct images is equivalent to the equality of Hodge numbers between the generic fiber and the special fiber, and it is in turn also equivalent to the condition that $X$ is Mazur--Ogus. In fact, if $X$ admits a Mazur--Ogus lifting, then equality must hold everywhere in the following chain of inequalities:
    \[
    \dim \rH_\dR^k(X)\leq \sum_{i+j=k}h^{i,j}(X)=\sum_{i+j=k}h^{i,j}(\cX_{\bar{\eta}})=b_k(\cX_{\bar{\eta}})=\rk \rH^k_{\cris}(X/W)\leq \dim \rH_\dR^k(X).
    \]
    Then the Hodge-to-de Rham spectral sequence degenerates at $E_1$ page and all crystalline cohomology is torsion free. Conversely, if the latter condition is satisfied, then the upper-semicontinuity inequality in the following 
    \[
    \dim \rH_\dR^k(X)=\sum_{i+j=k}h^{i,j}(X)\geq\sum_{i+j=k}h^{i,j}(\cX_{\bar{\eta}})=b_k(\cX_{\bar{\eta}})=\rk \rH^k_{\cris}(X/W)= \dim \rH_\dR^k(X)
    \]
    is an equality, hence this lifting is a Mazur--Ogus lifting. Since being Mazur--Ogus is a condition on $X$ (not on the lifting), we see that if $X$ admits a Mazur--Ogus lifting, then all liftings are Mazur--Ogus.
\end{remark}

Having a Mazur--Ogus lifting to an irreducible symplectic variety implies some good properties. For instance, we will use the following result:
\begin{proposition}[{\cite[Proposition~2.1.4]{YangISV}}]\label{prop: integral first chern class map}
Let $k$ be an algebraically closed field of characteristic $p>0$ and let $X$ be a smooth projective variety over $k$ that admits a Mazur--Ogus lifting to an irreducible symplectic variety. Then we have:
\begin{enumerate}
    \item[(a)] $c_{1, \ell}:\mathrm{NS}(X)\otimes \mathbb{Z}_\ell \to \rH^2_{\mathrm{\acute et}}(X,\mathbb{Z}_\ell(1))$ is injective with torsion-free cokernel for every prime number $\ell\neq p$.
    \item[(b)] $c_{1, \cris}:\mathrm{NS}(X)\otimes \mathbb{Z}_p \to \rH^2_{\cris}(X/W)^{F=p}$ is injective with torsion-free cokernel.
    \item[(c)] $\mathrm{NS}(X)=\mathrm{Pic}(X)$ and $\mathrm{NS}(X)$ is torsion free.
\end{enumerate}
\end{proposition}

In the definition of Mazur--Ogus lifting, if the lifting is over $W$ itself instead of a finite extension DVR, then this lifting will automatically carry a lifting of a primitive polarization by the following lemma. This will be useful later when we provide examples of irreducible symplectic varieties of known types.

\begin{lemma}\label{lem: ExsitenceOfPrimitivePolarization}
    Let $k$ be an algebraically closed field of characteristic $p>0$ and $W=W(k)$. Assume that there is a smooth projective family $\pi:\cY\to\Spec W$ such that
    \begin{enumeratei}
        \item The Hodge-to-de Rham spectral sequence degenerates at $E_1$-page, and
        \item $\mathbf{R}^j\pi_*(\Omega_{\cY/W}^i)$ is a free $W$-module for all $i, j\in \mathbb{N}$.
    \end{enumeratei}
    Then there is a primitive polarization on the special fiber $\cY_0$ that can be extended to $\cY$.
\end{lemma}
\begin{proof}
    Let $\cL$ be a relatively ample line bundle on $\cY/W$. Its restriction $\cL_0$ on $\cY_0$ is an ample line bundle. $\cL_0$ is a multiple of some primitive ample line bundle $H$ on $\cY_0$. By \cite[Proposition 1.12]{Ogus1979}, $H$ can be extended to $\cY$.
\end{proof}

We introduce the concept of \emph{excellent irreducible symplectic varieties}, inspired by Yang \cite{YangISV}.
\begin{definition}[Excellence]\label{def: ExcellentReduction}
    Let $X$ be an irreducible symplectic variety over $k$. We say that $X$ is an excellent irreducible symplectic variety if $X$ admits a Mazur--Ogus lifting $\cX\to V$ with $V$ is a finite extension DVR of $W(k)$, such that the geometric generic fiber \(\mathcal{X}_{\bar{\eta}}\) is an irreducible symplectic variety and $\cX$ admits a polarization whose restriction to the special fiber is primitive. 
\end{definition}

\begin{remark}\label{remark: remark on IS}
    In fact, under some mild conditions, we do not need to assume that $X$ is an irreducible symplectic variety in \Cref{def: ExcellentReduction} at first. We will see in \Cref{theorem: Mazur--Ogus implies IS} that any smooth projective $k$-variety $X$ with $\dim X<p$ satisfying the lifting condition in \Cref{def: ExcellentReduction} and \Cref{condition: discriminant prime to p} is automatically an irreducible symplectic variety.
\end{remark}

\subsection{Deformation types} 
\subsubsection{Characteristic zero}
We first assume that the base field $k$ is of characteristic zero. By the Lefschetz principle, it is harmless to assume that $k$ can be embedded into $\CC$. Fix a complex irreducible symplectic manifold $M$. We say that $X$ is \textit{of deformation type} $M$ if $X_\CC:=X\times_{k, \sigma}\CC$ is deformation equivalent to $M$ for some embedding $\sigma\colon k\hookrightarrow\CC$. Clearly, $X$ can only have finitely many deformation types. It is an interesting open question whether the deformation type is actually independent of the embedding $k\hookrightarrow\CC$.  Note that all known deformation types, given as follows, are independent of the embedding $k\hookrightarrow\CC$ by \cite[Proposition 2.3]{fu2026unpolarizedshafarevichconjectureshyperkahler}.

Over the complex numbers, the following families are the only known deformation types of irreducible symplectic manifolds \cite{Beauville_1983,OGrady_K3,OGrady_Abelian}:
\begin{enumerate}
\item \emph{\({\rK3^{[n]}}\)-type}: deformation equivalent to the Hilbert scheme \(S^{[n]}\) of length-\(n\) subschemes on a \(K3\) surface \(S\).
\item \emph{\({\rm Kum^{n}}\)-type}: deformation equivalent to the generalized Kummer variety \(K_n(A)\) associated with an abelian surface \(A\).
\item \emph{\(\mathrm{OG}6\)-type}: deformation equivalent to O’Grady’s six-dimensional example, obtained as a crepant resolution of the Albanese fiber of a moduli space of sheaves on an abelian surface.
\item \emph{\(\mathrm{OG}10\)-type}: deformation equivalent to O’Grady’s ten-dimensional example, obtained as a crepant resolution of a moduli space of sheaves on a K3 surface.
\end{enumerate}

The next proposition shows that the deformation type of an irreducible symplectic variety is invariant under deformation in characteristic zero, which is a generalization of \cite[Proposition 2.3.2]{YangISV}.

\begin{proposition}\label{prop: defornation type in characteristic zero}
    Let $k$ be a field of characteristic zero and let $S$ be a connected variety over $k$. Suppose that there is a smooth projective family $\cX$ over $S$. If some geometric fiber is of deformation type $M$, then every geometric fiber is of deformation type $M$.
\end{proposition}
\begin{proof}
    We first assume that $k$ is algebraically closed. Let $s$ be any $k$-point of $S$ and let $\eta$ be the generic point of $S$. It suffices to show that $\cX_s$ is of deformation type $M$ if and only if $\cX_{\bar{\eta}}$ is of deformation type $M$. There is a natural embedding of residue fields $k\hookrightarrow\kappa(\bar{\eta})$ and we can reduce to the case when both of them have finite transcendental degree over $\bar{\QQ}$. If $\cX_s$ is of deformation type $M$ under some embedding $\iota:k\hookrightarrow\CC$, we can extend $\iota$ to $\bar{\iota}:\kappa(\bar{\eta})\hookrightarrow \CC$. Since $S\otimes_\iota\CC$ is still connected, we have $\cX_{\bar{\eta}}$ is also of deformation type $M$ under the embedding $\bar{\iota}$. Similarly, if $\cX_{\bar{\eta}}$ is of deformation type $M$ under some embedding $\kappa(\bar{\eta})\hookrightarrow\CC$, $\cX_{s}$ is also of deformation type $M$ under the composition of embeddings $k\hookrightarrow\kappa({\bar{\eta}})\hookrightarrow\CC$.

    If $k$ is not algebraically closed, we use the fact that $\Gal(\bar{k}/k)$ acts transitively on the connected components of $S\otimes_k\bar{k}$ to reduce to the algebraically closed case.
\end{proof}

\subsubsection{Positive characteristic}
In positive characteristic, the discussion of deformation type of irreducible symplectic variety is more subtle. For technical reasons, we can only define the deformation type for excellent irreducible symplectic varieties.
\begin{definition}\label{def: HK of known types}
Let $M$ be a deformation type of irreducible symplectic varieties in characteristic zero. 
Let $X$ be an excellent irreducible symplectic variety over an algebraically closed field $k$ of characteristic $p>0$. We say that $X$ is \emph{of deformation type} $M$ if the Mazur--Ogus lifting in \Cref{def: ExcellentReduction} has geometric generic fiber of deformation type $M$.

When $M$ is one of the known types, we say that \(X\) is \emph{of known type}. In this case, $X$ has a unique deformation type if $p>\dim X$ thanks to \Cref{prop: deformation type independent on lifting} that we will establish later (all the hypotheses are satisfied by the known deformation types). 
\end{definition}



We show in the following examples that excellent irreducible symplectic varieties of each known deformation type exist. Results of torsion-freeness established in \Cref{appendix} are used. 
\begin{example}[$\rK3^{[n]}$-type, {\cite[Proposition 5.1.3]{YangISV}}]\label{example:K3^n}
    Suppose that $p>2n+1$ and $S$ is a K3 surface over $k$. Let $v$ be a primitive Mukai vector with $v^2+2=2n$, and let $H$ be a $v$-generic polarization. Suppose that we are in the following cases:  $S$ is non-supersingular; or $c_1(v)$ is a multiple of $H$ and $c_{1,\dR}(H)\notin \Fil^2\rH^2_\dR(S/k)$. Then the moduli space of stable sheaves $\cM_H(S,v)$ is an excellent irreducible symplectic variety of $\rK3^{[n]}$-type. In particular, the Hilbert scheme $S^{[n]}$ is always an excellent irreducible symplectic variety of $\rK3^{[n]}$-type.
\end{example}

\begin{example}[$\mathrm{Kum}^n$-type]\label{example: Kummer}
    Suppose that $p>2n+1$ and $A$ is an abelian surface over $k$. Let $K_n(A)$ be the smooth symplectic variety in \cite[Section 6]{SSIV}. By \cite{Oo87}, $A$ can be lifted to $\cA/W$, then $K_n(A)$ can be lifted to $K_n(\cA)$. By \Cref{thm:torsion free}, the cohomology $\rH^i_\et(K_n(\cA_{\bar{\eta}}),\ZZ_p)$ is torsion free for all $i$. Thus by \cite{FM87}, $\rH^j_\cris(K_n(A)/W)$ is torsion free for all $j\leq p-1$, and hence for all $j\geq 0$ by the Poincar\'e duality. Together with \Cref{lem: ExsitenceOfPrimitivePolarization}, we have $K_n(A)$ is excellent irreducible symplectic variety of $\Kum^n$-type.
\end{example}

\begin{example}[$\mathrm{OG}6$-type]
    Let $k$ be an algebraically closed field of characteristic $p>7$ and  $A$ be an ordinary abelian surface over $k$. Let $v=(2,0,-2)$ be a Mukai vector and $H$ be a $v$-generic polarization on $A$. Consider the singular moduli space $\cM_H(A,v)$ of $H$-semistable sheaves on $A$ with Mukai vector $v$, and its crepant resolution $\widetilde{\cM}_H(A,v)$. Denote by $\widetilde{K}_H(A,v)$ the Albanese fiber. By \cite{SSOG6}, $\widetilde{K}_H(A,v)$ can be lifted to $W$, and the generic fiber is an OG6-type variety in characteristic zero. By \Cref{thm: TorsionOfOG6}, $\widetilde{K}_H(A,v)$ is of OG6-type by the same argument as in \Cref{example:K3^n,example: Kummer}.
\end{example}

\begin{example}[$\mathrm{OG}10$-type]
     Let $k$ be an algebraically closed field of characteristic $p>11$ and  $S$ ibe an ordinary K3 surface over $k$. Let $v=(2,0,-2)$ be a Mukai vector and $H$ be a $v$-generic polarization on $S$. Consider the crepant resolution $\widetilde{\cM}_H(S,v)$ of the moduli space of $H$-semistable sheaves with Mukai vector $v$ on $S$, which can be lifted to $W$. By \Cref{thm;pTorsionFreenessOfOG10}, the same argument as above shows that $\widetilde{\cM}_H(S,v)$ is of OG10-type.
\end{example}

\subsection{Supersingular irreducible symplectic varieties}\label{subsec:Br}
We recall here the definition of supersingular irreducible symplectic varieties defined in \cite{SSIV}. 
Artin and Mazur \cite{AM77} showed that for a smooth projective variety \(X\) with \(\rH^1(X,\mathcal{O}_X)=0\) the functor sending an Artin local $k$-algebra $R$ to 
\[
\widehat{\operatorname{Br}}_X(R)=\ker\bigl(\rH^2_{\et}(X\times_k R,\mathbb{G}_{\mathrm{m}})\to \rH^2_{\et}(X,\mathbb{G}_{\mathrm{m}})\bigr)
\]
is pro‑representable by a formal group, called the \emph{formal Brauer group} of \(X\). Its Lie algebra is canonically isomorphic to \(\rH^2(X,\mathcal{O}_X)\). Recently, Grammatica \cite{Gr25} showed that the formal Brauer group is formally smooth when $\rH^3_{\cris}(X/W)$ is torsion free.

\begin{definition}\label{def:ArtinSS}
An irreducible symplectic variety \(X\) is called \emph{Artin‑supersingular} if the formal Brauer group \(\widehat{\operatorname{Br}}_X\) is isomorphic to the formal additive group \(\widehat{\mathbb{G}}_{\mathrm{a}}\).
\end{definition}

We assume that $X$ is an excellent irreducible symplectic variety. The formal Brauer group \(\widehat{\operatorname{Br}}_X\)   is \(1\)-dimensional because \(\rH^2(X,\mathcal{O}_X)\simeq k\). In this situation Artin‑supersingularity is equivalent to the statement that the Newton polygon of the \(F\)-crystal \((\rH^2_{\text{cris}}(X/W(k)),\varphi)\) is \emph{supersingular}. This equivalence is a variant of a result of Artin and Mazur and will be used later.

\begin{proposition}[{\cite[Proposition 3.6]{SSIV}}]\label{prop:ArtinEquiv}
Let \(X\) be an excellent irreducible symplectic variety. Then \(\widehat{\operatorname{Br}}_X\) is formally smooth, and \(X\) is Artin‑supersingular if and only if the crystalline cohomology \(\rH^2_{\cris}(X/W(k))\) is a supersingular \(F\)-crystal (i.e. its Newton polygon is isoclinic of slope 1).
\end{proposition}

\section{Beauville–Bogomolov forms on irreducible symplectic varieties}\label{sec:BB-form}

\subsection{Beauville–Bogomolov form in characteristic zero}

For a complex irreducible symplectic manifold $X$ of dimension $2n$, the second cohomology group $\rH^2(X,\mathbb{Z})$ carries a primitive integral quadratic form $q_X$, called the \emph{Beauville–Bogomolov (BB) form}, which satisfies the following properties (see \cite{ Beauville_1983}, \cite[Section 23]{CompactHK_Huy}):

\begin{enumerate}
    \item $q_X$ is non‑degenerate and has signature $(3, b_2(X)-3)$.
    \item There exists a positive rational number $c_X$, called the \emph{Fujiki constant}, such that for every $\alpha_1,\cdots\alpha_{2n}\in H^2(X,\mathbb{Z})$
    \[
    \prod_{i=1}^{2n}\alpha_i= \frac{c_X}{n!\,2^n}\sum_{\sigma\in \mathfrak{S}_{2n}} q(\alpha_{\sigma(1)},\alpha_{\sigma(2)})\cdots q(\alpha_{\sigma(2n-1)},\alpha_{\sigma(2n)}).
    \]
    \item The Hodge decomposition $\rH^2(X,\mathbb{C}) = \rH^{2,0}(X) \oplus \rH^{1,1}(X) \oplus \rH^{0,2}(X)$ satisfies
    \[
    (\rH^{2,0}(X)\oplus \rH^{0,2}(X))\perp   \rH^{1,1}(X) 
    \]
    with respect to $q_X\otimes\mathbb{C}$.
\end{enumerate}
The Beauville--Bogomolov form and the Fujiki constant are deformation invariants and their computations for the known deformation types are summarized below in Table~\ref{tab:BBform} (see for example \cite{BB_form}):

\begin{table}[htb]\label{table: BB form}
    \centering
    \begin{tabular}{|c|c|c|c|c|}
    \hline
    Type & $\dim X$ & $b_2(X)$ & $c_X$ & ($\rH^2(X,\ZZ)$, $q_X$)\\
    \hline
    $\rK3^{[n]}$ & $2n$ & $23$ &1 & $\rU^{\oplus3}\oplus \rE_8(-1)^{\oplus2}\oplus\langle -2(n-1)\rangle$ \\
    \hline
    $\operatorname{Kum}^n$ & $2n$ & $7$ & $n$+1& $\rU^{\oplus3}\oplus\langle -2(n+1)\rangle$ \\
    \hline
    OG6 & $6$ & $8$ &4& $\rU^{\oplus3}\oplus\langle -2\rangle^{\oplus2}$ \\
    \hline
    OG10 & $10$ & $24$ &1& $\rU^{\oplus3}\oplus \rE_8(-1)^{\oplus2}\oplus \rA_2(-1)$ \\
    \hline
    \end{tabular}
    \vspace{.1cm}

    \caption{BB forms and Fujiki constants of the known types}
    \label{tab:BBform}
\end{table}

Now let $X$ be an irreducible symplectic variety defined over a field $k$ of characteristic zero. Fix an embedding $\iota:k\hookrightarrow\mathbb{C}$. Artin’s comparison isomorphism
$$\mathrm{H}^{2}\bigl(X\times_{k,\iota}\operatorname{Spec}\mathbb{C},\mathbb{Z}(1)\bigr)\otimes\widehat{\mathbb{Z}}\;\cong\; \mathrm{H}^{2}_{\mathrm{\acute{e}t}}(X_{\bar{k}},\widehat{\mathbb{Z}}(1))$$
allows us to transport the Beauville--Bogomolov form $q_X$ on the left‑hand side to a $\widehat{\mathbb{Z}}$-valued quadratic form
$$q: \rH^2_{\mathrm{\acute{e}t}}(X_{\bar{k}},\widehat{\mathbb{Z}}(1))\longrightarrow\widehat{\mathbb{Z}},$$
where $\bar{k}$ denotes the algebraic closure of $k$ inside $\mathbb{C}$. As shown in \cite[Lemma 4.2.1]{BindtThesis}, this $\widehat{\mathbb{Z}}$-form is independent of the embedding; it is called the \emph{adelic Beauville--Bogomolov form} of $X$. The Fujiki constant $c_X$ is defined as the Fujiki constant of $X\times_{k,\iota}\mathbb{C}$, which also does not depend on $\iota$. For convenience, we introduce the following notation:
\begin{notation}
      For an irreducible symplectic variety $X$ in characteristic zero, we write $\disc(X)$ for the discriminant of the integral Beauville--Bogomolov form on $\rH^2(X\times_{k,\iota}\CC,\ZZ)$ for one/any $\iota\colon k\to \CC$. Note that $\disc(X)$ is independent of the choice of $\iota$, since it coincides with the discriminant of the adelic Beauville--Bogomolov form.
\end{notation}

\subsection{Beauville--Bogomolov form in positive characteristic}
Following Yang \cite{YangISV}, we give a construction of Beauville–Bogomolov forms on excellent irreducible symplectic varieties of deformation type $M$ that has both Fujiki constant and the discriminant coprime to $p$. Slightly more generally, let $X$ be a smooth projective $k$-variety satisfying the following:

\begin{condition}\label{condition: discriminant prime to p}
    There is a Mazur--Ogus lifting $\cX/V$ of $X$ over a finite extension DVR $V$ of $W(k)$, such that the geometric generic fiber $\cX_{\bar{\eta}}$ is an irreducible symplectic variety, whose discriminant $\disc(\cX_{\bar{\eta}})$ and Fujiki constant $c_{\cX_{\bar{\eta}}}$ both have trivial $p$-adic valuations.
\end{condition}

Under this condition, we will have a good notion of Beauville--Bogomolov form of $X$:
\begin{itemize}
    \item For each prime $\ell\neq p$, the adelic Beauville--Bogomolov form on $\mathcal{X}_\eta$ gives, via smooth and proper base change, a non‑degenerate symmetric bilinear form
    \[
    q_\ell: \rH^2_{\mathrm{\acute{e}t}}(X_{\bar{k}},\mathbb{Z}_{\ell}(1))\times \rH^2_{\mathrm{\acute{e}t}}(X_{\bar{k}},\mathbb{Z}_{\ell}(1))\longrightarrow \mathbb{Z}_{\ell}.
    \]
    \item There also exists a bilinear form on the crystalline cohomology
    \[
    q_{\mathrm{cris}}: \rH^2_{\mathrm{cris}}(X/W(k))\times \rH^2_{\mathrm{cris}}(X/W(k))\longrightarrow W(k)
    \]
    that gives $\rH^2_{\mathrm{cris}}(X/W(k))$ the structure of a $K3$ crystal (\Cref{Def: K3 crystal}). The construction relies on the integral $p$-adic Hodge theory, and this is where the \Cref{condition: discriminant prime to p} is used. We will study this crystalline form in detail later.
\end{itemize}

The following elementary lemma, which characterizes the Fujiki relation, will be essential when we transfer the Beauville--Bogomolov form to different cohomology theories.

\begin{lemma}
\label{lem:fujiki}
    Let $R$ be an integral domain with fraction field $F$ such that $2n!$ is non-zero in $R$. Let $V$ be a free $R$-module of finite rank. Suppose $w: V^{\otimes 2n}\to R$ and $q: V^{\otimes 2}\to R$ are symmetric multi-linear maps satisfying
    \begin{equation}\label{eq:fujiki}
        w(\alpha_1,\dots,\alpha_{2n}) = \frac{c}{n!\,2^n}\sum_{\sigma\in \mathfrak{S}_{2n}} q(\alpha_{\sigma(1)},\alpha_{\sigma(2)})\cdots q(\alpha_{\sigma(2n-1)},\alpha_{\sigma(2n)}),
    \end{equation}
    where $c$ is a nonzero element in $F$. Then for any $\xi\in V$ such that $q(\xi,\xi)\neq 0$, one has $w(\xi,\dots,\xi,\alpha)=0$ if and only if $q(\xi,\alpha)=0$. Moreover, $q$ is completely determined by $w$ and the value $q(\xi,\xi)$.
\end{lemma}
\begin{proof}
    This lemma can be proven after base change to $F$. So after replacing $w$ by $c^{-1}w$, the lemma follows from \cite[Lemma 2.1.1]{YangISV}
\end{proof}

We can now define the Beauville--Bogomolov form for excellent irreducible symplectic varieties over fields of positive characteristic, under the non-degeneracy \Cref{condition: discriminant prime to p}. 

\begin{proposition}[cf.~\cite{YangISV}]
\label{prop: BB form}
    Let $k$ be an algebraically closed field of characteristic $p$ and  $X$ be a smooth projective variety over $k$ satisfying \Cref{condition: discriminant prime to p} and $\dim X<p$. \\
    For every prime number $\ell\neq p$, there is a unique pairing
    $$q_{\ell}:\rH^2_{\et}(X,\ZZ_\ell(1))^{\otimes2}\to\ZZ_\ell$$ 
    such that $q_{\ell}$ satisfies the relation in \Cref{lem:fujiki} with $w$ induced by the cup product.\\
    There exists a perfect pairing 
    $$q_{\cris}:\rH^2_{\cris}(X/W)^{\otimes2}\to W(-2)$$
    such that $q_{\cris}$ satisfies the relation in \Cref{lem:fujiki} with $w$ induced by the cup product.\\
    Furthermore, for $q=q_{\ell}$ or $q_{\cris}$, we have $q(c_1(\zeta),c_1(\zeta))\in\ZZ$ for any $\zeta\in\NS(X)$, and it is positive whenever $\zeta$ is ample.
\end{proposition}

\begin{proof}
    The proof is exactly the same as that of \cite[Proposition 2.1.5]{YangISV}. For the $\ell$-adic cohomology, we simply apply the smooth proper base change theorem to recover $\rH^2_{\et}(X,\ZZ_\ell)$ from $\rH^2_{\et}(\mathcal{X}_{\bar{\eta}},\ZZ_\ell)$. For the crystalline cohomology, we use the integral $p$-adic Hodge theory to recover $\rH^2_{\cris}(X/W)$ from $\rH^2_{\et}(\mathcal{X}_{\bar{\eta}},\ZZ_p)$. 
\end{proof}

By the crystalline-de Rham comparison isomorphism and the universal coefficient theorem, we have $\rH^i_{\dR}(X/k)=\rH^i_{\cris}(X/W)\otimes_W k$. So the quadratic form $q_\cris$ gives a perfect quadratic form $q_{\dR}$ on $\rH^2_{\dR}(X/k)$. It is easy to see that $q_{\dR}$ also satisfies the Fujiki relation in \Cref{lem:fujiki} and $\Fil^1\rH^2_{\dR}(X/k)=(\Fil^2\rH^2_{\dR}(X/k))^\perp$ under $q_{\dR}$.

We can now prove the assertion mentioned in \Cref{remark: remark on IS}.
\begin{theorem}\label{theorem: Mazur--Ogus implies IS}
 Let $k$ be an algebraically closed field of characteristic $p$ and  $X$ be a smooth projective variety over $k$ satisfying \Cref{condition: discriminant prime to p} and $\dim X<p$. Then $X$ is an irreducible symplectic variety in the sense of \Cref{Def: HK variety}.
\end{theorem}

\begin{proof}
    We first show that the canonical bundle of $X$ is trivial. By definition, we find a finite flat extension $V$ of $W$ and a smooth projective family $\cX\to\Spec V$ whose geometric generic fiber $\cX_{\bar{\eta}}$ is an irreducible symplectic variety. Then the canonical bundle of $\cX_\eta$ is trivial. Choose a non-zero section $\alpha$ of $K_{\cX_\eta}$, which can be seen as a rational section of $K_{\cX/V}$. It is straightforward to see that the divisor $\mathrm{div}(\alpha)$ supports in the special fiber $X$, then $\mathrm{div}(\alpha)$ is a multiple of $[X]$. But $X$ is defined by the uniformizer of $V$, then $\mathrm{div}(\alpha)$ is principal. This means that $K_{\cX/V}$ is trivial, hence so is $K_X$. We now verify the two conditions in \Cref{Def: HK variety}.
    
    For (i), by the assumption on the Hodge numbers, $\rH^0(X,\Omega_X^2)$ is generated by an element $\sigma$ and $\rH^2(X, \mathcal{O}_X)$ is generated by an element $\rho$. We only need to show that $\sigma$ is nowhere degenerate. Let $\tau\in \rH^2_{\rmd\rR}(X/k)$ be a lifting of $\rho$ via the canonical surjective morphism $\rH^2_{\dR}(X/k)\to \rH^2(X, \mathcal{O}_X)$. Since $\mathrm{Fil}^1=(\mathrm{Fil}^2)^\perp$, we have $q_{\rmd\rR}(\sigma,\tau)\neq 0$. Therefore, $\sigma^n\tau^n=cn!\cdot q_{\rmd\rR}(\sigma,\tau)^n\neq0$. Then $\sigma^n$ is a non-zero element in $\rH^0(X,\Omega_X^{2n})$. Since the canonical bundle $K_X$ of $X$ is trivial, any non-zero section of $K_X$ is nowhere vanishing. Hence $\sigma^n$ is nowhere vanishing, i.e.~$\sigma$ is nowhere degenerate.

    For (ii), we use the surjectivity of the specialization homomorphism for \'etale fundamental groups, see \cite{SGA1}. 
\end{proof}


\begin{corollary}[cf. \cite{YangISV}]
\label{corollary: perfect pairing on tangent-cotangent}
 Let $k$ be an algebraically closed field of characteristic $p$ and  $X$ be a smooth projective variety over $k$ satisfying \Cref{condition: discriminant prime to p} and $\dim X<p$. Then the natural cup product pairing
\[
\rH^1(X,T_X)\times \rH^1(X,\Omega_X^1)\to \rH^2(X,\mathcal{O}_X)\simeq k
\]
is perfect, and the map
\[
\rH^2(X,\mathcal{O}_X)\to \rH^{2n}(X,\mathcal{O}_X)
\]
defined by taking cup product with $\rho^{n-1}$ is an isomorphism.
\end{corollary}
\begin{proof}
    The proof is the same as \cite[Corollary 2.1.7]{YangISV}.
\end{proof}

\subsection{The structure of K3 crystal} In the study of supersingular K3 surfaces, the second cohomology $\rH^2_{\cris}(X/W)$ equipped with the Poincar\'e pairing plays an important role (see \cite{Artin_K3} and \cite{Shafarevich} for example). We will see that an analogous structure exists for irreducible symplectic varieties, namely a K3 crystal in the sense of Ogus \cite{Ogus1979}. The K3 crystal structure is crucial in the study of supersingular irreducible symplectic varieties.

\begin{definition}[Ogus]\label{Def: K3 crystal}
    A \textit{K3 crystal} of rank $n$ is a free $W$-module $H$ of rank $n$ together with a $\sigma$-linear injective map $\varphi:H\to H$, and a symmetric bilinear form 
    \[
    \langle-,-\rangle: H\otimes_W H\to W
    \]
    such that
    \begin{enumerate1}
        \item $p^2H\subset \im(\varphi)$,
        \item $\varphi\otimes_Wk$ has rank 1,
        \item $\langle-,-\rangle$ is perfect,
        \item $\langle\varphi(x),\varphi(y)\rangle=p^2\sigma(\langle x,y\rangle)$.
    \end{enumerate1}
\end{definition}

\begin{proposition}
    Let $k$ be an algebraically closed field of characteristic $p$ and  $X$ be a smooth projective variety over $k$ satisfying \Cref{condition: discriminant prime to p} and $\dim X<p$. Then $(\rH^2_\cris(X/W),q_{\cris}, \varphi)$ is a K3 crystal. Here $\varphi$ is the $\sigma$-linear map induced by the absolute Frobenius morphism.
\end{proposition}
\begin{proof}
    We verify the four conditions in \Cref{Def: K3 crystal}. Condition (3) is verified in \Cref{prop: BB form}. Condition (4) follows from the compatibility between the Beauville--Bogomolov form and the cup product.

    For Conditions (1) and (2), denote $\rH^2_\cris(X/W)$ by $H$ and $\rank(H)=n$ for convenience. By Mazur--Nygaard--Ogus \cite[Theorem 3.8(2)]{Lectures_SSK3}, we have
    \[
    H/\varphi(H)=(W/pW)^{\oplus n-2}\oplus (W/p^2W).
    \]
    Then (1) and (2) follow directly.
\end{proof}

\section{Deformation of lattice-polarized irreducible symplectic varieties}\label{sec:deformation}

\subsection{Local deformation theory}
Our main goal in this section is to study the deformation property of irreducible symplectic variety of certain good deformation type. The whole section follows and adapts Yang's work \cite[Section 2]{YangISV}, with the modifications described below. We work in a slightly more general setting. Langer and Zink \cite[Section 3 and 4]{LZ19} gave a notion of \textit{K3-type variety} and studied its deformation property. However, the condition that $h^{p,q}=0$ when $p+q=3$ is not always true for irreducible symplectic varieties (for example, it fails for the generalized Kummer deformation type). Instead, using the deformation theory of Calabi--Yau varieties developed in \cite{Unobstructness}, we are led to consider the following natural condition:
\begin{condition}\label{condition: perfect IS type}
    Let $X$ be a smooth projective variety of dimension $2n$ over an algebraically closed field $k$, either of characteristic zero or of characteristic $p>2n$. We assume that
    \begin{enumeratea}
        \item $h^{1,0}=h^{0,1}=0,h^{2,0}=h^{0,2}=1, h^{1,1}\geq 1$ and $\sum_{p+q=r}h^{p,q}=b_r$ for any $0\leq r\leq 4n$. Here, $b_r=\dim \rH^r_\et(X,\QQ_\ell)$ is the $r$-th Betti number,
        \item $\rH^0(X,\Omega_X^2)$ is generated by a nowhere degenerate closed $2$-form $\sigma$,
        \item let $\rho$ be a generator of $\rH^2(X,\cO_X)$. We require that the pairing 
        \[
        \rH^1(X,\Omega_X^1)\times \rH^1(X,\Omega_X^1)\to k,\quad\quad (\omega_1,\omega_2)\to \int_X\omega_1\omega_2(\sigma\rho)^{n-1}
        \]
        is a perfect pairing.
    \end{enumeratea}
\end{condition}

\begin{remark}
    By \Cref{corollary: perfect pairing on tangent-cotangent}, any smooth projective variety $X$ with $\dim X<p$ satisfying \Cref{condition: discriminant prime to p} automatically satisfies \Cref{condition: perfect IS type}.
\end{remark}

The hypotheses in \Cref{condition: perfect IS type} ensure the following two important geometric properties. Both of them are well-known in characteristic zero, so we focus on the case of positive characteristic.

\begin{proposition}\label{prop: smoothness of the deformation functor}
    Let $X$ be a smooth projective variety over an algebraically closed field $k$ of characteristic $p$ satisfying \Cref{condition: perfect IS type}. The mixed characteristic deformation functor $\Def(X)$ is formally smooth.
\end{proposition}
\begin{proof}
    The condition (b) implies that the section $\sigma^n$ trivializes the canonical bundle. By the condition (a), every inequality in $\sum_{p+q=r}h^{p,q}\geq h_\dR^r(X/k)\geq b_r$ is in fact an equality. Thus, the Hodge-to-de Rham spectral sequence of $X$ degenerates at the $E_1$-page and all crystalline cohomology groups $\rH_\cris^r(X/W)$ are torsion free. Thus, by \cite[Theorem A]{Unobstructness}, the mixed characteristic deformation functor $\Def(X)$ is formally smooth.
\end{proof}

Let $X$ be a smooth projective variety over an algebraically closed field $k$ of characteristic $p$ satisfying \Cref{condition: perfect IS type}. Thanks to \Cref{prop: smoothness of the deformation functor} and the fact that $\rH^0(X,T_X)=0$, the functor $\mathrm{Def}(X)$ is pro-representable by a formal scheme $\DD := \Spf \mathscr{R}$ for $\mathscr{R}= W[[x_1, \cdots , x_{h^{1,1}} ]]$. We denote the universal family over $\DD $ by $\pi:\uX\to\DD$ and its special fiber $\uX\otimes_Wk\to\DD\otimes_Wk$ by $\pi_0:\uX_0\to\DD_0$. The same proof as that of \cite[Proposition 1.5]{Deligne1981} shows that for any line bundle $\xi$ on $X$, $\mathrm{Def}(X,\xi)$ is pro-representable by a formal subscheme of $\DD_\xi:= \Spf \mathscr{R}/(f_\xi)$, where
$f_\xi\in \mathscr{R}$ is an element depending on $\xi$.

\begin{proposition}\label{prop: relative de Rham cohomology}
    Let $X$ be a smooth projective variety over an algebraically closed field $k$ of characteristic $p$ satisfying \Cref{condition: perfect IS type}. Then
    \begin{enumeratea}
        \item The coherent sheaf $\rH^{j}(\uX,\Omega^i_{\uX/\DD})$ and $\rH_{\rmd\rR}^r(\uX/\DD)$ are free for any $i,j,r$. Moreover, for any $\mathscr{R}$-algebra $R$, the natural morphisms $\rH_{\rmd\rR}^r(\uX/\DD)\otimes R\to\rH_{\rmd\rR}^r(\uX_R/R)$ and $\rH^{j}(\uX,\Omega^i_{\uX/\DD})\otimes R\to\rH^{j}(\uX_R,\Omega^i_{\uX_R/R})$ are isomorphisms.
        \item Let $F_{\uX}$ be the Frobenius structure, then there are Mazur--Ogus inequalities for $\rH_{\rmd\rR}^2(\uX/\DD)$ :
        \begin{align*}
            \mathrm{Fil}^1\rH_{\rmd\rR}^2(\uX/\DD)&\subset\{x\in\rH_{\rmd\rR}^2(\uX/\DD):F_{\uX}(x)\in p\rH_{\rmd\rR}^2(\uX/\DD)\}\\
            \mathrm{Fil}^2\rH_{\rmd\rR}^2(\uX_0/\DD_0)&=\mathrm{Im}(\{x\in\rH_{\rmd\rR}^2(\uX/\DD):F_{\uX}(x)\in p^2\rH_{\rmd\rR}^2(\uX/\DD)\}\to\rH_{\rmd\rR}^2(\uX_0/\DD_0)).
        \end{align*}
        \item $\rH^r_{\cris}(\uX_0)$ is a crystal of vector bundle over $\DD_0$.
    \end{enumeratea}
\end{proposition}
\begin{proof}
    We first mention that once we have proved the freeness in (a), the other statements will follow in the same way as \cite[Proposition 2.2.2(b)(c)]{YangISV}. What we have now is the freeness of $\rH^r_\cris(X/W)$, but this cannot give the freeness of $\rH_{\rmd\rR}^r(\uX/\DD)$ directly since the crystalline-de Rham comparison theorem only works for PD-thickenings. Our strategy is to choose suitable PD-thickenings, and then the freeness of $\rH_{\rmd\rR}^r(\uX/\DD)$ can be detected over these PD-thickenings. This is inspired by the proof of \cite[lemma 7.11]{Unobstructness}.

    Consider the derived pushforward $\rR\pi_*(\Omega_{\uX/\DD}^\bullet)$. Since $\mathscr{R}$ is a Noetherian complete local ring, $\rR\pi_*(\Omega_{\uX/\DD}^\bullet)$ is quasi-isomorphic to a bounded complex $(\mathscr{F}^\bullet,d^i:\mathscr{F}^i\to \mathscr{F}^{i+1})$ that consists of finite free $\mathscr{R}$-modules. A little linear algebra shows that we can assume $d^i(\mathscr{F}^i)\subset \frm_{\mathscr{R}}\mathscr{F}^{i+1}$, i.e. the differentials in $\mathscr{F}^\bullet\otimes k$ are all zero. We have
     \[
    \mathscr{F}^\bullet\otimes k=\rR\pi_*(\Omega_{\uX/\DD}^\bullet)\otimes^{\mathbf{L}} k=\rR\pi_*(\Omega_{X/k}^\bullet).
    \]
    Thus, we have $\mathscr{F}^r\otimes k=\rH^r_\dR(X/k)$ is of rank $b_r(X)$. Then $\mathscr{F}^r$ is also free of rank $b_r(X)$. 
    
    We claim that the differentials in $\mathscr{F}^\bullet$ are all zero. To see this, suppose that $d^r\neq 0$, then there is a non-zero $\alpha\in\frm_{\mathscr{R}}$ which is one entry of the matrix representation of $d^r$. Write $\alpha$ in the form $\alpha=up^eg+h$, where $u$ is a unit, $g$ is a homogeneous polynomial of degree $s$ which is non-zero modulo $p$, and $h$ consists of terms of higher degree.

     Let $W_m$ be the quotient $W/p^m$ and $W_{m,d}$ be the truncated divided power polynomial ring
    \[
    W_{m,d}=W_{m}\langle t\rangle/( \gamma_i(t) :i\geq d).
    \]
    Note that $W_{m}$ can be viewed as $W_{m,1}$.  Choose $m\gg 0$ such that $p^es!$ is non-zero in $W_m$. Take $a_1,\cdots,a_n\in k$ such that $g(a_1,\cdots,a_{h^{1,1}})\neq 0\ \mathrm{mod}\ p$. Define a map
    \begin{align*}
         \varphi: \mathscr{R}=W[[x_1,\cdots,x_{h^{1,1}}]]&\to W_{1,2}=k[t]/(t^2)\\
         x_i &\mapsto a_i t.
    \end{align*}
    Since $\mathscr{R}$ is formally smooth, the map $\varphi$ can be lifted to a map $\widetilde{\varphi}:\mathscr{R}\to W_{m,s+1}$. Then 
    \[
   \widetilde{ \varphi}(\alpha)=up^es! g(\widetilde{a}_1,\cdots, \widetilde{a}_{h^{1,1}})\gamma_s(t)\neq 0.
    \]
    Here $\widetilde{a}_i$ is a lifting of $a_i$ in $W_m$. Consider
    \[
    \mathscr{F}^\bullet\otimes_{\widetilde{\varphi}} W_{m,s+1}=\rR\pi_*(\Omega_{\uX/\DD}^\bullet)\otimes_{\widetilde{\varphi}}^{\mathbf{L}} W_{m,s+1}=\rR\pi_*(\Omega_{X_{m,s+1}/W_{m,s+1}}^\bullet).
    \]
    Here, $X_{m,s+1}$ is the pullback of $\uX$ to $W_{m,s+1}$. By the crystalline-de Rham comparison theorem, $\rH_\dR^r(X_{m,s+1}/W_{m,s+1})$ is free of rank $b_r(X)$, then the differentials in $\mathscr{F}^\bullet\otimes W_{m,s+1}$ are all zero. This contradicts $\widetilde{\varphi}(\alpha)\neq 0$. Thus, all the differentials in $\mathscr{F}^\bullet$ are zero, then $\rH_{\rmd\rR}^r(\uX/\DD)$ is free of rank $b_r(X)$. The similar argument shows that $\rH^{j}(\uX,\Omega^i_{\uX/\DD})$ is free of rank $h^{i,j}(X)$.
\end{proof}

\begin{remark}\label{remark: compatible discussion}
    The only condition in the definition of a K3-type variety that may fail for a smooth projective variety satisfying \Cref{condition: perfect IS type} is the vanishing of $h^{p,q}=0$ when $p+q=3$. In \cite{LZ19}, this condition gives the smoothness of the deformation functor and the freeness of the relative de Rham cohomology. However, in our case, we can still have these two properties by the above two propositions. Therefore, many propositions and proofs in \cite[Section 3 and 4]{LZ19} apply to our setting with only minor modifications, as we will see later.
\end{remark}

The next proposition shows that there is a unique quadratic form on $\rH^2_{\dR}(\uX/\DD)$, which is compatible with the Hodge filtration and the crystal structure on $\rH^2_{\cris}(\uX_0)$.

\begin{proposition}
    There is a perfect horizontal quadratic form $\uq_{\rmd\rR}$ on $\rH^2_{{\rmd\rR}}(\uX/\DD)$ such that $$\mathrm{Fil}^1\rH^2_{{\rmd\rR}}(\uX/\DD)=(\mathrm{Fil}^2\rH^2_{{\rmd\rR}}(\uX/\DD))^\perp$$ under $\uq_\dR$, unique up to multiplication by a unit. Moreover, if $X$ satisfies the condition in \Cref{prop: BB form}, then
    \begin{enumeratea}
        \item the evaluation of $\uq_\dR$ on any $W$ point on $\DD$ agrees with the quadratic form $q_X$ constructed in \Cref{prop: BB form} via the crystalline-de Rham comparison isomorphism, after multiplication by a suitable unit;
        \item $\uq_\dR$ satisfies the Fujiki relation in \Cref{lem:fujiki} with $w$ being the cup product;
        \item If we let $\Phi:F^*_{\DD_0/W}\rH^2_{\cris}(\uX_0)\to \rH^2_{\cris}(\uX_0)$ be the morphism given by the $F$-crystal structure and $\uq_\cris$ be the quadratic form on $\rH^2_{\cris}(\uX_0)$ induced by $\uq_\dR$, then $(\rH^2_{\cris}(\uX_0),\Phi,\uq_{\cris},\mathrm{Fil}^2\rH^2_{{\rmd\rR}}(\uX/D))$ defines a K3 crystal in the sense of Nygaard-Ogus \cite[Definition 5.1]{NO85}.
    \end{enumeratea}
\end{proposition}
\begin{proof}
    The proof of this proposition is the same as in \cite[Proposition 2.2.3]{YangISV}, so we just explain the main idea here. The quadratic form $\uq_\dR$ is given in \cite[Definition 23]{LZ19} and its property is given by \cite[Lemma 25]{LZ19}. The remaining can be checked on a $W$-point, since this form is horizontal.
\end{proof}

The next two propositions, adapted from \cite{YangISV}, give a complete description of the deformation functor $\Def(X,\xi)$. Although the conditions are slightly different, the proof is exactly the same.

\begin{proposition}[{\cite[Proposition 2.2.4]{YangISV};\cite[Theorem 31]{LZ19}}]\label{prop: isotropic lines}
    Let $R'$ be an Artinian $W$-algebra and $R \in \mathrm{Cris}(R'/W)$. Let $\tilde{X}'$ be a deformation of $X$ over $R'$. Let $\cI(\tilde{X}', R)$ be the set of isotropic direct summands of $H^2_{\mathrm{cris}}(\tilde{X}')_R$ which lift $\mathrm{Fil}^2 H^2_{\mathrm{dR}}(\tilde{X}'/R')$.
    \begin{enumeratea}
        \item The map $\Psi$ from deformations of $\tilde{X}'$ over $R$ to $\cI(\tilde{X}', R)$ defined by sending $\tilde{X}$ to $\mathrm{Fil}^2 H^2_{\mathrm{dR}}(\tilde{X}/R) \subset H^2_{\mathrm{dR}}(\tilde{X}/R) 
        \cong H^2_{\mathrm{cris}}(\tilde{X}')_R$ is an isomorphism.
        \item Let $\xi'$ be a line bundle on $\tilde{X}'$. Then $\xi'$ extends to $\tilde{X}$ if and only if $\Psi(\tilde{X})$ is orthogonal to $c_{1,\mathrm{cris}}(\xi')_R$.
\end{enumeratea}
\end{proposition}

The deformation property of $\Def(X,\xi)$ can be summarized as follows.
\begin{proposition}[{\cite[Proposition 2.2.7]{YangISV}}]\label{prop: local deformation}
    Let $X$ be a smooth projective variety over an algebraically closed field $k$ of characteristic $p$ satisfying \Cref{condition: perfect IS type}.
    \begin{enumeratea}
        \item If $\rH^2_{\cris}(X/W)$ is not supersingular and $\xi_1\cdots\xi_m$ span a direct summand of $\Pic(X)$, the functor $\Def(X,\xi_1,\cdots,\xi_m)$ is smooth over $W$.
        \item For any primitive line bundle $\xi\in\Pic(X)$, if $c_1(\xi)\notin \mathrm{Fil}^2H^2_{\rmd\rR}(X/k)$, then $\Def(X,\xi)$ is smooth over $W$.
        \item If $c_1(\xi)\in \mathrm{Fil}^2H^2_{\rmd\rR}(X/k)$, then $q_X(\xi,\xi)$ has $p$-adic valuation 1. Moreover, $\Def(X,\xi)$ is isomorphic to
        \[
        \Spf W[[x_1,\cdots,x_{h^{1,1}}]]/(x_1^2+\cdots +x_{h^{1,1}}^2+p).
        \]
        \item In any case, $\Def(X,\xi)$ is quasi-healthy regular.
    \end{enumeratea}
\end{proposition}

In order to compare different choices of liftings, we need to put them into a common space.  This inspires the following definition.
\begin{definition}\label{def: Hilbert subscheme}
    Let \(P\) be a polynomial in \(\mathbb{Q}[T]\), let $N := P(1 )-1$,
    and let \(\mathrm{Hilb}_{P}\) be the Hilbert scheme over
    \(\mathrm{Spec}\,\mathbb{Z}_{(p)}\) which parametrizes closed subschemes of
    \(\mathbb{P}^{N}\) with Hilbert polynomial \(P\). Let \(\mathcal{Z}\) be the
    universal family over \(\mathrm{Hilb}_{P}\). Let \(\mathrm{Hilb}^{+}_{P}\) be the possibly empty maximal locally closed subscheme of \(\mathrm{Hilb}_{P}\) such that, for every geometric point $s \in \mathrm{Hilb}^{+}_{P}$, the fiber \(\mathcal{Z}_{s}\) is a smooth projective variety of dimension $2n$ satisfying \Cref{condition: perfect IS type}, $ h^{1,1}\geq 2$, and $\rH^{i}\bigl(\mathcal{Z}_{s},\mathcal{O}_{\mathcal{Z}_{s}}(1)\bigr)=0
    \ \text{for all } i>0$. For every positive integer \(m\), let \(\mathrm{Hilb}^{+,m}_{P}\) denote the possibly empty open subscheme of \(\mathrm{Hilb}^{+}_{P}\) such that, for every geometric point $s \to \mathrm{Hilb}^{+,m}_{P}$, the line bundle \(\mathcal{O}_{\mathcal{Z}_{s}}(1)\) is the \(m\)-th power of a primitive polarization on \(\mathcal{Z}_{s}\).
\end{definition}

The most important geometric property of $\mathrm{Hilb}_P^{+,m}$ we will need is the following:
\begin{lemma}\label{lem: connectedness of gneric fiber}
    Let $k$ be an algebraically closed field with characteristic $p>2n$. Assume that $p\nmid m$. The generic fiber of every connected component of both $\mathrm{Hilb}^{+,m}_{P}$ and $\mathrm{Hilb}^{+,m}_{P,W}$ is also connected. Moreover, $\mathrm{Hilb}^{+,m}_{P,W}$ is smooth over $\Def(X,m\xi)$.
\end{lemma}
\begin{proof}
    This is essentially the proof of 
    \cite[Lemma 2.3.4]{YangISV}. We explain the main idea. By \cite[\href{https://stacks.math.columbia.edu/tag/054F}{055J}]{stacks-project}, it suffices to prove that $\mathrm{Hilb}^{+,m}_{P,W}$ is flat over $W$ with reduced special fiber. For a $k$-point $s$ on $\mathrm{Hilb}^{+,m}_{P,W}$, let $X=\mathcal{Z}_s$ and $\xi$ be the primitive polarization on $X$ with \(\mathcal{O}_{\mathcal{Z}_{s}}(1)=m\xi\). By the universal property, we have a natural morphism from the formal neighborhood of $s$ to $\Def(X,\xi)$, which is smooth by the formal smoothness criterion. Then the lemma follows from \Cref{prop: local deformation}.
\end{proof}
With the help of $\mathrm{Hilb}^{+,m}_{P}$, we can show that the deformation type is stable under specialization and generalization in characteristic $p$.

\begin{proposition}\label{prop: spreading out in positive characteristic}
Let $S$ be a connected quasi-compact scheme over $\mathbb{F}_p$ and
$(f : \mathcal{X} \to S,\boldsymbol{\xi})$ be a primitively polarized smooth proper scheme. Assume that every geometric fiber of $\mathcal{X}$ satisfies \Cref{condition: perfect IS type} and $h^{1,1}\geq 2$. 
For a geometric point $s$ on $S$, we say that $s$ satisfies property
$(\ast_{\mathrm{weak}})$ (resp.\ $(\ast_{\mathrm{strong}})$) provided that for
some (resp.\ any) mixed characteristic discrete valuation ring $R$ with residue
field $k(s)$, the generic fiber of some (resp.\ any) deformation of
$\mathcal{X}_s$ over $R$ to which $\xi_s$ extends is of $M$-type.

Then every geometric point $s$ satisfies $(\ast_{\mathrm{strong}})$ provided some
$s$ satisfies $(\ast_{\mathrm{weak}})$. In particular, if one geometric fiber is
of $M$-type, so is any other geometric fiber.
\end{proposition}
\begin{proof}
    We sketch the proof, which is the same as {\cite[Lemma 2.3.5]{YangISV}}. We may assume that $S$ is affine and connected. Then choose $m\gg 0. p\nmid m$ such that $m\boldsymbol{\xi}$ is very ample and $\rH^i(\cX_s,m\boldsymbol{\xi}_s)=0$ for all $i>0$ and all geometric points $s$ of $S$. Let $P$ be the Hilbert polynomial of $m\boldsymbol{\xi}_s$. We obtain a morphism $S\to \mathrm{Hilb}^{+,m}_{P,W}$, whose image lands in some connected component $T$ of $\mathrm{Hilb}^{+,m}_{P,W}$. We can assume that  $S=T\otimes\FF_p$. Then any deformation of $\cX_s$ to which $\boldsymbol{\xi}_s$ extends is given by some $R$-valued point of $T$ lifting $s$. The conclusion follows from \Cref{lem: connectedness of gneric fiber}.
\end{proof}

\subsection{Tate conjecture for divisors}

Let $X$ be an excellent irreducible symplectic variety over an algebraically closed field $k$ of characteristic $p>\dim X$ of deformation type $M$ such that $h^{1,1}(X)\geq2$ and the discriminant $\disc(M)$ and the Fujiki constant $c_M$ both have trivial $p$-adic valuation. Suppose that $X$ is supersingular. We prove the Tate conjecture for divisors for $X$ in this section. Our method is a direct generalization of Madapusi Pera \cite{Pe15} and Yang \cite{YangISV}. The proof is based on Spin Shimura varieties and period morphisms, 

We briefly recall Clifford algebras and spin groups, referring to \cite[Section 1]{PK2016} for details. Let $R$ be an integral domain with $2$ invertible and $R_\QQ$ is a field. Let $L$ be an $R$-lattice of rank $m$ with a nondegenerate quadratic form. We denote by $H=\Cl(L)$ the Clifford algebra, viewed as a $\Cl(L)$-bimodule. $H$ has a natural $\ZZ/2\ZZ$-grading given by $\Cl(L)=\Cl^+(L)\oplus\Cl^-(L)$. The natural map
\[
L\hookrightarrow\Cl(L)\hookrightarrow\End(H)
\]
embeds $L$ as a direct summand of $\End(H)$. We can equip $\End(H_\QQ)$ with a symmetric pairing given by $(\alpha,\beta)\mapsto2^{-m}\mathrm{tr}(\alpha\circ\beta)$, so that $L_\QQ$ embeds into $\End(H_\QQ)$ isometrically. Let $\pi\in H^{\otimes(2,2)}$ be the orthogonal projection from $\End(H_\QQ)$ onto $L_\QQ$. The image of $\End(H)\subset \End(H_\QQ)$ under $\pi$ is exactly the dual $L^\vee\subset L_\QQ$.

Define the group $\CSpin(L)$ by
\[
\CSpin(L)=\{v\in\Cl^+(L)^\times:vLv^{-1}\subset L\},
\]
which has a structure of $R$-group scheme. The group $\CSpin(L)$ is equipped with two natural representations: a spin representation $\mathrm{sp}:\CSpin(L)\to\mathrm{GL}(H)$ given by the left multiplication and an adjoint representation $\mathrm{ad}:\CSpin(L)\to\mathrm{SO}(L)$ given by conjugation, fitting in the exact sequence
\[
1\to\GG_m\to\CSpin(L)\to\mathrm{SO}(L)\to 1.
\]

Now, let $L$ be an even quadratic lattice of signature $(2, m-2)$ over $\ZZ$, and let $\Omega$ be the period domain
\[
\{\omega\in\PP(L\otimes\CC):\langle\omega,\bar{\omega}\rangle>0,\langle\omega,\omega\rangle=0\},
\]
which parametrizes the Hodge structures of K3 type on $L$. Let $G$ denote the algebraic group $\CSpin(L_\QQ)$ and $G^{\mathrm{ad}}$ denote the group $\rS\rO(L_\QQ)$. The pair $(G,\Omega)$ (resp.\ $(G^{\mathrm{ad}},\Omega)$) defines a Shimura datum of Hodge type (resp.\ abelian type) with reflex field $\mathbb{Q}$. We can equip $H$ with a non-degenerate symplectic form $\psi$ such that the spin representation $\mathrm{sp}:G\to \mathrm{GL}(H_{\mathbb{Q}})$ factors through $\mathrm{GSp}:=\mathrm{GSp}(H_{\mathbb{Q}},\psi)$. Let ${S}^{\pm}$ be the associated Siegel half spaces. There is an inclusion of Shimura data $i:(G,\Omega)\to (\mathrm{GSp},{S}^{\pm})$. 

Choosing a small enough level structure $\mathsf{K}$, we write $\mathrm{Sh}_\sfK(L)$ for the Shimura variety $\mathrm{Sh}_\sfK(G,\Omega)$. The inclusion $i$ endows $\mathrm{Sh}_\sfK(L)$ with a family of abelian schemes, which we denote by $a:\mathscr A\to \mathrm{Sh}_\sfK(L)$. Let $\sfK^{\mathrm{ad}}$ denote the image of $\sfK$ in $G^{\mathrm{ad}}(\mathbb{A}_f)$. We denote the Shimura variety $\mathrm{Sh}_{\sfK^{\mathrm{ad}}}(G^{\mathrm{ad}},\Omega)$ by $\mathrm{Sh}^{\mathrm{ad}}_\sfK(L)$.

The representation $\mathrm{sp}:\mathrm{CSpin}(L)\to \mathrm{GL}(H)$ (resp.\ $\mathrm{ad}:\mathrm{CSpin}(L)\to \mathrm{SO}(L)$) endows $\mathrm{Sh}_\sfK(L)_{\mathbb{C}}$ with a $\widehat{\mathbb{Z}}$-local system $\mathbf{H}_B$ (resp.\ $\mathbf{L}_B$). $\mathbf{H}_B$ can be identified with the Betti cohomology $R^1a_{\mathbb{C},*}\widehat{\mathbb{Z}}$. Let $\mathbf{H}_\ell:=R^1a_{\acute{e}t,*}\mathbb{Z}_\ell$ and let $\mathbf{H}_{\mathrm{dR}}$ be the first relative de Rham cohomology of $\mathscr{A}$. Over $\mathrm{Sh}_\sfK(L)_{\mathbb{C}}$, we have canonical isomorphisms $\mathbf{H}_B\otimes \mathbb{Z}_\ell \simeq \mathbf{H}_\ell$ and $\mathbf{H}_{\mathrm{dR}} \simeq \mathbf{H}_B\otimes \mathcal{O}$, where $\mathcal{O}$ denotes the structure sheaf. The tensors stabilized by $\mathrm{CSpin}(L)$ naturally spread out as global sections of these sheaves. More precisely, $\mathscr{A}$ is equipped with a $\mathbb{Z}/2\mathbb{Z}$-grading, a left $\mathrm{Cl}(L)$-action, and global sections $\boldsymbol{\pi}_B$, $\boldsymbol{\pi}_\ell$ and $\boldsymbol{\pi}_{\mathrm{dR}}$ of $(\mathbf{H}_B\otimes \mathbb{Q})^{\otimes(2,2)}$, $(\mathbf{H}_\ell\otimes \mathbb{Q}_\ell)^{\otimes(2,2)}$ and $\mathbf{H}_{\mathrm{dR}}^{\otimes(2,2)}$ respectively (cf.\ \cite[Prop.~3.11]{PK2016}). We call the triple of the $\mathbb{Z}/2\mathbb{Z}$-grading, left $\mathrm{Cl}(L)$ action and various realizations of $\boldsymbol{\pi}$ the $\mathrm{CSpin}$ structure on the universal abelian scheme $\mathscr{A}$. Note that $\mathbf{L}_B$ equals the dual of $\boldsymbol{\pi}_B(\mathrm{End}(\mathbf{H}_B))$. We denote the dual of the images $\boldsymbol{\pi}_\ell(\mathrm{End}(\mathbf{H}_\ell))$ and $\boldsymbol{\pi}_{\mathrm{dR}}(\mathrm{End}(\mathbf{H}_{\mathrm{dR}}))$ by $\mathbf{L}_\ell$ and $\mathbf{L}_{\mathrm{dR}}$.

In our situation, we will always assume that $L^\vee_p/L_p$ is cyclic and $p^2\nmid \disc(L_p)$. Then $\mathrm{Sh}_\sfK(L)$ has a canonical integral model $\mathscr{S}_\sfK(L)$ over $\ZZ_{(p)}$. The constructions $\mathscr{A}\ \text{and}\ \mathbf{L}_\ell$ can be extended to $\mathscr{S}_\sfK(L)$, and we also have a crystalline construction $\mathbf{L}_\cris$ over $\mathscr{S}_K(L)\otimes\FF_p$. See \cite[Section 7]{PK2016} for the details. The integral model $\mathscr{S}_\sfK^{\mathrm{ad}}(L)$ of $\mathrm{Sh}_\sfK^{\mathrm{ad}}(L)$ is constructed as an \'etale quotient of $\mathscr{S}_\sfK(L)$, and the sheaves $\mathbf{L}_B,\mathbf{L}_\dR,\mathbf{L}_\ell,\mathbf{L}_p,\mathbf{L}_\cris$ on various fibers of $\mathscr{S}_\sfK(L)$ descend to the corresponding fibers of $\mathscr{S}_\sfK^{\mathrm{ad}}(L)$. We denote them by the same letters. 

For any point $s$ on $\mathscr{S}_\sfK(L)$ whose residue field $k(s)$ has characteristic $p>2$, an endomorphism $f\in \End(\mathscr{A}_s)$ is called a $\emph{special endomorphism}$ if the cohomological realization of $f_{\bar{s}}$ lies in $\mathbf{L}_{\ell,\bar{s}}$ and $\mathbf{L}_{\cris,\bar{s}}$. The space of special endomorphisms is denoted by $\mathrm{LEnd}(\mathscr{A}_s)$, equipped with a natural quadratic form given by $f\mapsto f\circ f\in\ZZ$. We have the following proposition: 

\begin{proposition}[{\cite[Proposition 3.2.3]{YangISV}}]\label{prop: special endomorphisms}
    Let $k$ be an algebraically closed field of characteristic $p>2$. For any $k$-point $s$ on the supersingular locus $\mathscr{S}_\sfK^{ss}(L)$, we have
    \begin{enumeratea}
        \item $\mathrm{LEnd}(\mathscr{A}_s)\otimes \RR$ is negative definite.
        \item The natural maps $\mathrm{LEnd}(\mathscr{A}_s)\otimes\ZZ_\ell\to \mathbf{L}_{\ell,s}$ and $\mathrm{LEnd}(\mathscr{A}_s)\otimes\ZZ_p\to\mathrm{L}_{\cris,s}^{F=1}$ are isomorphisms.
    \end{enumeratea}
\end{proposition}

Choose a primitive polarization $\xi$ on $X$. By applying Artin's approximation theorem to $\Def(X,\xi)$, we obtain a family $(\cX\to S,\boldsymbol{\xi})$. Up to shrinking $S$, we may assume that $S$ is connected, flat over $W$, quasi-healthy and has connected generic fibers. By \Cref{prop: spreading out in positive characteristic}, every geometric fiber over $S\otimes \FF_p$ is an irreducible symplectic variety of same deformation type. Moreover, we can assume that the family $(\cX\to S,\boldsymbol{\xi})$ satisfies the following two properties:
\begin{itemize}
    \item $\cX$ is everywhere a universal deformation,
    \item there is a pointed lattice $(\Lambda,\lambda)$ such that for every geometric point $s\to S$, $(\rH^2_\et(\cX_s,\widehat{\ZZ}^p),c_1(\boldsymbol{\xi}_s))\cong(\Lambda\otimes \widehat{\ZZ}^p,\lambda)$, where the left side is equipped with the Beauville Bogomolov form.
\end{itemize}
For the family $(\cX\to S,\boldsymbol{\xi})$, we denote by $\mathbf{H}^2_*$ the second relative cohomology of the suitable fibers of $\cX\to S$ and by $\mathbf{P}^2_*$ its primitive part, where $*$ means $B,\ell,\dR,\cris\ \text{or}\ \widehat{\ZZ}^p$. Set $L=\lambda^\perp$. As in \cite[3.3.4]{YangISV}, let $\widetilde{S}$ be the double cover of $S$ such that a morphism $T \to \widetilde{S}$ corresponds to a morphism \(T \to S\) together with an isometric trivialization
\[
\epsilon_{2,T}\colon \underline{\det(L_2)}_T \simeq \det(\mathbf{P}^2_{2,T}).
\]
Let \(\widetilde{S}^{\#}\) be the finite étale cover of \(\widetilde{S}\) such that a morphism \(T \to \widetilde{S}^{\#}\) corresponds to a morphism \(T \to \widetilde{S}\) and an isometric trivialization
\[
\Delta_T\colon \underline{\mathrm{disc}(L^p)}_T
\simeq
\mathrm{disc}(\mathbf{P}^2_{\widehat{\mathbb{Z}}^p,T}).
\]

By the universal property of $\mathrm{Sh}_\sfK^{\mathrm{ad}}(L)$ (See \cite[Proposition 3.3.5]{YangISV} or \cite[Proposition 4.3]{Pe15}), we have a period morphism $\rho_\CC: \widetilde{S}^{\#}_\CC\to \mathrm{Sh}_\sfK^{\mathrm{ad}}(L)_\CC$ and we have isometries $\alpha_B:\rho_\CC^*\mathbf{L}_B(-1)\to \mathbf{P}_{B,\CC}^2$ and $\alpha_{\dR,\CC}:\rho_\CC^*\mathbf{L}_{\dR}(-1)\to \mathbf{P}_{\dR,\CC}^2$. The same argument in \cite[3.3.6, 3.3.8]{YangISV} shows that $\rho_\CC$ first descends to $\rho_\sfK:\widetilde{S}^{\#}_K\to \mathrm{Sh}_\sfK^{\mathrm{ad}}(L)_K$, and then extends to $\rho:\widetilde{S}^{\#}\to \mathscr{S}_\sfK^{\mathrm{ad}}(L)$.

We now prove the Tate conjecture for divisors.
\begin{theorem}[Supersingular Tate conjecture for divisors]\label{thm: Tate conjecture for divisors}
Let $k$ be an algebraically closed field of characteristic $p$ and $X$ be a supersingular excellent irreducible symplectic variety over $k$ such that $\dim X<p$ and $h^{1,1}(X)\geq 2$. Suppose that $X$ is of $M$-type such that the discriminant $\disc(M)$ and the Fujiki constant $c_M$ both have trivial $p$-adic valuation.  Then the maps $c_1:\NS(X)\otimes\ZZ_\ell\to\rH^2_\et(X,\ZZ_\ell(1))$ and $c_1:\NS(X)\otimes\ZZ_p\to\rH^2_{\cris}(X/W)^{F=p}$ are isomorphisms for all $\ell\neq p$. 
\end{theorem}
\begin{proof}
    By \Cref{prop: integral first chern class map}, it is sufficient to prove this theorem rationally. Let $t$ be a point in $\widetilde{S}^{\#}(k)$ with fiber $\cX_t\cong X$ and $s$ be a point in $\mathscr{S}_{\mathsf{K}}(L)$ lifting $\rho(t)$. By the same proof of \cite[Theorem 3.3.9]{YangISV}, the period morphism $\rho$ gives us two commutative diagrams:
\[\begin{tikzcd}
	{\mathrm{LEnd}(\mathscr{A}_s)} & {\langle\boldsymbol{\xi}_t^\perp\rangle} & {\mathrm{LEnd}(\mathscr{A}_s)} & {\langle\boldsymbol{\xi}_t^\perp\rangle} \\
	{\mathbf{L}_{\cris}^{F=1}} & {\mathrm{P}_{\cris}^2(\cX_t/W)^{F=p}} & {\mathbf{L}_\ell} & {\mathrm{P}_{\et}^2(\cX_t,\ZZ_\ell(1))}
	\arrow[from=1-1, to=1-2]
	\arrow[from=1-1, to=2-1]
	\arrow[from=1-2, to=2-2]
	\arrow[from=1-3, to=1-4]
	\arrow[from=1-3, to=2-3]
	\arrow[from=1-4, to=2-4]
	\arrow["{\alpha_{\cris,t}}", from=2-1, to=2-2]
	\arrow["{\alpha_{\ell,t}}", from=2-3, to=2-4]
\end{tikzcd}\]
    The bottom horizontal arrows are isomorphisms. Moreover, since $X$ is supersingular, \Cref{prop: special endomorphisms} gives that the left vertical morphisms are isomorphisms rationally. Thus the right vertical arrows are surjective rationally, then the rational Tate conjecture holds.
\end{proof}

 \begin{remark}[Supersingular Torelli]
As is kindly pointed out to us by Ziquan Yang, based on the deformation theory developped in this section, one can establish a supersingular crystalline Torelli theorem for irreducible symplectic varieties of the known deformation types, generalizing his result \cite[Theorem~4]{YangISV} for varieties of $\mathrm{K}3^{[n]}$-type. Roughly speaking, under suitable hypotheses on the characteristic and the liftings, an isomorphism between the N\'eron--Severi lattices of two such supersingular varieties is induced by an isomorphism of the varieties if it preserves an ample class and extends compatibly to an isomorphism of their second crystalline cohomology groups and to an appropriate \'etale parallel transport operator. The proof follows essentially the same strategy as in loc.\ cit. Since the present paper focuses on the Tate conjecture and does not use this Torelli theorem, we leave the readers to give the precise formulation and the proof by adapting \cite{YangISV}.
\end{remark}

\subsection{Lifting with line bundles}
Let $k$ be an algebraically closed field of characteristic $p$ and $X$ be an excellent irreducible symplectic variety over $k$ such that $\dim X<p$ and $h^{1,1}(X)\geq 2$. Suppose that $X$ is of $M$-deformation type such that the discriminant $\disc(M)$ and the Fujiki constant $c_M$ both have trivial $p$-adic valuations. Let $m\in\NN$ and $(\zeta_0,\cdots,\zeta_m)$ be line bundles of $X$ generating a direct summand of $\Pic(X)$ and $\zeta_0=\xi$ being a polarization. We want to study whether $(X_,\zeta_0,\cdots,\zeta_m) $ can be lifted to a finite flat extension $W'$ of $W$. Some special cases are already known.
\begin{itemize}
    \item If $X$ is non-supersingular, this is always true for any $m$ by \Cref{prop: local deformation}(a).
    \item If $m=0$, we only need to lift one line bundle, then \Cref{prop: local deformation} also gives the lifting. 
\end{itemize}

To deal with the supersingular case, our strategy is to deform $(X,\zeta_0,\cdots,\zeta_m)$ to a non-supersingular irreducible symplectic variety, then lift it. To ensure this, recall that we have obtained an algebraic universal deformation family $(\cX\to S,\boldsymbol{\xi})$. We want to bound the dimension of  the supersingular locus $S_0^{ss}$ in $S_0=S\otimes k$.

\begin{lemma}\label{lem: dimension of supersingular locus}
    If $X$ is supersingular with $b_2(X)\geq 4$, then the dimension of the supersingular locus $S_0^{ss}$ is at most $\lfloor\frac{b_2}2{\rfloor-1}$. Here, $\lfloor\frac{b_2}2{\rfloor}=\max\{i\in\ZZ:i\leq\frac{b_2}2\}$.
\end{lemma}
\begin{proof}
    The proof is the same as that of \cite[Lemma 4.1.2]{YangISV}. The only additional point is that the Artin invariant $\sigma_0$ of a supersingular $M$-type satisfies $2\sigma_0\leq b_2$. Then we have the given bound.
\end{proof}

Now, we are able to discuss the lifting problem as follows. In \cite[Proposition 4.1.3]{YangISV}, the author shows the result for $K3^{[n]}$-type irreducible symplectic variety. But it is easy to see that the same argument works for general $M$-type.
\begin{proposition}\label{prop: lifting with line bundles}
    Let $(X,\xi=\zeta_0,\cdots,\zeta_N)$ be as above, where $N=\max\{0,\lceil \frac{b_2(X)}{2}\rceil-3\}$. Here, $\lceil \frac{b_2(X)}{2}\rceil=\min\{i\in\ZZ:i\geq \lceil \frac{b_2(X)}{2}\rceil\} $. Then there is a finite flat extension $W'$ of $W$ and a lifting $\widehat{X}$ of $X$ to $W'$ such that $(\zeta_0,\cdots,\zeta_N)$ deforms to $\widehat{X}$.
\end{proposition}
\begin{proof}
    We may assume that $b_2(X)\geq 6$ and $X$ is supersingular, since other cases are already known. This generalizes \cite[Proposition 1.5]{Ch16}, building on \cite[Proposition A.1]{LO15}. The proof adapts arguments from \cite[Appendix A]{LO15} with technical refinements.

    First, choose $m\gg 0$ which is prime to $p$ such that $m\xi$ is very ample and $\rH^i(X, m\xi) = 0$ for all $i > 0$. Let $P$ be the Hilbert polynomial of $m\xi$ and consider $\mathrm{Hilb}^{+,m}_P$. Denote this Hilbert scheme by $U$ with universal family $\mathcal{Y} \to U$. Let $s\in U$ be a point with fiber $\cY_s\cong X$. Construct the $(N+1)$-fold fiber product of the relative Picard scheme:
    \[
    \mathcal{P} := \mathrm{Pic}_{\mathcal{Y}/U} \times_U \cdots \times_U \mathrm{Pic}_{\mathcal{Y}/U}
    \]
    The tuple $(\xi, \zeta_1, \ldots, \zeta_N)$ corresponds to $t \in \mathcal{P}$ over $s \in U$. 
    
    We then find a characteristic zero point of $\cP$ that specializes to $t$. To see this, note that
    \[
    \dim \operatorname{Def}(X; \xi, \zeta_1, \ldots, \zeta_N) \geq h^{1,1}-(N+1)= \lfloor\frac{b_2}{2}\rfloor.
    \]
    Using this dimension estimate and the fact that the formal neighborhood of $s$ in $U$ is smooth over $\Def(X,\xi)$, we can find a point $s'\in U(k[[x]])$ extending $s$ such that the geometric generic fiber $\cY_{s'\otimes\bar{\eta}}$ is non-supersingular and carries lifts of all $\zeta_j$'s. We can lift $(\cY_{s'\otimes\bar{\eta}},\zeta_0,\cdots,\zeta_N)$ to $W(\bar{\eta})$, then there is a characteristic zero point on $\cP$ that specialize to $t$.

     By a standard argument involving systems of parameters, we find a $W'$ point $s_{W'}$ on $\cP$ extends $s$, where $W'$ is a finite flat extension of $W$. Set $\widehat{X}=\cY|_{s_{W'}}$. The isomorphism $\operatorname{Pic}(\widehat{X}_{W'}/W')(W') \cong \operatorname{Pic}(\widehat{X})$ by \cite[p. 203]{BLR1990} ensures that $(\zeta_0,\cdots,\zeta_N)$ extends to $\widehat{X}$.
     \end{proof}

As a corollary, we can show the following theorem, which means that the deformation type is independent of the choice of lifting, under some suitable conditions.
\begin{theorem}\label{prop: deformation type independent on lifting}
    Let $k$ be an algebraically closed field of characteristic $p>0$ and $X$ an excellent irreducible symplectic $k$-variety of deformation type $M$ with  $\dim X < p$ and $b_2(X)\geq7$. Suppose that the discriminant $\disc(M)$ and the Fujiki constant $c_M$ both have trivial $p$-adic valuation. Then for any deformation $\cX'$ of $X$ over a finite extension $V'$ of $W$ such that a primitive polarization on $X$ extends to $\cX'$, the geometric generic fiber of $\cX'$ is of deformation type $M$.
\end{theorem}
\begin{proof}
Suppose we have two liftings $\cX/V$ and $\cX'/V'$, and they carry an extension of primitive polarization $\xi$ and $\xi'$ respectively. Suppose that the generic fiber of $\cX$ is of $M$-type. By \Cref{prop: lifting with line bundles}, we can always find a lifting $\cX''/V''$ that carries both $\xi$ and $\xi'$. By \Cref{prop: spreading out in positive characteristic}, the generic fiber of $\cX''$ is of $M$-type. By \Cref{prop: spreading out in positive characteristic} again, the generic fiber of $\cX'$ is also of $M$-type.
\end{proof}

\begin{remark}\label{remark: primitivity is necessary}
    Let us explain the reason why the primitivity of the polarization in \Cref{prop: deformation type independent on lifting} is necessary. To compare two different liftings, we need both of them lie in a same connected family. This requires that the generic fiber of the local deformation family in \Cref{prop: local deformation} is connected. This holds when $\xi$ is primitive by \cite[Lemma 4.2.3]{YangISV}. However, when $\xi$ is not primitive, say $\xi=m\zeta$ with $\zeta$ be primitive, we cannot obtain the connectedness in general. If $p\nmid m$, then $\Def(X,m\zeta)=\Def(X,\zeta)$, which causes no problem. But if $p\mid m$, the connectedness can really fail. Here is a simple example. Consider an ordinary K3 surface $S$ with a polarization $\xi$. Assume that $\xi=p\zeta$ for a primitive polarization $\zeta$. Then by \cite[2.2.9]{Deligne1981b}, $\Def(X,\xi)=W[[x_1.\cdots,x_{20}]]/(x_1g(x_1))$. Here, $g(T)$ is the polynomial $g(T)=\frac{(T+1)^p-1}{T}$. Clearly the generic fiber has two connected components.
\end{remark}

\subsection{Integral Beauville--Bogomolov form}
In characteristic zero, the restriction of the primitive integral Beauville--Bogomolov form on $\rH^2(X,\ZZ)$ to $\NS(X)$ gives an integral quadratic form $q_X$ on $\NS(X)$, which  may not be primitive. Motivated by this, we construct an analogous integral quadratic form on $\NS(X)$ in characteristic $p$. Although the definition is intrinsic in characteristic $p$, we still need to lift line bundles to characteristic zero to verify some important properties. This is why we leave it to the end of this section.

\begin{proposition}\label{prop: integral BB form}
    Let $X$ be an excellent irreducible symplectic variety of deformation type $M$ over an algebraically closed field $k$ of characteristic $p>\dim X=2n$. Assume that $b_2(X)\geq 7 $ and the discriminant $\disc(M)$ and the Fujiki constant $c_M$ both have trivial $p$-adic valuation. Then there is a unique integral quadratic form $q_X$ on $\NS(X)$ such that
    \begin{enumeratea}
        \item (Fujiki relation) For any $L_i\in\NS(X)$ \[(L_1\cdot L_2\cdots L_{2n})= \frac{c_M}{n!\,2^n}\sum_{\sigma\in \mathfrak{S}_{2n}} q_X(L_{\sigma(1)},L_{\sigma(2)})\cdots q_X(L_{\sigma(2n-1)},L_{\sigma(2n)}),\] 
        \item $q_X$ is compatible, via the first Chern class map, with the Beauville--Bogomolov form constructed in \Cref{prop: BB form},
        \item for any ample divisor $L$, $q_X(L)$ is positive.
    \end{enumeratea}
\end{proposition}

\begin{proof}
    We first consider a quadratic form $q'$ on $\NS(X)$ given by
    \[
    q'(L)=\td_{2n-2}(X)\cdot L^2,
    \] 
    where $\td_{2n-2}(X)$ is the degree $2n-2$ part of the Todd class. This form may not satisfy (a) and (b), so we adjust it by a constant as follows.

    We only need to deal with primitive ample line bundles since they generate $\NS(X)$. For any primitive ample $L\in\NS(X)$, by \Cref{prop: local deformation}, the pair $(X, L)$ can be lifted to characteristic zero. Thus, we have a smooth projective family $(\cX,\cL)$ over a finite flat extension $V$ of $W$ whose special fiber is $(X,L)$ and generic fiber $(\cX_\eta,\cL_\eta)$ is an irreducible symplectic variety in characteristic zero. By \cite[Corollary 23.17]{CompactHK_Huy}, there is a constant $c'\in \QQ$, such that 
    \[
    \td_{2n-2}({\cX_\eta})\cdot\cL_{\eta}^2 =c'\cdot q_{\cX_\eta}(\cL_\eta),
    \]
    where $q_{\cX_\eta}$ is the integral BB-form in characteristic zero. The constant $c'$ depends only on the coefficients of the Riemann-Roch polynomial, whose values are positive by \cite{JiangChen2023}. By \Cref{prop: deformation type independent on lifting}, the constant $c'$ is independent of the choice of primitive ample line bundle and lifting.

    Thus, if we define $q_X=\frac{1}{c'}q'$, $q_X$ will satisfy (a) and (c) because intersection number is preserved in specialization. Condition (b) also follows because both sides satisfy the same Fujiki relation. The uniqueness is guaranteed by (a) and (c).
\end{proof}

\begin{proposition}
    The quadratic form $(\NS(X),q_X)$ is a non-degenerate lattice with signature $(1,{\rank(\NS(X))}-1)$. Moreover, if $X$ is of known deformation type, then $q_X$ is even.
\end{proposition}
\begin{proof}
    In our case, we always have $\NS(X)=\rN^1(X)$. Fix a primitive ample line bundle $H$, then $q_X(H)$ is a positive integer. By \Cref{lem:fujiki}, we have
    \[
    H^\perp:=\{L\in\NS(X):q_X(L,H)=0\}=\{L\in\NS(X):L\cdot H^{2n-1}=0\}.
    \]
    By \cite[Proposition 2.9]{HuFei2020}, we have for any $L\neq0\in H^\perp$, $L^2\cdot H^{2n-2}< 0$. Then $q_X(L)<0$ by \Cref{lem:fujiki}. Therefore, $q_X$ has signature $(1,{\rank(\NS(X))}-1)$. When $X$ is of known type, the evenness follows from the evenness of the Beauville--Bogomolov quadratic form for known types in characteristic zero.
\end{proof}

\begin{remark}
    The quadratic form $(\NS(X),q_X)$ is not deformation invariant, since the Artin invariant varies in connected families, as we will see in later sections. 
\end{remark}

\section{Classification of supersingular N\'eron--Severi lattices of known types}\label{sec:NS-classification}

Let $S$ be a supersingular K3 surface over an algebraically closed field of characteristic $p>0$. By Artin's work~\cite{Artin_K3}, its N\'eron--Severi lattice $\NS(S)$ satisfies $\disc(\NS(S)) = - p^{2\sigma(S)}$ for some integer $1 \leq \sigma(S) \leq 10$, where $\sigma(S)$ is the Artin invariant. Moreover, the lattice $\NS(S)$ is determined up to isometry by $\sigma(S)$. Such lattices are completely classified by Rudakov and Shafarevich~\cite{Shafarevich} and Shimada \cite{Shimada}. We generalize these results to supersingular irreducible symplectic varieties of known deformation types. Through out this section, we assume that $X$ is an excellent irreducible symplectic variety of known type and $\dim X<p$.

\subsection{Artin invariants}

To characterize $\NS(X)$ for supersingular irreducible symplectic varieties, we study $\NS(X) \otimes \ZZ_\ell$ at each prime $\ell$:
\begin{itemize}
    \item For $\ell \neq p$: By the supersingular Tate conjecture for divisors established in \Cref{thm: Tate conjecture for divisors}, we can identify $\NS(X)\otimes \ZZ_\ell$ with $\rH^2_\et(X,\ZZ_\ell)$. Then we lift it to characteristic zero and apply the base change theorem of cohomology.
    \item For $\ell = p$: We use crystalline cohomology. Again by \Cref{thm: Tate conjecture for divisors},
\end{itemize}
\begin{equation}\label{eq:p-adic}
    \NS(X) \otimes \ZZ_p  = \mathrm{H}^2_{\mathrm{\cris}}(X/W)^{F=p} =: \mathrm{T}_\mathrm{H}
\end{equation}
where $\mathrm{T}_\mathrm{H}$ (the $p$-eigenspace of Frobenius) is the Tate module. The bilinear form on $\rT_\rH$ inherited from $\mathrm{H}^2_{\mathrm{\cris}}(X/W)$ endows it with the structure of a $\ZZ_p$-lattice. This lattice is completely classified in \cite[Section 3]{Ogus1979}:

\begin{proposition}\label{prop:TateModule}\cite[Theorem 3.4]{Ogus1979}
    With notation as above:
    \begin{enumerate}
        \item $\disc(\mathrm{T}_{\mathrm{H}}) = d p^{2\sigma_0}$ where $\sigma_0 \geq 1$ (Artin invariant) and $d \in \ZZ_p^\times$ is defined modulo squares.
        \item $\mathrm{T}_\mathrm{H}$ decomposes as:
        \begin{equation}\label{eq:TateDecomp}
            \mathrm{T}_\mathrm{H} = \mathrm{T}_0(p) \oplus \mathrm{T}_1
        \end{equation}
        with $\mathrm{T}_0, \mathrm{T}_1$ unimodular $\ZZ_p$-lattices of ranks $2\sigma_0$ and $b_2(X)-2\sigma_0$, and discriminants $d_0, d_1$ satisfying:
        \[
        \left(\frac{d_0}{p}\right) = -\left(\frac{-1}{p}\right)^{\sigma_0}, \quad 
        \left(\frac{d_1}{p}\right) = -\left(\frac{-1}{p}\right)^{\sigma_0} \left(\frac{d}{p}\right)
        \]
        \item The discriminant group is $(\ZZ/p\ZZ)^{\oplus 2\sigma_0}$.
    \end{enumerate}
\end{proposition}

We want to show that $\NS(X)$ is determined by the Artin invariant $\sigma_0$ for supersingular irreducible symplectic variety of known type. This is a totally lattice-theoretic result, which can be geberallized to the following proposition:

\begin{proposition}\label{prop: lattice with unique genus}
    Let $L$ be an even indefinite lattice with $\rank(L)=n\geq 3$. For any prime $\ell$, we write $L_\ell=L\otimes\ZZ_\ell$. Suppose that
    \begin{enumerate1}
        \item There is an odd prime $p$ such that $p\nmid\disc(L)$ and the discriminant group of $L_p$ is of the form $\left(\ZZ/p\ZZ\right)^{2\sigma_0}$, where $\sigma_0$ is a nonnegative integer;
        \item For any prime $q\neq p$, $l(L_q)\leq n-2$. Here, $l(L_q)$ is the least number of generators of the discriminant group of $L_q$, as a finite Abelian group.
    \end{enumerate1}
    Then $L$ is uniquely determined by its signature and discriminant form, up to isomorphisms.
\end{proposition}
\begin{proof}
    We can first assume that $2\sigma_0<n$. If $2\sigma_0=n$, $L=L'(p)$ for another lattice $L'$. Then we can replace $L$ by $L'$.

    If $2\sigma_0\leq n-2$, then by Nikulin~\cite[Theorem 1.13.2]{Nikulin}, the lattice $L$ is determined by its signature  and discriminant form.

   If $2\sigma_0 = n-1$, Nikulin's conditions fail, but uniqueness still holds by a more explicit computation of the genus. We rely on the useful note \cite{MM09_Embeddings_of_Integral_Quadratic_Forms} of Miranda and Morrison to compute the genus $g(L)$. We start by introducing some basic notation. Let $\Gamma_q=\{\pm1\}\times\QQ_q^*/(\QQ_q^*)^2$, $\Gamma_{q,0}=\{\pm1\}\times\ZZ_q^*/(\ZZ_q^*)^2$ and $\Sigma(L_q)=\mathrm{Im}((\det,\spin):\rO(L_q)\to\Gamma_q)$. For a global description, define the adelic construction $\Gamma_\AA=\{(d_q,s_q)\in\underset{q}{\prod}\Gamma_q\mid (d_q,s_q)\in\Gamma_{q,0}\ \mathrm{for\ all\ but\ finite }\ q\}$ and $\Gamma_{\AA,0}=\underset{q}{\prod}\Gamma_{q,0}$. Then $\Sigma(L)=\underset{q}{\prod}\Sigma(L_q)$ is a subgroup of $\Gamma_\AA$. Now we compute $\Sigma(L)$. 
   
    By \cite[Theorem 12.5, 12.7, 12.9, Chapter VII]{MM09_Embeddings_of_Integral_Quadratic_Forms}, we have
    \begin{equation}
        \Sigma(L_p)=\Gamma_p
    \end{equation}
    For prime $\ell\nmid \disc(L)$,  by \cite[Corollary 12.11, Chapter VII]{MM09_Embeddings_of_Integral_Quadratic_Forms}, we have
    \begin{equation*}
        \Sigma(L_\ell)=\Gamma_{\ell,0}
    \end{equation*}
    For odd prime $q\mid\disc(L)$, by \cite[Theorem 12.5, 12.7, 12.9, Chapter VII]{MM09_Embeddings_of_Integral_Quadratic_Forms}, we have
    \[
    \Sigma(L_q)= \Gamma_{q}\ \text{or}\ \Gamma_{q,0},
    \]
    depending on whether the conditions in \cite[Theorem 12.7, Chapter VII]{MM09_Embeddings_of_Integral_Quadratic_Forms} hold.
    
    For prime $2$ if $2\mid\disc(L)$, by \cite[Theorem 12.6, 12.10, Chapter VII]{MM09_Embeddings_of_Integral_Quadratic_Forms}, we have
    \[
    \Sigma(L_2)\supset \Gamma_{2,0}.
    \]
    
     Note that for any prime $q$, $\Gamma_{q,0}\subset\Sigma(L_q)$. Hence $\Gamma_{\AA,0}\subset\Sigma(L)$. Then by \cite[Lemma 4.1, Chapter VIII]{MM09_Embeddings_of_Integral_Quadratic_Forms}, we have $\Gamma_\AA=\Gamma_\QQ\cdot\Sigma(L)$. It follows from \cite[Corollary 3.3, Chapter VIII]{MM09_Embeddings_of_Integral_Quadratic_Forms} that $L$ is determined by its signature and discriminant form.
\end{proof}

\begin{proposition}\label{prop: NS determined by Artin invariant}
    Let $X$ be a supersingular irreducible symplectic variety of known deformation type with Artin invariant $\sigma_0$. Then $\NS(X)$ is uniquely determined by $\sigma_0$ up to isomorphism.
\end{proposition}

\begin{proof}
    For $\ell \neq p$, lift $X$ to characteristic $0$ to obtain $X_0$. Cohomology base change yields:
    \[
    \NS(X) \otimes \ZZ_\ell = \mathrm{H}^2_{\text{\'et}}(X, \ZZ_\ell) = \mathrm{H}^2_{\text{\'et}}(X_0, \ZZ_\ell) = \mathrm{H}^2(X_0, \ZZ) \otimes \ZZ_\ell,
    \]
    which is known when $X$ is of known deformation type, as summarized in Table \ref{tab:BBform}. For $\ell = p$, we use Proposition~\ref{prop:TateModule}. Clearly, the conditions in \Cref{prop: lattice with unique genus} are satisfied. Hence $\NS(X)$ is determined by its signature and discriminant form, then by the Artin invariant $\sigma_0$.
\end{proof}

In the remainder of this section, we construct the N\'eron--Severi lattices of supersingular irreducible symplectic varieties of known types. Before we start, we introduce some notation about $p$-adic lattices for any prime $p$.

\begin{notation}
    Let $p$ be any prime number. There are $p$-adic lattices:
    \begin{itemize}
        \item $K_\theta^{(p)}(p^k)$ is the one-dimensional $p$-adic lattice given by the matrix $\langle\theta p^k\rangle$, where $\theta\in\ZZ_p^*/(\ZZ_p^*)^2$.
        \item $U^{(2)}(2^k)$ and $V^{(2)}(2^k)$ are two-dimensional $2$-adic lattices given by matrices 
        $\begin{pmatrix}
            0 & 2^k\\ 2^k & 0
        \end{pmatrix}$
         and
         $\begin{pmatrix}
             2^{k+1} & 2^k\\ 2^k & 2^{k+1}
         \end{pmatrix}$.
         \item $q_\theta^{(p)}(p^k),u^{(2)}(2^k),v^{(2)}(2^k)$ are discriminant forms of the above lattices.
    \end{itemize}
\end{notation}
In fact, the above lattices generate the semigroup of all $p$-adic lattices by \cite[Proposition 1.8.1]{Nikulin}, and are subject to the relations in \cite[Proposition 1.8.2]{Nikulin}.

\subsection{$\rK3^{[n]}$-type}
We begin by introducing some key lattices. For each $1 \leq \sigma_0 \leq 10$, define $\Lambda_{p,\sigma_0}$ as follows:
\begin{alignat*}{2}
    \sigma_0 &= 1, &\quad \Lambda_{p,\sigma_0} &= \rU \oplus \rH^{(p)} \oplus \rE_8^{\oplus 2} \\
    \sigma_0 &= 2, &\quad \Lambda_{p,\sigma_0} &= \rU(p) \oplus \rH^{(p)} \oplus \rE_8^{\oplus 2} \\
    \sigma_0 &= 3, &\quad \Lambda_{p,\sigma_0} &= \rU \oplus (\rH^{(p)})^{\oplus 3} \oplus \rE_8 \\
    \sigma_0 &= 4, &\quad \Lambda_{p,\sigma_0} &= \rU(p) \oplus (\rH^{(p)})^{\oplus 3} \oplus \rE_8 \\
    \sigma_0 &= 5, &\quad \Lambda_{p,\sigma_0} &= \rU \oplus \rH^{(p)} \oplus \rE_8 \oplus \rE_8(p) \\
    \sigma_0 &= 6, &\quad \Lambda_{p,\sigma_0} &= \rU(p) \oplus \rH^{(p)} \oplus \rE_8 \oplus \rE_8(p) \\
    \sigma_0 &= 7, &\quad \Lambda_{p,\sigma_0} &= \rU \oplus (\rH^{(p)})^{\oplus 3} \oplus \rE_8(p) \\
    \sigma_0 &= 8, &\quad \Lambda_{p,\sigma_0} &= \rU(p) \oplus (\rH^{(p)})^{\oplus 3} \oplus \rE_8(p) \\
    \sigma_0 &= 9, &\quad \Lambda_{p,\sigma_0} &= \rU \oplus \rH^{(p)} \oplus \rE_8(p)^{\oplus 2} \\
    \sigma_0 &= 10, &\quad \Lambda_{p,\sigma_0} &= \rU(p) \oplus \rH^{(p)} \oplus \rE_8(p)^{\oplus 2}
\end{alignat*}
where $\rH^{(p)}$ is the rank $4$ lattice with Gram matrix:
\[
\begin{pmatrix}
    2 & 1 & 0 & 0 \\
    1 & \frac{q+1}{2} & 0 & \gamma \\
    0 & 0 & \frac{p(q+1)}{2} & p \\
    0 & \gamma & p & \frac{2(p + \gamma^2)}{q}
\end{pmatrix}
\]
for prime $q \equiv 3 \pmod{8}$ satisfying $\left( \frac{-q}{p} \right) = -1$, and $\gamma \in \mathbb{Z}$ with $\gamma^2 + p \equiv 0 \pmod{q}$.

\begin{proposition}\label{prop: local property of H^p}
    Let $\ZZ_{(p)}$ be the localization at $(p)$. Then $\rH^{(p)}\otimes\ZZ_{(p)}=\langle2\rangle\oplus\langle\frac{q}{2}\rangle\oplus\langle\frac{p}{2}\rangle\oplus\langle\frac{2p}{q}\rangle$. Moreover, $\rH^{(p)}\otimes\ZZ_p=K_1^{(p)}(1)\oplus K_\epsilon^{(p)}(1)\oplus K_1^{(p)}(p)\oplus K_\epsilon^{(p)}(p)$, where $\left(\frac{\epsilon}{p}\right)=-\left(\frac{-1}{p}\right)$.
\end{proposition}
\begin{proof}
    Consider matrices
    {\[
    \renewcommand{\arraystretch}{1.5}
    P=\begin{pmatrix}
        1 & 0 & 0 & \frac{\gamma}{q}\\
        0 & 1 & 0 & -\frac{2\gamma}{q}\\
        0 & 0 & 1 & 0\\
        0 & 0 & 0 & 1
    \end{pmatrix}\ \text{and}\
    Q=\begin{pmatrix}
        1 & -\frac{1}{2} & 0 & 0\\
        0 & 1 & 0 & 0\\
        0 & 0 & 1 & 0\\
        0 & 0 & -\frac{q}{2} & 1
    \end{pmatrix}.
    \]}
    Let $M$ be the Gram matrix of $\rH^{(p)}$ given above, then a direct computation shows
    {\[
    \renewcommand{\arraystretch}{1.5}
    Q^TP^TMPQ=
    \begin{pmatrix}
        2 & 0 & 0 & 0\\
        0 & \frac{q}{2} & 0 & 0\\
        0 & 0 & \frac{p}{2} & 0\\
        0 & 0 & 0 & \frac{2p}{q}\\
    \end{pmatrix}.
    \]}
    This gives the description of $\rH^{(p)}\otimes\ZZ_{(p)}$. Using this description, we can prove that $\rH^{(p)}\otimes\ZZ_p=K_1^{(p)}(1)\oplus K_\epsilon^{(p)}(1)\oplus K_1^{(p)}(p)\oplus K_\epsilon^{(p)}(p)$ directly by the relation \cite[Proposition 1.8.2(a)]{Nikulin}.
\end{proof}

The following lemma is useful in our later construction.
\begin{lemma}\label{lemma: squares in Hp}
    For any $\delta\in \FF_p^*$, there exists $y\in\rH^{(p)}$ such that $y^2\equiv \delta \ \text{mod}\ p$.
\end{lemma}
\begin{proof}
    Let $L$ be a sublattice of $\rH^{(p)}$ whose Gram matrix is $\begin{pmatrix}
        2 & 1\\
        1 & \frac{q+1}{2}
    \end{pmatrix}$ with respect to the basis $\{e_1,e_2\}$. We only need to prove the statement for $L$. It suffices to find integers $m$ and $n$ such that $2m^2+2mn+\frac{q+1}{2}n^2\equiv \delta\ \text{mod}\ p$, then taking $y=me_1+ne_2$ completes the proof.

    To see this, consider the quadratic extension $\FF_p(\sqrt{-q})\supset \FF_p$ and the norm morphism
    \begin{align*}
        \mathrm{Nm}: \FF_p(\sqrt{-q})^* &\to \FF_p^*\\
        u+v\sqrt{-q} & \mapsto u^2+qv^2.
    \end{align*}
    The morphism Nm is surjective, so there are $u,v\in\FF_p$ such that $u^2+qv^2=2\delta$. Take any integers $m$ and $n$ such that $2m\equiv u-v\ \mathrm{mod}\ p$ and $n\equiv v\ \mathrm{mod}\ p$, then $2m^2+2mn+\frac{q+1}{2}n^2\equiv \frac{1}{2}(u^2+qv^2) \equiv \delta\ \text{mod}\ p$.
\end{proof}

Now we begin to construct $\NS(X)$. By \cite{Shafarevich,Shimada}, $\Lambda_{p,\sigma_0}(-1)$ is isomorphic to $\NS(S)$ for a supersingular K3 surface $S$ with Artin invariant $\sigma_0$.

\begin{proposition}\label{prop: K3^n small Artin invariant}
    Let $X$ be a supersingular $\rK3^{[n]}$-type variety with Artin invariant $1 \leq \sigma_0 \leq 10$. Then
    \[
    \NS(X) \cong \Lambda_{p,\sigma_0}(-1) \oplus \langle -2(n-1) \rangle.
    \] 
\end{proposition}

\begin{proof}
    By the discussion above, we only need to verify that $\Lambda_{p,\sigma_0}(-1)\oplus \langle-2(n-1)\rangle$ and $\NS(X)$ agree when completing at each prime number.

    For any odd prime $\ell\neq p$, it is easy to see that 
    \[
        \Lambda_{p,\sigma_0}(-1)\otimes\ZZ_\ell= K_{-1}^{(\ell)}(1)\oplus\left( K_{1}^{(\ell)}(1)\right)^{\oplus 21} =(\rU^{\oplus3}\oplus \rE_8(-1)^{\oplus2})\otimes \ZZ_\ell.
    \]
    Then we have
    \[
        \NS(X)\otimes\ZZ_\ell=(\Lambda_{p,\sigma_0}(-1)\oplus (-2(n-1)))\otimes\ZZ_\ell
    \]

    For the prime $2$, we have
    \[
        \Lambda_{p,\sigma_0}(-1)\otimes\ZZ_2= U^{(2)}(1)^{\oplus 11}=(\rU^{\oplus3}\oplus \rE_8(-1)^{\oplus2})\otimes \ZZ_2.
    \]
    Then we have
    \[
        \NS(X)\otimes\ZZ_2=(\Lambda_{p,\sigma_0}(-1)\oplus (-2(n-1)))\otimes\ZZ_2
    \]
    
    By \Cref{prop:TateModule}, we have
    \[
        \NS(X)\otimes\ZZ_p=\rT_0(p)\oplus\rT_1
    \]
where $\rT_0,\rT_1$ are unimodular $\ZZ_p$-lattices with rank $2\sigma_0$ and $23-2\sigma_0$ and discriminant $d_0$ and $d_1$ such that $(\frac{d_0}{p})=-(\frac{-1}{p})^{\sigma_0}$ and $(\frac{d_1}{p})=-(\frac{-1}{p})^{\sigma_0}(\frac{2(n-1)}{p})$

By a direct computation using \Cref{prop: local property of H^p}, we have
\[
    \Lambda_{p,\sigma_0}(-1)\otimes\ZZ_p=K^{(p)}_{\mu_{\sigma_0}}(p)\oplus \left( K_1^{(p)}(p)\right)^{\oplus (2\sigma_0-1)}\oplus K_{-\mu_{\sigma_0}}^{(p)}(1)\oplus \left( K_1^{(p)}(1)\right)^{\oplus (21-2\sigma_0)},
\]
where $\mu_{\sigma_0}\in\ZZ_p^*/(\ZZ_p^*)^2$ satisfies $\left(\frac{\mu_{\sigma_0}}{p}\right)=-\left(\frac{-1}{p}\right)^{\sigma_0}$. This means that
\[
\NS(X)\otimes\ZZ_p=(\Lambda_{p,\sigma_0}(-1)\oplus (-2(n-1)))\otimes\ZZ_p,
\]
completing the proof.
\end{proof}

\begin{remark}\label{remark:geometry of k3n}
    The above construction shows that, when $1\leq\sigma_0\leq10$, $\NS(X)$ can be seen as the orthogonal complement of a vector of norm $2n-2$ in the algebraic Mukai lattice $\tilde{\rH}(S)=\rU\oplus\NS(S)$ (see \cite[Section 4.2]{SSIV}), where $S$ is a supersingular K3 surface with Artin invariant $\sigma_0$. This is quite similar to the characteristic 0 case. This strategy fails when $\sigma_0=11$, for there is no K3 surface with Artin invariant 11. However, this suggests that $\NS(X)$ should come from the algebraic Mukai lattice of a twisted K3 surface $S$ with Artin invariant 10, as the following proposition explains.
\end{remark}

\begin{proposition}\label{prop: K3n_big}
    Let $X$ be a supersingular $\rK3^{[n]}$-type variety with Artin invariant $\sigma_0=11$. Then $\NS(X)$ is given as follows. Let $M=\rU(p)^{\oplus2}\oplus\rH^{(p)}(-1)\oplus\rE_8(-p)^{\oplus2}$, then there exists a primitive vector $v\in M$, unique up to $O(M)$, such that $v^2=2n-2$ and $\NS(X)=v^\perp\subset M$. Moreover, $\NS(X)$ contains $\rU(p)$ as a direct summand.
\end{proposition}
\begin{proof}
    We first show the existence of $v$. In fact, $v$ can be taken in $\rU(p)\oplus\rH^{(p)}(-1)$. By \Cref{lemma: squares in Hp}, take $y\in\rH^{(p)}(-1)$ such that $(\frac{y^2}{p})=(\frac{2n-2}{p})$, then there exist $k,m\in\ZZ$ such that $2kp+m^2y^2=2n-2$. Let $e,f\in\rU(p)$ be the standard basis, then $v=e+kf+my\in\rU(p)\oplus\rH^{(p)}(-1)$ satisfying $v^2=2n-2$ has the desired property. Since $M$ is $p$-elementary, $v$ has divisibility $1$.
    
    By \cite[Theorem 5.15]{BC20}, there is only one orbit of vectors with norm $2n-2$ and divisibility 1 in $M$. So $v^\perp$ is independent of the choice of $v$ up to isomorphism. By \cite[Lemma 5.2]{BC20}, the signature of $v^\perp$ is (1,22) and $q_{v^\perp}=q_{\NS(X)}$. Then by the proof of \Cref{prop: NS determined by Artin invariant}, we have $\NS(X)\cong v^\perp$. Finally, since $v$ can be taken in $\rU(p)\oplus\rH^{(p)}(-1)$, $\NS(X)$ certainly contains $\rU(p)$ as a direct summand.
\end{proof}

\subsection{$\mathrm{Kum}^n$-type}

The $\Kum^{n}$-type case is analogous to the $\rK3^{[n]}$-type case. We state the result.
\begin{proposition}\label{prop: Kum^n small Artin invariant}
    Let $X$ be a supersingular $\Kum^{n}$-type variety with Artin invariant $1\leq\sigma_0\leq 2$, then 
    \begin{enumeratea}
        \item  $\sigma_0=1$, $\NS(X)=\rU\oplus\rH^{(p)}(-1)\oplus(-2(n+1))$
        \item  $\sigma_0=2$, $\NS(X)=\rU(p)\oplus\rH^{(p)}(-1)\oplus(-2(n+1))$
    \end{enumeratea}
\end{proposition}
\begin{proof}
    It suffices to check that the two lattices agree at each place, as we just did in the case of $K3^{[n]}$-type.
\end{proof}

\begin{proposition}\label{prop: Kum^n_big}
    Let $X$ be a supersingular $\Kum^{n}$-type variety with Artin invariant $\sigma_0=3$, then $\NS(X)$ is given as follows. Let $N=\rU(p)^{\oplus2}\oplus\rH^{(p)}(-1)$, then there is a primitive $v\in N$, unique up to $O(N)$, such that $v^2=2n+2$ and $\NS(X)=v^\perp\subset N$. Moreover, $\NS(X)$ contains a $\rU(p)$ as direct summand. 
\end{proposition}
\begin{proof}
    The proof is the same as that of \Cref{prop: K3n_big}.
\end{proof}

\subsection{OG6 type}
For varieties of OG6-type, the case when the Artin invariant is not maximal is very similar to $\rK3^{[n]}$-type.
\begin{proposition}\label{prop: OG6 small Artin invariant}
    Let $X$ be a supersingular OG6-type variety with Artin invariant $1\leq\sigma_0\leq 3$, then
    \begin{enumeratea}
        \item  $\sigma_0=1$, $\NS(X)=\rU\oplus\rH^{(p)}(-1)\oplus(-2)^{\oplus2}$
        \item  $\sigma_0=2$, $\NS(X)=\rU(p)\oplus\rH^{(p)}(-1)\oplus(-2)^{\oplus2}$
        \item  $\sigma_0=3$, 
        $\NS(X)=\left\{
        \begin{aligned}
            &\rU({p})\oplus\rH^{(p)}(-1)\oplus(-2p)^{\oplus 2},\ (\frac{-1}{p})=1\\
            &\rU\oplus \rD_6(-p),\ (\frac{-1}{p})=-1
        \end{aligned}
        \right.$
    \end{enumeratea}
    where $\rD_6$ is the positive definite root lattice.
\end{proposition}
\begin{proof}
    Since $\sigma_0\leq 3$, $\NS(X)$ is determined by its discriminant form. Therefore, once we know the discriminant form of $\rD_6$ and $\rH^{(p)}(-1)$, the discriminant forms of all the blocks above are clear. A simple computation gives that $\rD_6$ has discriminant form $q_3^{(2)}(2)\oplus q_3^{(2)}(2)$ and $\rH^{(p)}(-1)$ has discriminant form $q_1^{(p)}(p)\oplus q_{-\epsilon}^{(p)}(p)\oplus q_1^{(p)}(1)\oplus q_{-\epsilon}^{(p)}(1)$, where $\left(\frac{\epsilon}{p}\right)=(\frac{-1}{p})$. Here we use the notation in \cite{Nikulin}. Then by checking discriminant form, we complete the proof.
\end{proof}

The above description of $\NS(X)$ is complete, since we can write down the explicit matrix of the lattice. However, when $\sigma_0=3$, the above description lacks a geometric interpretation as we explained in Remark \ref{remark:geometry of k3n} for $\rK3^{[n]}$-type and it is not well suited for later applications. Therefore, we give the following description. Although it has no explicit matrix description, it admits a geometric interpretation similar to that in characteristic 0 (see \cite[section 2]{Gro22} for a good introduction of characteristic 0 case).
\begin{proposition}\label{prop: OG6 another description}
    Let $X$ be a supersingular OG6-type variety with Artin invariant $\sigma_0= 3$. Let $L=\rU(p)^{\oplus2}\oplus\rH^{(p)}(-1)$ and $v\in L$ be primitive and satisfy $v^2=2$. Then $\NS(X)\cong v^{\perp}\oplus(-2)$.
\end{proposition}
\begin{proof}
    The existence and uniqueness of $v$ and $v^\perp$ are proved as in \Cref{prop: Kum^n_big}, using \cite{BC20}. Then the discriminant form of $v^{\perp}\oplus(-2)$ is given by $q=q_3^{(2)}(2)\oplus q_3^{(2)}(2) \oplus q_{-\epsilon}^{(p)}(p)\oplus q_1^{(p)}(p)^{\oplus5}$( \cite[Lemma 5.2]{BC20}), where $\epsilon=(\frac{-1}{p})$. It is easy to see that $q$ is isomorphic to the discriminant form of $\NS(X)$. Then $\NS(X)\cong v^{\perp}\oplus(-2)$ since $\NS(X)$ is unique in its genus.
\end{proof}

The case for OG6 with Artin invariant $\sigma_0=4$ is quite different.
\begin{proposition}\label{prop: OG6 big Artin invariant}
     Let $X$ be a supersingular OG6-type variety. Then the case $\sigma_0=4$ occurs only when $(\frac{-1}{p})=-1$; in this case, $\NS(X)=\rU(p)\oplus\rD_6(-p)$.  
\end{proposition}
\begin{proof}
    Since $\sigma_0=4$, we have 
    \begin{equation}
        \NS(X)\otimes\ZZ_p=p\rT_0
    \end{equation}
    Then $p$  divides $\langle x, y\rangle$ for all $x,y\in\NS(X)$.
    
    So we can define a new bilinear form on $\NS(X)$ by $(x,y)=\frac{1}{p}\langle x,y\rangle$. Denote this lattice by $N$. $N$ is an even hyperbolic lattice with discriminant group $\ZZ/2\ZZ\oplus\ZZ/2\ZZ$, which means $N$ is 2-elementary. Therefore, by \cite{Shafarevich}, $N$ is determined by its discriminant and whether it belongs to type I or Type II. But in our setting with $\rank(N)=8$, the type I case doesn't exist. So $N$ is determined by its discriminant. Notice that $\rU\oplus\rD_6(-1)$ is such a lattice, which must be isomorphic to $N$. So the possible choice of $\NS(X)$ can only be $\rU(p)\oplus\rD_6(-p)$.

    However, when $(\frac{-1}{p})=1$, $(\rU(p)\oplus\rD_6(-p))\otimes\ZZ_2$ has discriminant form $q_1^{(2)}(2)\oplus q_1^{(2)}(2)$. But $\NS(X)\otimes\ZZ_2$ has discriminant form $q_3^{(2)}(2)\oplus q_3^{(2)}(2)$. These two forms are different. Therefore, when $(\frac{-1}{p})=1$, $\sigma_0$ cannot be 4.
\end{proof}

\subsection{OG10 type}We will need a new lattice block, which is even negative definite with discriminant 3 and rank 6. We construct this lattice by a process called evenization. Let
\begin{equation}
    \rL=(-3)\oplus(-1)^{\oplus5}
\end{equation}
be a lattice with standard basis \{$e_i$\}.
Let $v=\frac12\sum_i e_i$; then
\begin{equation}
    \rL_0=\{\Sigma x_ie_i | \Sigma x_i\ \mathrm{even}\}+\{nv|n\in\ZZ\}
\end{equation}
is the lattice we want. Notice that $\{e_0-e_1,e_1-e_2,e_2-e_3,e3-e_4,2e_4,v\}$ is a basis of $\rL_0$. Under this basis, it has matrix
\begin{equation}
    \begin{pmatrix}
    -4 & 1 & 0 & 0 & 0 & -1\\
    1 & -2 & 1 & 0 & 0 & 0\\
    0 & 1 & -2 & 1 & 0 & 0\\
    0 & 0 & 1 & -2 & 2 & 0\\
    0 & 0 & 0 & 2 & -4 & -1\\
    -1 & 0 & 0 & 0 & -1 & -2
\end{pmatrix}
\end{equation}

Diagonalizing over $\ZZ_3$, we have
\begin{equation}
    \rL_0\otimes \ZZ_3=K_{-1}^{(3)}(3)\oplus\mathrm{unimodular\ lattice}
\end{equation}
Now, we can construct the N\'eron Severi lattice of OG10.

\begin{proposition}\label{prop: OG10 small Artin invariant}
    Let $X$ be a supersingular OG10-type variety with Artin invariant $1\leq\sigma_0\leq 11$, then
    \begin{enumeratea}
        \item $1\leq\sigma_0\leq 10$, $\NS(X)=\Lambda_{p,\sigma_0}(-1)\oplus \rA_2(-1)$
        \item $\sigma_0=11$,
        $\NS(X)=\left\{
        \begin{aligned}
            &\rU(p)\oplus\rH^{(p)}(-1)\oplus\rA_2(-p)\oplus\rE_8(-p)^{\oplus2},\ (\frac{-3}{p})=1\\
            &\rU\oplus \rE_8(-p)^{\oplus2}\oplus\rL_0(p),\ (\frac{-3}{p})=-1
        \end{aligned}
        \right.$
    \end{enumeratea}
\end{proposition}
\begin{proof}
    This is basically the same argument as in the proof of \Cref{prop: OG6 small Artin invariant}. We know $\NS(X)$ is determined by its discriminant form. So we only need to check whether the discriminant form of our given lattice equals the discriminant form of $\NS(X)$. This follows from the known discriminant forms of the blocks.
\end{proof}

Now, similar to \Cref{prop: OG6 another description}, we will give another description of $\NS(X)$ when $\sigma_0=11$. For geometric inspiration, see \cite[section 3]{FGG24} for a good introduction of characteristic 0 case.
\begin{proposition}\label{prop:OG10 another description}
    Let $X$ be a supersingular OG10-type variety with Artin invariant $\sigma_0= 11$. Then there exists a primitive $v\in\NS(X)$ with norm -6 and divisibility 3. Moreover $v^\perp\subset\NS(X)$ is isomorphic to $L$, given as follows. Let $M=\rU(p)^{\oplus2}\oplus\rH^{(p)}(-1)\oplus\rE_8(-p)^{\oplus2}$. There exists a primitive $w\in M$, unique up to $O(M)$, such that $w^2=2$. Then $L$ is given by $w^\perp\subset M$.
\end{proposition}
\begin{proof}
    The existence and uniqueness of $w$ follow as in the proof of \Cref{prop: K3n_big}, using \cite{BC20}. And the discriminant form of $L$ is 
    \[
    q_L=q_3^{(2)}(2)\oplus q_{-\epsilon}^{(p)}(p)\oplus q_1^{(p)}(p)^{\oplus21}.
    \]
    Here $\left(\frac{\epsilon}{p}\right)=(\frac{-1}{p})$.

    Let $\langle-6\rangle$ be a rank 1 lattice generated by $x$ such that $x^2=-6$. We want to find a primitive embedding $x\hookrightarrow \NS(X)$ with orthogonal complement $L$. Since $\NS(X)$ and $L$ are unique in their respective genera, we can use \cite[Proposition 1.5.1]{Nikulin}.

    The discriminant form of $(-6)$ is the group $\ZZ/6\ZZ$, generated by $\tilde{x}$ with quadratic form $q_x(\tilde{x})=-\frac{1}{6}\in\QQ/2\ZZ$. And the subgroup $q_3^{(2)}(2)$ of $q_L$ is the group $\ZZ/2\ZZ$, generated by $\tilde{y}$ with quadratic form $q(\tilde{y})=-\frac{1}{2}\in\QQ/2\ZZ$. Let $D$ be a subgroup of $\ZZ/6\ZZ$ generated by $3\tilde{x}$. Define $\gamma:D\to q_L$ by sending $3\tilde{x}$ to $\tilde{y}$. Then $q_L(\gamma(3\tilde{x}))=q_L(\tilde{y})=-\frac{1}{2}=-q_x(3\tilde{x})\in\QQ/2\ZZ$. Let $S_\gamma$ be the graph of $\gamma$ in $q_x\oplus q_L$. Actually, $S_\gamma$ is generated by $3\tilde{x}+\tilde{y}$. Then we have 
    \[
    S_\gamma^\perp/S_\gamma=q_1^{(3)}(3)\oplus q_{-\epsilon}^{(p)}(p)\oplus q_1^{(p)}(p)^{\oplus21}\cong q_{\NS(X)}.
    \]
    Therefore, by \cite[Proposition 1.5.1]{Nikulin}, we have a primitive embedding $x\hookrightarrow \NS(X)$ with orthogonal complement $L$.

    Now, consider $x$ as a primitive vector in $\NS(X)$. It remains to prove that $x$ has divisibility 3. Denote the divisibility of $x$ by $d$ and $h=\frac{6}{d}$. Define
    \begin{align*}
        \pi: \NS(X)&\to\langle x\rangle^\vee/\langle x \rangle\\
               y   &\mapsto -\frac{(x,y)}{6}x
    \end{align*}
    Let $H=\pi(\NS(X))\cong \frac{1}{h}\ZZ/\ZZ\cong \ZZ/h\ZZ$. We show that $H\cong \NS(X)/\langle x\rangle\oplus \langle x\rangle^\perp$, which is equivalent to $\ker \pi=\langle x\rangle\oplus \langle x\rangle^\perp$. In fact, if $y\in\ker \pi$, then $-6\mid(x,y)$. So $y=-\frac{(x,y)}{6}x+y+\frac{(x,y)}{6}x\in\langle x\rangle\oplus \langle x\rangle^\perp$, then $\ker \pi=\langle x\rangle\oplus \langle x\rangle^\perp$.

    Therefore, $h^2=[\NS(X):\langle x\rangle\oplus \langle x\rangle^\perp]^2=\frac{\det(\langle x\rangle)\cdot\det(\langle x\rangle^\perp)}{\det(\NS(X))}=4$. Thus, $d=3$.
\end{proof}

When $\sigma_0=12$, we have a result similar to \Cref{prop: OG6 big Artin invariant}.
\begin{proposition}\label{prop: OG10 big Artin invariant}
     Let $X$ be a supersingular OG10-type variety. Then the case $\sigma_0=12$ only occurs when $(\frac{-3}{p})=-1$. In this case,  $\NS(X)=\rU(p)\oplus \rE_8(-p)^{\oplus2}\oplus\rL_0(p)$.  
\end{proposition}
\begin{proof}
    As in the proof of \Cref{prop: OG6 big Artin invariant}, $\NS(X)$ is a 3-elementary lattice scaled by $p$. By \cite{Shafarevich}, it is determined by its discriminant. Thus $\NS(X)$ must be $\rU(p)\oplus \rE_8(-p)^{\oplus2}\oplus\rL_0(p)$. As before, we check the discriminant form to obtain $(\frac{-3}{p})$ cannot be 1.
\end{proof}

\section{Irreducible symplectic varieties from moduli spaces of sheaves}\label{sec:bridgeland}

In this section, we introduce examples of irreducible symplectic varieties of known deformation types given by moduli spaces of  Bridgeland stable objects. We then give numerical criterion to determine when an irreducible symplectic variety arises from a Bridgeland moduli space. Throughout this section, we will work over an algebraically closed field $F$ of characteristic zero. By the Lefschetz principle, we may assume that $F$ can be embedded into $\CC$. For any smooth projective variety $X$ over $F$, we will write $\rH^i(X,\ZZ)$ and $\rH^i(X,\QQ)$ for $\rH^i(X_\CC,\ZZ)$ and $\rH^i(X_\CC,\QQ)$.

\subsection{Twisted sheaves and the Mukai lattice}

Let $S$ be a projective K3 surface or abelian surface defined over $F$.  Let $\srS\rightarrow S$ be a $\mu_m$-gerbe over $S$. This corresponds to a pair $(S,\alpha)$ for some  $\alpha\in H^2_{\et}(S,\mu_m)$, where the cohomology is taken with respect to the \'etale topology. For any integer $m$, there is a Kummer exact sequence 
$$1\rightarrow \mu_m\rightarrow {\mathbb G}_{\mathrm{m}} \xrightarrow{x\mapsto x^m}  {\mathbb G}_{\mathrm{m}}\rightarrow 1$$ and it  induces a surjective map 
\begin{equation}\label{braumap}
\rH^2_{\et}(S,\mu_m)\rightarrow \Br(S)[m]. 
\end{equation}
We denote by $[\srS]$ the image of $\alpha$ in $ \Br(S)[m]$.

By Artin's comparison theorem, since $\rH^3(S,\ZZ)$ is torsion free, there is a $B$-field $B\in\rH^2(S,\QQ)$ such that $[\srS]$ is the image of $B$, see \cite{H_St2005}. We set
$$e^B=(1,B,\frac{B^2}{2})\in\widetilde{\rH}(S,\QQ)$$
and there is an isometry with respect to the Mukai pairing
\begin{align*}
    \exp(B): \widetilde{\rH}(S,\QQ)&\to\widetilde{\rH}(S,\QQ)\\
    (r,c,s)&\mapsto(r,c+rB,\frac{rB^2}{2}+s+c\cdot B)
\end{align*}

We now recall some basic definitions of Mukai lattice and its twisted version, which are fundamental for the study of moduli spaces of sheaves. We refer to \cite[section 3]{LMS14} and \cite[Section 2]{Bragg_2023} for a more comprehensive discussion.

The Mukai lattice of $S$ is given by 
        \[
        \widetilde{\rH}(S)=\bigoplus_{i=0}^2 \rH^{2i}(S,\ZZ(i-1)),
        \]
equipped with the natural Mukai pairing and the weight-$2$ integral Hodge structure. It has a natural subgroup, called the extended N\'eron--Severi lattice, given by
        \[
        N(S)= \ZZ\oplus\NS(S)\oplus \ZZ.
        \]

We now turn to the twisted version of the above definitions. The twisted Mukai lattice of $\srS\to S$ is given by
\[
\widetilde{\rH}(\srS)=e^B \widetilde{\rH}(S)\subset \widetilde{\rH}(S)\otimes\QQ.
\]
The twisted N\'eron--Severi lattice is given by 
\[
N(\srS)=e^B N(S)\subset\widetilde{\rH}(\srS).
\]

\begin{definition}
    An $\srS$-twisted sheaf $\cF$ on $\srS$ is an $\cO_\srS$ -module compatible with the $\mu_m$-gerbe structure (see \cite[Def 2.1.2.4]{Lie07}). 
\end{definition}

With the notation as above, the Mukai vector of $\cF$ is defined as
$$v^B(\cF)=e^B(\ch_\srS(\cF)\cdot\sqrt{\td(S)})\in N(\srS),$$
 where $\ch_\srS(\cF)$ is the twisted Chern character of $\cF$ (see \cite[3.3.4]{LMS14}).

\begin{definition}
    A vector $v=(r,\zeta,b)$ in $N(\srS)$ is called a Mukai vector. The vector $v$ is called primitive if $e^{-B}v$ is primitive in $N(S)$.
\end{definition}

\begin{remark}
    Via the Artin comparison theorem, the above construction of $B$-field, Mukai lattice, and Mukai vector can be generalized to $\ell$-adic cohomology groups together with their Galois representations. We will use these in later sections directly without further specification.
\end{remark}

\subsection{Irreducible symplectic varieties via moduli space of (twisted) sheaves}

Let $S$ be a K3 surface or abelian surface, and $\srS\to S$ be a $\mu_m$-gerbe with Brauer class $\alpha$. After fixing a Mukai vector $v\in\widetilde{H}(\srS)$ and a polarization $H$ on $S$, we have the following coarse moduli spaces:
\begin{itemize}
    \item The smooth but in general non-proper moduli space $\cM^{\rms\rmt}_H(\srS,v)$ of stable twisted sheaves on $\srS$.
    \item The proper but in general singular moduli space $\cM_H(\srS,v)$ of semistable twisted sheaves on $\srS$.
\end{itemize}

When $v$ is primitive and $H$ is $v$-generic, $\cM_H(\srS,v)=\cM^{\rms\rmt}_H(\srS,v)$ is a smooth projective moduli space. These moduli spaces, or their Albanese fibers, turn out to be irreducible symplectic varieties of $K3^{[n]}$-type or $\Kum^n$-type by the work of Mukai and Yoshioka.

\begin{theorem}[\cite{Yo06, DMMN26}]
\label{theorem:Bridgeland of K3^n type and Kum^n type}
    Let $(\srS,H,v)$ be as above and $v^2=2n$. Then $\cM_H(\srS,v)$ is a smooth projective symplectic variety of dimension $2n+2$. Moreover,
    \begin{enumeratea}
        \item If $S$ is a K3 surface, then $\cM_H(\srS,v)$ is an irreducible symplectic variety of $K3^{[n+1]}$-type. There is a Hodge isometry
        \[
        \theta: \widetilde{\rH}(\srS)\supset v^\perp \to \rH^2(\cM_H(\srS,v),\ZZ)
        \]
        induced by the quasi-universal sheaf.
        \item If $S$ is an abelian surface and $2n\geq4$, then the Albanese fiber $\cK_H(\srS,v)$ of $\cM_H(\srS,v)$ is an irreducible symplectic variety of generalized Kummer deformation type. There is a Hodge isometry 
        \[
        \theta: \widetilde{\rH}(\srS)\supset v^\perp \to \rH^2(\cK_H(\srS,v),\ZZ)
        \]
        induced by the quasi-universal sheaf.
    \end{enumeratea}
\end{theorem}

On the other hand, if $v$ is not primitive, the moduli space $\cM_H(\srS,v)$ is singular. In the special case of $v=(2,0,-2)$, O'Grady gives a crepant resolution $\widetilde{\cM}_H(\srS,v)$. And then Sorger and Lehn generalized this construction to the case when $v=2v_0$ with $v_0$ primitive and $v_0^2=2$. These constructions give two new types of irreducible symplectic varieties as follows:

\begin{theorem}[\cite{OGrady_K3, OGrady_Abelian, LS06}]
\label{theorem: OG type example}
    Using the above notations, $\widetilde{\cM}_H(\srS,v)$ is a smooth projective symplectic variety. Moreover,
    \begin{enumeratea}
        \item When $S$ is a K3 surface, $\widetilde{\cM}_H(\srS,v)$ is an irreducible symplectic variety of OG10-type. The class of the exceptional divisor $c_1(E)$ has norm $-6$ and divisibility $3$ in $\rH^2(\widetilde{\cM}_H(\srS,v),\ZZ)$, whose orthogonal complement is identified with $\rH^2(\cM_H(\srS,v),\ZZ)$. Moreover, there is a Hodge isometry:
        \[
        \theta : \rH^2(\cM_H(\srS,v),\ZZ)\to v^\perp\subset\widetilde{\rH}(\srS)
        \]
        \item When $S$ is an abelian surface, the Albanese fiber $\widetilde{\cK}_H(\srS,v)$ of $\widetilde{\cM}_H(\srS,v)$ is an irreducible symplectic variety of OG6-type. The class of the exceptional divisor $c_1(E)$ is divisible by $2$. Write $c_1(E)=2\delta$. Then $\delta$ has norm $-2$ and divisibility $2$ in $\rH^2(\widetilde{\cK}_H(\srS,v),\ZZ)$, whose orthogonal complement is identified with $\rH^2(\cK_H(\srS,v),\ZZ)$. Moreover, there is a Hodge isometry:
        \[
        \theta : \rH^2(\cK_H(\srS,v),\ZZ)\to v^\perp\subset\widetilde{\rH}(\srS)
        \]
    \end{enumeratea}
\end{theorem}

\begin{notation}
    Given an irreducible symplectic variety $X$ of known type, we will say that $X$ is \emph{(twisted) modular} if it is birational to one of those in \Cref{theorem:Bridgeland of K3^n type and Kum^n type,theorem: OG type example}.
\end{notation}

\subsection{A numerical  characterization}
Given an irreducible symplectic variety of known type over $F$, we want to give a numerical criterion to determine whether it is (twisted) modular. This question has been settled in recent years through the work of several authors. We summarize in the following theorem.

\begin{theorem}\label{Thm:moduli}
    Let $X$ be an irreducible symplectic variety over $F$ of known type. Then there exist twisted K3 or abelian surface $\srS$, algebraic Mukai vector $v$ and a $v$-generic polarization $H$,
    \begin{enumerate}
        \item $K3^{[n]}$-type:
            $X \bir \cM_H(\srS, v)$ if $\NS(X)$ contains a nonzero isotropic vector.
            \vspace{.1cm}
        \item {$Kum_n$-type:} 
            $X \bir \cK_H(\srS,v)$ if $\NS(X)$ contains a nonzero isotropic vector.
             \vspace{.1cm}
        \item {OG10 type:} 
            $X \bir  \widetilde{\cM}_H(\srS, {v})$ if $\NS(X)$ contains a nonzero isotropic vector whose image in $\rH^2(X,\ZZ)$ lies in $\sigma^\perp$. Here $\sigma\in H^{1,1}(X,\ZZ)$ has square $-6$ and divisibility $3$ in $\rH^2(X,\ZZ)$.
             \vspace{.1cm}
        \item {OG6 type:} 
             $X \bir  \widetilde{\cK}_H(\srS, {v})$ if $\NS(X)$ contains a nonzero isotropic vector whose image in $\rH^2(X,\ZZ)$ lies in $\sigma^\perp$. Here $\sigma\in \rH^{1,1}(X,\ZZ)$ has square $-2$ and divisibility $2$ in $\rH^2(X,\ZZ)$.
             \vspace{.1cm}
    \end{enumerate}
\end{theorem}
\begin{proof}
Over $\CC$, the result is known by the work of: \cite{Addington_2016,Huybrechts_2017} for $K3^{[n]}$-type; \cite{DMMN26,Mongardi_Monodromy_HK} for $\Kum^n$-type; \cite{FGG24} for OG10-type; \cite{Gro22} and \Cref{lem: numerical twisted moduli of OG10} for OG6-type. We take the $K3^{[n]}$-type as an example to show how to use a spreading out argument to obtain the result over $F$ and the other three cases are similar. Over $\CC$, we know that there is a tuple $(S',\srS'\to S',v',H')$ over $\CC$ such that $X_\CC\bir \cM_H(\srS',v')$. The tuple $(S',\srS'\to S',v',H')$ can be defined over a subfield $L$ of $\CC$, which is finitely generated over $F$. Let $U$ be an irreducible smooth variety over $F$ such that the generic point $\eta$ has residue field $L$. Up to shrinking $U$, we can assume that the tuple $(S',\srS'\to S',v',H')$ spreads out to a tuple $(\widetilde{S},\widetilde{\srS}\to \widetilde{S},\tilde{v},\tilde{H})$ over $U$. We construct the relative moduli space $\cM_{\tilde{H}}(\widetilde{\srS},\tilde{v})$, whose generic fiber is birational to the generic fiber $X_L$ of the trivial family $X\times_F U$. Thus, after shrinking $U$ if necessary, there exists an $F$-point $u\in U$ such that $X$ is birational to the fiber $\cM_{\tilde{H}_u}(\widetilde{\srS}_u,\tilde{v}_u)$. This proves the assertion.
\end{proof}

The next lemma finishes the proof of the previous theorem for the twisted OG6 moduli space case. It might be well-known, but we could not find a suitable reference, as the available references focus only on the untwisted case. So we give a proof here.

\begin{lemma}\label{lem: numerical twisted moduli of OG10}
    Let $X$ be an irreducible symplectic variety of OG6-type over $\CC$. Then there exist a twisted abelian surface $S$, an algebraic Mukai vector $v$ and a $v$-generic polarization $H$ such that $X$ is birational to $\widetilde{\cK}_H(\srS, {v})$ if and only if the following conditions hold:
    \begin{enumeratea}
        \item There is $\sigma\in\rH^{1,1}(X,\ZZ)$ that has square $-2$ and divisibility $2$ in $\rH^2(X,\ZZ)$.
        \item The embedding $\sigma^\perp\hookrightarrow \boldsymbol{\Lambda_8}=\rU^{\oplus 4}$ gives a Hodge structure on $\boldsymbol{\Lambda}_8$ such that $\boldsymbol{\Lambda}_8^{1,1}$ contains an isotropic vector.
    \end{enumeratea}
\end{lemma}
\begin{proof}
    The proof is similar to \cite[Theorem 1.1]{Gro22} but we are working in the twisted case. See also \cite[Proposition 2.16]{DMMN26} for a similar argument for the twisted modular $\Kum^n$-case. One direction follows from \Cref{theorem: OG type example}, so we only need to prove the other direction. Since $\boldsymbol{\Lambda}_8^{1,1}$ contains an isotropic vector, it must contain a copy of $\rU(m)$ for some integer $m$. By an argument similar to that in \cite[Lemma 2.6]{Huybrechts_2017} and Shioda's Torelli theorem \cite{Shioda_Peroid_Abelian_surface}, there is an abelian surface $S$ and a $\mu_k$-gerbe $\srS\to S$ such that there is a Hodge isometry
    \[
    \boldsymbol{\Lambda}_8\to \widetilde{\rH}(\srS).
    \]
    It is straightforward to see that $\sigma^\perp$ is isometric to $\rU^{\oplus 3}\oplus \langle-2\rangle$, and the image of $\sigma^\perp$ in $\widetilde{\rH}(\srS)$ has orthogonal complement generated by $w$ with $w^2=2$. We may assume $w$ is positive by replacing $w$ by $-w$ if necessary and let $v=2w$. Take a $v$-generic polarization $H$ on $S$, we can construct the OG6 variety $\widetilde{\cK}_H(\srS,v)$ and we have a Hodge isometry
    \[
    \rH^2(\widetilde{\cK}_H(\srS,v),\ZZ)=\rH^2({\cK}_H(\srS,v),\ZZ)\oplus \ZZ(\frac{1}{2}E)\to \sigma^\perp \oplus\ZZ\sigma=\rH^2(X,\ZZ).
    \]
    Then by the Torelli theorem for OG6-type irreducible symplectic variety, $X$ is birational to $\widetilde{\cK}_H(\srS,v)$.
\end{proof}

If we strengthen the conditions in \Cref{Thm:moduli} suitably, we can furthermore obtain that $X$ is birational to the untwisted moduli space, i.e. the gerbe $\srS\to S$ is trivial. But we will only need this for the generalized Kummer type in the next section, so we only state and prove it in this case while the other three cases are similar.

\begin{corollary}\label{Cor: untwisted moduli for the Kummer type}
    Let $X$ be an irreducible symplectic variety of $\Kum^n$-type over $F$. If there is a primitive vector $w\in\NS(X)$ such that $w^2=-(2n+2)$ and $w$ spans a direct summand of the lattice $\rH^2(X,\ZZ)$, then $X$ is birational to the generalized Kummer variety $K_n(A)$ for some abelian surface $A$.
\end{corollary}
\begin{proof}
    By the same spreading out argument as before, we only need to work over $\CC$. By the assumption, we can write the lattice $\rH^2(X,\ZZ)$ in the form $\rU^{\oplus 3}\oplus\langle w\rangle$. Then we can choose an embedding 
    \[
    \rH^2(X,\ZZ)\hookrightarrow \boldsymbol{\Lambda}_8
    \]
    by embedding the rank one summand $\langle w\rangle$ into one copy of $\rU$. The Hodge structure of $\boldsymbol{\Lambda}_8$ induced by $\rH^2(X,\ZZ)$ clearly satisfies that $\boldsymbol{\Lambda}_8^{1,1}$ contains a $\rU$. Then by \cite[Remark 4.4]{Mongardi_Monodromy_HK}, $X$ is birational to some $K_n(A)$.
\end{proof}

\subsection{Motive of the (twisted) modular irreducible symplectic variety}
So far, we have already seen plenty of examples of irreducible symplectic varieties arising from the moduli spaces of sheaves on a K3 surface or abelian surface. The geometry of the moduli spaces can be studied through the relationship with the underlying surfaces. In particular, as we will see in this section, the Chow motive of the (twisted) modular irreducible symplectic variety can be controlled by the Chow motive of the underlying surface. For the $\Kum^n$-type and OG6-type, our argument below only treats the untwisted case. To start, we need the result that the Chow motive of irreducible symplectic varieties over $\CC$ is birational invariant.

\begin{theorem}\label{thm:birational Chow motives}\cite[Theorem 3.13]{SSIV}
    Let $X$ and $Y$ be two birational irreducible symplectic varieties over $\CC$, then there is an isomorphism $\frh(X)\cong\frh(Y)$.
\end{theorem}

\begin{theorem}\label{thm: motive of Bridgelan moduli space}
    Let $X$ be as follows:
    \begin{enumeratea}
        \item When $S$ is a K3 surface and $\srS\to S$ is a $\mu_m$-gerbe, $X$ is the $\cM_H(\srS,v)$ in \Cref{theorem:Bridgeland of K3^n type and Kum^n type} or the $\widetilde{\cM}_H(\srS,v)$ in \Cref{theorem: OG type example}.
        \item When $S$ is an abelian surface, $X$ is the generalized Kummer variety $K_n(S)$ or the untwisted $\widetilde{\cK}_H(S,v)$ in \Cref{theorem: OG type example}.
    \end{enumeratea}
    Then the motive $\frh(X)$ belongs to the pseudo-abelian tensor subcategory of motives generated by the motive of $S$.
\end{theorem}
\begin{proof}
    Over $\CC$, the result is well known by \cite[Theorem 0.1]{Bu20}, \cite[Theorem 1.5]{Motive_Kummer}, \cite[Theorem 1.3]{Motive_OG10}, and \cite[Theorem 1.1]{Motive_OG6}. Over the field $F$, it follows easily from a similar spreading-out argument in the proof of \Cref{Thm:moduli}, together with the \Cref{lem: specialization of motives} below.
\end{proof}
The next lemma shows that the embedding of Chow motives is stable under specialization. For the sake of later sections, we work in a more general case.

\begin{lemma}\label{lem: specialization of motives}
    Let $R$ be a DVR and $U=\Spec R$. Suppose that $\cX$ and $\cY$ are smooth proper algebraic spaces over $U$ such that their special fibers and generic fibers are all smooth projective varieties. Let $\eta$ be the generic point of $U$ and $u$ be its closed point. If, on the generic fiber, $\frh(\cX_\eta)$ belongs to the pseudo-abelian tensor subcategory of motives generated by the motive of $\cY_\eta$, then, on the fiber over $u$, $\frh(\cX_u)$ also belongs to the pseudo-abelian tensor subcategory of motives generated by the motive of $\cY_u$.
\end{lemma}
\begin{proof}
    By \cite[\href{https://stacks.math.columbia.edu/tag/0EDQ}{Tag 0EDQ}]{stacks-project} and \cite[Section 6]{DW1998}, there is a structure of Chow ring for a smooth proper algebraic space and the Gysin homomorphism and the specialization homomorphism are well-defined, just like the usual theory for schemes.
    
    On the generic fiber, we have an embedding
    \[
    \frh(\cX_\eta)\hookrightarrow\bigoplus_{i=1}^n\frh(\cY_\eta^{s_i})(t_i).
    \]
    Let $t=\max_{i=1}^n(t_i)$ and let $w$ be any point in $U$ ($w=\eta\ \text{or}\ u$), then we have 
    \[
    \frh(\cY_w^{s_i})(t_i)=\frh(\cY_w^{s_i})(t)\otimes \LL^{t-t_i}=\left(\cY_w^{s_i}\times (\PP_{k(w)}^1)^{t-t_i},[\Delta_{\cY_w^{s_i}}]\times [\PP_{k(w)}^1\times\{e\}]^{t-t_i},t\right).
    \]
    Then
    \[
    \bigoplus_{i=1}^n\frh(\cY_w^{s_i})(t_i)=\left( \coprod_{i=1}^n(\cY_w^{s_i}\times (\PP_{k(w)}^1)^{t-t_i}),\sum_{i=1}^n  [\Delta_{\cY_w^{s_i}}]\times [\PP_{k(w)}^1\times\{e\}]^{t-t_i},t \right).
    \]
    For convenience, we denote $\sum_{i=1}^n  [\Delta_{\cY_w^{s_i}}]\times [\PP_{k(w)}^1\times\{e\}]^{t-t_i}$ by $c_w$, which is a correspondence in $\mathrm{Corr}^0(\cZ_w,\cZ_w)$. Clearly, $c_\eta$ specializes to $c_u$.

    We construct a new space $\cZ$ over $U$ by
    \[
    \cZ=\coprod_{i=1}^n\left(\cY^{s_i}\times (\PP_U^1)^{t-t_i}\right).
    \]

    Now, since $\frh(\cX_\eta)$ can be embedded into $(\cZ_\eta,c_\eta,t)$, there are morphisms $\alpha_\eta\in \Hom\left(\frh(X_\eta),(\cZ_\eta,c_\eta,t)\right)$ and $\beta_\eta\in \Hom\left((\cZ_\eta,c_\eta,t),\frh(X_\eta)\right)$ such that $\beta_\eta\circ\alpha_\eta=\id$. Applying specialization to $\alpha_\eta$ and $\beta_\eta$, we have correspondences $\alpha_u\in\mathrm{Corr}^{t}(\cX_u,\cZ_u)$ and $\beta_u\in\mathrm{Corr}^{-t}(\cZ_u,\cX_u)$. Since the specialization map preserves the intersection product and is compatible with pullback and pushforward, $\alpha_u$ and $\beta_u$ are morphisms between $\frh(\cX_u)$ and $(\cZ_u,c_u,t)$. Moreover, they give an embedding 
    \[
    \frh(\cX_u)\hookrightarrow (\cZ_u,c_u,t)=\bigoplus_{i=1}^n\frh(\cY_u^{s_i})(t_i).
    \]
    Therefore, $\frh(\cX_u)$ belongs to the pseudo-abelian tensor subcategory of motives generated by the motive $\frh(\cY_u)$.
\end{proof}

\section{Supersingular Tate conjecture for $\rK3^{[n]}$-type, $\Kum^n$-type, and OG10-type}\label{sec:three-types}

In this section, we prove the \Cref{thm:intro:SSMotive} of this paper for $\rK3^{[n]}$-type, $\Kum^n$-type, and OG10-type. The proofs are similar, with some minor technical adjustments. However, the proof for the OG6-type needs a new argument, so we leave it to the next section.

\begin{assumption}\label{Good settings}
    Throughout this section, let $X$ be a supersingular excellent irreducible symplectic variety over $k$ of the following types:
    \begin{enumeratei}
    \item $\rK3^{[n]}$-type,
    \item $\Kum^n$-type with Artin invariant $\sigma_0\neq 3$,
    \item OG10-type with Artin invariant $\sigma_0\neq 12$.
    \end{enumeratei}
\end{assumption}

\subsection{Lifting to (twisted) modular irreducible symplectic variety}

We first show that $X$ in  \Cref{Good settings} can be lifted to a (twisted) modular irreducible symplectic variety over a field of characteristic zero by the numerical criterion in \Cref{Thm:moduli}. This is based on the following proposition.
\begin{proposition}\label{prop: lifting to Bridgeland over K3}
Let $X$ be as in the case $(i)$ (resp. case $(iii)$) of \Cref{Good settings}.  Then there exist the following data:
\begin{itemize}
    \item a finite flat extension $V$ of $W$ and a smooth projective family $\cX\to \Spec V$ whose special fiber $\cX_0\cong X$,
    \item a K3 surface $S$ over $\bar{K}$ and a $\mu_m$-gerbe $\srS\to S$,
\end{itemize}
such that the generic fiber $\cX_{\bar{K}}$ is birational to an irreducible symplectic variety $\cM_H(\srS,v)$ (resp.  $\widetilde{\cM}_H(\srS,v)$) as in \Cref{theorem:Bridgeland of K3^n type and Kum^n type} (resp. \Cref{theorem: OG type example}). Moreover, we can arrange the underlying surface $S$ such that $S$ carries a line bundle $L$ of degree prime to $p$.
\end{proposition}
\begin{proof}
    By the classification of the lattice $\NS(X)$ in Section \ref{sec:NS-classification}, we can choose a suitable saturated sublattice $N\subset\NS(X)$ and then lift it to characteristic zero. If $X$ is of case (i), we can choose $N$ to contain:
    \begin{itemize}
        \item an ample line bundle,
        \item a sublattice $\rU\oplus \langle v\rangle$ or $\rU(p)\oplus \langle v\rangle$ with $p\nmid v^2$
    \end{itemize}
    by \Cref{prop: K3^n small Artin invariant} and \Cref{prop: K3n_big}. If $X$ is of case (iii), we can choose $N$ to contain:
    \begin{itemize}
        \item  an ample line bundle,
        \item a $\sigma\in \NS(X)$ with square $-6$ and divisibility $3$ in $ \NS(X)$,
        \item a sublattice $\rU\oplus \langle v\rangle$ or $\rU(p)\oplus \langle v\rangle$ with $p\nmid v^2$ which is orthogonal to $\sigma$.
    \end{itemize}
    by \Cref{prop: OG10 small Artin invariant} and \Cref{prop:OG10 another description}.  In each case, we can assume that $\rank (N)\leq 5$. Thus, by \Cref{prop: lifting with line bundles}, we can find a finite flat extension $V$ of $W$ and a smooth projective family $\cX\to \Spec V$ whose special fiber $\cX_0\cong X$, and the generic fiber $\cX_{\bar{K}}$ carries a lifting of $N$. If $X$ is of case (iii), by the Tate conjecture for divisors, \Cref{thm: Tate conjecture for divisors} and the Artin comparison theorem, we know that the image of $\sigma$ in $\rH^2(\cX_{\bar{K}},\ZZ)$ also has divisibility $3$. Now, by the numerical criterion in \Cref{Thm:moduli}, we obtain that $\cX_{\bar{K}}$ is birational to the (twisted) modular irreducible symplectic variety we need. 
    
    What remains is to show that $S$ carries a line bundle $L$ of degree prime to $p$. Recall that $S$ is obtained by the period argument in \cite[Lemma 2.6]{Huybrechts_2017}. So in order to show that $S$ carries a line bundle of degree prime to $p$, we only need a more explicit period computation in \Cref{lem: period of K3 with line bundle prime to p} below. This is why we need $N$ to contain $\rU\oplus v$ or $\rU(p)\oplus v$ with $p\nmid v^2$ .
\end{proof}

\begin{lemma}\label{lem: period of K3 with line bundle prime to p}
    Use the notation of \cite[lemma 2.6]{Huybrechts_2017}. Let $x\in Q$ be a period point. If the $(1,1)$-part of the Hodge structure $\widetilde{\Lambda}$ defined by $x$ contains a primitive sublattice isomorphic to $\rU\oplus\langle v\rangle$ or $\rU(p)\oplus \langle v\rangle$ with $p\nmid v^2$, then $x$ is the period point of a gerbe $\srS\to S$ such that  $S$ carries a line bundle $L$ of degree prime to $p$.
\end{lemma}
\begin{proof}
    We prove this lemma for the case when $\rU(p)\oplus \langle v\rangle\subset \widetilde{\Lambda}^{1,1}$. The case when $\rU\oplus \langle v\rangle\subset \widetilde{\Lambda}^{1,1}$ is much easier. Let $e,f$ be the standard basis of $\rU(p)$. Since $v^2$ is coprime to $p$, $e$ has divisibility $1$ in $v^{\perp}_{\tilde{\Lambda}}$. Thus there exists $w\in v^\perp$ such that $w^2=0$ and $(w,e)=1$. Then $\tilde{\Lambda}$ admits a orthogonal decomposition
    $$\tilde{\Lambda}=\rU\oplus\Lambda,$$ where $\rU$ is generated by $e,w$ and $v\in\Lambda$. As in \cite[Lemma 2.6]{Huybrechts_2017}, $(\tilde{\Lambda}\otimes\CC)^{2,0}$ is generated by $\sigma+\lambda e$ with $\sigma\in\Lambda\otimes\CC$ and $\lambda\in\CC$. By construction, $\sigma$ will be the period of $S$. Since $v\cdot\sigma=v\cdot(\sigma+\lambda e)=0$, we have $v\in\rH^{1,1}(S,\ZZ)$ corresponds to the first Chern class of a line bundle with degree coprime to $p$.
\end{proof}

\begin{remark}
    This explains why we cannot handle all the OG10-type cases by this argument. As shown in \Cref{prop: OG10 big Artin invariant}, if  $\left(\frac{-3}{p}\right)=-1$ and $\sigma_0=12$, there is no element in $\NS(X)$ with square $-6$ and divisibility $3$. Thus, the modular lifting argument does not apply to this case.
\end{remark}

The case of generalized Kummer type is similar:
\begin{proposition}\label{prop: lifting Kummer to Bridgeland over abelian}
     Let $X$ be as in the case $(ii)$ of  \Cref{Good settings}. There are the following data:
\begin{itemize}
    \item a finite flat extension $V$ of $W$ and a smooth projective family $\cX\to \Spec V$ whose special fiber $\cX_0\cong X$,
    \item an abelian surface $S$ over $\bar{K}$,
\end{itemize}
such that the generic fiber $\cX_{\bar{K}}$ is birational to the generalized Kummer variety $K_n(S)$.
\end{proposition}
\begin{proof}
    By \Cref{prop: Kum^n small Artin invariant}, we can always find a $w\in \NS(X)$ such that $w^2=-(2n+2)$ and $w$ generates a direct summand of $\NS(X)$. Write $w$ in the form $w=c_1(L_1\otimes L_2^\vee)$ with $L_1$ and $L_2$ ample. By \Cref{prop: lifting with line bundles}, we can find a finite flat extension $V$ of $W$ and a smooth projective family $\cX\to \Spec V$ whose special fiber $\cX_0\cong X$, and the generic fiber $\cX_{\bar{K}}$ carries a lifting of $L_1$ and $L_2$. In particular, $w$ lifts to $\cX_{\bar{K}}$, whose image in $\NS(\cX_{\bar{K}})$ is denoted by $\tilde{w}$.
    
    By the Tate conjecture for divisors \Cref{thm: Tate conjecture for divisors}, we know $\NS(X)\otimes\ZZ_\ell=\rH^2_\et(\cX_{\bar{K}},\ZZ_\ell)$ for any $\ell\neq p$. Together with the Artin comparison theorem, the divisibility of $\tilde{w}$ in $\rH^2(\cX_{\bar{K}},\ZZ)$ is of the form $p^k (2n+2)$ with $k\in\ZZ$. Furthermore, $k=0$ because the absolute value of the determinant of $\rH^2(\cX_{\bar{K}},\ZZ)=2n+2$ is prime to $p$.

    Now, $\tilde{w}\in \rH^2(\cX_{\bar{K}},\ZZ)$ is of square $-(2n+2)$ and divisibility $2n+2$, so $\tilde{w}^\perp\subset \rH^2(\cX_{\bar{K}},\ZZ)$ is unimodular of signature $(3,3)$, which is then isomorphic to $\rU^3$. Thus, we can write $\rH^2(\cX_{\bar{K}},\ZZ)$ in the form $\rU^3\oplus\langle\tilde{w}\rangle$. Then we complete the proof by \Cref{Cor: untwisted moduli for the Kummer type}.
\end{proof}

\begin{remark}
    In fact, any supersingular excellent irreducible symplectic variety admits a (twisted) modular lifting, by the same methods as before. But we can only deal with the untwisted case for the main theorem. This is why we exclude the twisted case in \Cref{prop: lifting Kummer to Bridgeland over abelian}.
\end{remark}

\subsection{Good reduction of the underlying surface}
So far, we have already lifted $X$ to a (twisted) modular irreducible symplectic variety on a K3 or abelian surface $S$ over a field of characteristic zero. In order to go back to positive characteristic, we need the underlying surface $S$ to have good reduction.

Let $\cO_F$ be a Henselian discrete valuation ring, with fraction field  $F$ and residue field $k$. Let $Y$ be a smooth projective variety over $F$. A smooth proper model of $Y$ is a smooth proper algebraic space
\begin{equation}
    \cY\to\spec\cO_F
\end{equation}
such that the generic fiber is $Y$ and the special fiber is a smooth projective variety. If such model exists, we say that $Y$ has good reduction. Moreover, we say $Y$ has potentially good reduction if after base change to a finite extension $F'$, $Y_{F'}$ has good reduction.

If $Y$ has good reduction, it is easy to see that the Galois representation $\rH^r_\et(Y_{\bar{F}},\QQ_\ell)$ is \emph{unramified}, i.e., the inertia group $I$ acts on it trivially.  Note that in our case $k$ is algebraically closed, so the inertia group is the whole absolute Galois group $G_F$. Thus, the ``unramified" in fact means the Galois representation is trivial. Conversely, to deduce good reduction from unramifiedness for K3 surfaces or abelian surfaces, we have the following two theorems.

\begin{theorem}\label{Thorem: good reduction of K3}
    Let $F$ be a finite extension of $K=\Frac(W)$, and let $S$ be a K3 surface over $F$. Assume that there is a line bundle $L$ on $S$ such that $\mathrm{gcd}(L^2,p)=1$ and the Galois representation $\rH_{\et}^2(S_{\bar{F}},\QQ_\ell)$ is unramified, equivalently trivial, for some prime $\ell\neq p$. Then $S$ has potentially good reduction.
\end{theorem}
\begin{proof}
    This is the theorem \cite[Theorem 1.8]{Bragg_2023}, where the Hecke orbit conjecture is guaranteed by \cite[Remark 1.2.2]{DH2022_Hecke_orbit}.
\end{proof}

\begin{theorem}\label{Thorem: good reduction of abelian}
    Let $F$ be a finite extension of $K=\Frac(W)$, and let $S$ be an abelian surface over $F$. If the Galois representation $\rH_{\et}^2(S_{\bar{F}},\QQ_\ell)$ is unramified, equivalently trivial, for some prime $\ell\neq p$, then $S$ has potentially good reduction.
\end{theorem}
\begin{proof}
    Let $V=\rH^1_\et(S_{\bar{F}},\QQ_\ell)$ together with the Galois action $\rho:G_F\to\mathrm{GL}(V)$. We know that $\rH^2_\et(S_{\bar{F}},\QQ_\ell)=\bigwedge^2 V$ and the Galois action is given by $\bigwedge^2 \rho$, which is trivial by the assumption. Since $\bigwedge^2\rho=\id$, $\rho$ preserves every plane in $V$. Since $\dim V=4$, every line in $V$ is the intersection of two distinct planes in $V$. So $\rho$ preserves every line in $V$. This means $\rho(g) $ is a scalar for every $g\in G_F$. Again by $\bigwedge^2\rho=\id$, we must have $\rho(g)=\pm 1$. We will write
    \[
    \rho(g)=\chi(g)\cdot\id,
    \]
    where $\chi$ is a character $\chi:G_F\to\{\pm1\}$.

    If $\chi$ is trivial, then $\rho$ is trivial. By \cite{ST68}, $S$ has good reduction. 
    
    If $\chi$ is nontrivial, then $\ker(\chi)$ is a normal subgroup of $G_F$ of index $2$. Let $F'$ be the quadratic subextension of $\bar{F}/F$ fixed by $\ker(\chi)$. The absolute Galois group $G_{F'}=\ker(\chi)$ acts trivially on $\rH^1_\et(S_{\bar{F}},\QQ_\ell)$. Again, by \cite{ST68}, $S_{F'}$ has good reduction.
\end{proof}

For an irreducible symplectic variety, instead of considering the whole Galois representation on $\rH^2(X,\QQ_\ell)$, we can consider its transcendental part given as follows.

\begin{definition}\label{Definition: transcdental part}
    Let $F$ be a field of characteristic zero and $X$ be an irreducible symplectic variety over $F$. For  any prime $\ell$, the algebraic part $\rH_\mathrm{alg}^2(X,\QQ_\ell(1))$ is defined to be the image of $\NS(X_{\bar{F}})\otimes\QQ_\ell$ in $\rH_\et^2(X_{\bar{F}},\QQ_\ell(1))$, and the transcendental part $\rH_\mathrm{tr}^2(X,\QQ_\ell(1))$ is defined to be the orthogonal complement of $\rH_\mathrm{alg}^2(X,\QQ_\ell(1))$ under the $\ell$-adic Beauville--Bogomolov form.
\end{definition}

Both $\rH_\mathrm{alg}^2(X,\QQ_\ell(1))$ and $\rH_\mathrm{tr}^2(X,\QQ_\ell(1))$ are stable under the action of the absolute Galois group, hence they are moreover Galois sub-representations. 

\begin{proposition}\label{Proposition: isomorphism between transcdental}
    Let $\cX$ and $S$ be in either \Cref{prop: lifting to Bridgeland over K3} or \Cref{prop: lifting Kummer to Bridgeland over abelian}. Then for any prime $\ell$, after replacing $F$ by a finite extension, there is a Galois equivariant isomorphism between $\rH_\mathrm{tr}^2(\cX_{\bar{K}},\QQ_\ell(1))$ and $\rH_\mathrm{tr}^2(S,\QQ_\ell(1))$ induced by an algebraic cycle defined over $F$.
\end{proposition}
\begin{proof}
    Consider the isometry $\theta$ in \Cref{theorem:Bridgeland of K3^n type and Kum^n type} and \Cref{theorem: OG type example} induced by the quasi-universal sheaf. After replacing $F$ by a finite extension, we can assume that the Mukai vector and the quasi-universal sheaf are both defined over $F$. Moreover, for the OG10-type case, we can replace $F$ by a further finite extension to ensure that the class of the exceptional divisor is also defined over $F$. Then $\theta$ gives the isomorphism between transcendental parts that we need.
\end{proof}
\begin{proposition}\label{prop: good reduction of the underlying surface}
    Let $S$ be the K3 or abelian surface in \Cref{prop: lifting to Bridgeland over K3} or \Cref{prop: lifting Kummer to Bridgeland over abelian}. Then $S$ has good reduction, after replacing $F$ by a finite extension. 
\end{proposition}
\begin{proof}
    Let $\ell\neq p$ be any prime. We replace $F$ by a large enough finite extension. Then we can assume that the Galois representations $\rH_\mathrm{alg}^2(\cX_{\bar{K}},\QQ_\ell(1))$ and $\rH_\mathrm{alg}^2(S,\QQ_\ell(1))$ are trivial, and there is a Galois equivariant isomorphism between $\rH_\mathrm{tr}^2(\cX_{\bar{K}},\QQ_\ell(1))$ and $\rH_\mathrm{tr}^2(S,\QQ_\ell(1))$. Since $X$ already has good reduction, the Galois representation $\rH_\mathrm{\et}^2(\cX_{\bar{K}},\QQ_\ell(1))$ is unramified. Thus, $\rH_\mathrm{\et}^2(S,\QQ_\ell(1))$ is also unramified. The proof is finished by \Cref{Thorem: good reduction of K3} and \Cref{Thorem: good reduction of abelian}.
\end{proof}

\subsection{Proof of the supersingular Tate conjecture}

We can now prove the supersingular Tate conjecture for $X$ satisfying \Cref{Good settings}. We first summarize the situation. 

We have found in \Cref{prop: lifting to Bridgeland over K3} and \Cref{prop: lifting Kummer to Bridgeland over abelian} a finite flat extension $V$ of $W$ with fraction field $F$, such that $X$ lifts to a family $\cX\to\Spec V$, whose geometric generic fiber $\cX_{\bar{K}}$ is birational to a (twisted) modular irreducible symplectic variety over a K3 or abelian surface $S$. After replacing $V$ by a further finite flat extension (then replace $F$ by a finite extension), we have that $S$ is defined over $F$ and has good reduction by \Cref{prop: good reduction of the underlying surface}. Let $\cS\to\Spec V$ be the algebraic space corresponding to a good reduction of $S$ with special fiber $\cS_0$ over $k$ being a K3 surface or Abelian surface. We will show that the special fiber is a supersingular K3 or abelian surface.

\begin{proposition}\label{prop: reduction is supersingular}
    The surface $\cS_0$ is supersingular.
\end{proposition}
\begin{proof}
    We prove this for $X$ in case $(i)$ of \Cref{Good settings}, while the case $(ii)$ and $(iii)$ are similar. We replace $F$ by a finite extension to make the Galois representation $\rH_\mathrm{alg}^2(\cX_{\bar{K}},\QQ_p(1))$ and $\rH_\mathrm{alg}^2(S,\QQ_p(1))$ trivial. By \Cref{Proposition: isomorphism between transcdental}, we obtain an isomorphism of $p$-adic representations:
    \[
    \rH_\mathrm{tr}^2(\cX_{\bar{K}},\QQ_p(1))\cong \rH_\mathrm{tr}^2(S,\QQ_p(1))
    \]
    By \cite[Corollary 3.9]{GKP2022}, we have $\DD_{\cris}(\rH^2_\et(\cX_{\bar{K}},\QQ_p(1)))=\rH^2_{\cris}(X/K)(1)$ and $\DD_{\cris}(\rH^2_\et(\cS_{\bar{K}},\QQ_p(1)))=\rH^2_{\cris}(\cS_0/K)(1)$. Since the former has all slopes zero, so does the latter. Therefore, $\cS_0$ is supersingular.    
\end{proof}

Finally, we have completed all the preparatory work to prove the case (i), (ii), and (iii) of the \Cref{thm:intro:SSMotive}.
\begin{theorem}\label{theorem:main theorem for K3^n, Kum^n, and OG10}
   If $X$ is as in case (i) or (iii) in \Cref{Good settings}, then the Chow motive of $X$ is of Tate type; If $X$ is as in case (ii) in \Cref{Good settings}, then the Chow motive of $X$ is of supersingular abelian type.
\end{theorem}
\begin{proof}
    Keep the notation as above.   
    By \Cref{thm: motive of Bridgelan moduli space}, we have an embedding of motives
    \[
    \frh(\cX_{\bar{K}})\hookrightarrow \bigoplus_{i=1}^m\frh(S_{\bar{K}}^{s_i})(t_i).
    \]
    Again, after replacing $V$ by a further finite flat extension, we can assume this embedding is defined over $F$. So, by \Cref{lem: specialization of motives}, we have an embedding of motives in special fiber:
    \[
    \frh(X)\hookrightarrow \bigoplus_{i=1}^m\frh(\cS_0^{s_i})(t_i).
    \]
    Note that $\cS_0$ is a supersingular K3 surface in case (i) or (iii), and it is a supersingular abelian surface in case (ii). Applying \Cref{thm:SSK3TateMotive} completes the proof.
\end{proof}

\section{Supersingular Tate conjecture for OG6-type}\label{sec:og6}

In this section, we prove \Cref{thm:intro:SSMotive} for a supersingular excellent irreducible symplectic variety $X$ of OG6-type such that the Artin invariant $\sigma_0\neq 4$. Unlike the other types proved in the previous section, it is hard to lift $X$ to a (twisted) modular irreducible symplectic variety in characteristic zero, because \Cref{prop: lifting with line bundles} only guarantees a lifting of $X$ with two line bundles. Instead, we use the fact that $X$ admits a rational double cover by a $K3^{[3]}$-type irreducible symplectic variety, which was first constructed by Mongardi, Rapagnetta, and Sacc\`a in \cite{MRG2018}, and later generalized by Floccari and Fu in \cite{FF2026}.
\subsection{Singular OG6-variety and resolution}
We first recall some basic definitions and properties of singular OG6-varieties and their resolution over $\CC$. Our main reference is \cite{FF2026}.

\begin{definition}
    A \textit{singular OG6-variety} is a compact K\"{a}hler complex analytic space which is a locally trivial deformation of O'Grady's $\cK_A(2,0,-2)$. An \textit{OG6-resolution} is a compact hyper-K\"ahler manifold isomorphic to the crepant resolution $\widetilde{\cK}$ of a singular OG6-variety $\cK$ obtained by blowing up its singular locus.
\end{definition}

We can give a numerical criterion to determine whether a hyper-K\"ahler manifold $X$ of OG6-type is birational to an OG6-resolution:
\begin{proposition}\label{prop:numerical criterion of OG6-resoluion}
    Let $X$ be a hyper-K\"ahler manifold of OG6-type. Then $X$ is birational to an OG6-resolution $\widetilde{\cK}$ if and only if $\rH^{1,1}(X,\ZZ)$ contains an element $\sigma$ with square $-2$ and divisibility $2$ in $\rH^2(X,\ZZ)$.
\end{proposition}
\begin{proof}
    If $\rH^{1,1}(X,\ZZ)$ contains an element $\sigma$ with square $-2$ and divisibility $2$, we can assume that 
    \[
    \rH^2(X,\ZZ)=\rU^{\oplus3}\oplus\langle-2\rangle\oplus\langle\sigma\rangle.
    \]
    By the surjectivity of the period map, the Hodge structure on $\sigma^\perp$ can be realized as the Hodge structure of a singular OG6-variety $\cK$. By the Torelli theorem for hyper-K\"ahler manifolds, the resolution $\widetilde{\cK}$ is birational to $X$. The converse follows from \cite[Remark 3.2]{FF2026}.
\end{proof}

The singular locus $\Sigma$ of a singular OG6-variety $\cK$ is given by a quotient of an abelian fourfold of Weil type, which is closely related to the Kuga--Satake construction:
\begin{theorem}\label{theorem: Abelian fourfold and Kuga-Satake}
    Let $(\cK,h)$ be a polarized singular OG6-variety and let $\widetilde{\cK}$ be its resolution. There is an abelian fourfold $B_\cK$ such that:
    \begin{enumeratea}
        \item The singular locus $\Sigma$ of $\cK$ is isomorphic to $B_\cK/\pm 1$.
        \item The Kuga--Satake abelian variety $\rK\rS(\rH^2(\widetilde{\cK},\QQ)_{\mathrm{prim}})$ of $\widetilde{\cK}$ is isogenous to a self-product $B_\cK^m$ of $B_\cK$.
    \end{enumeratea}
\end{theorem}
\begin{proof}
    There is a Hodge isometry
    \[
    \rH^2_{\mathrm{tr}}(\cK,\QQ)\xrightarrow{\cong}\rH^2_{\mathrm{tr}}(\widetilde{\cK},\QQ).
    \]
    Thus, the Kuga--Satake abelian varieties of $\cK$ and $\widetilde{\cK}$ share the same isogeny factors by the same argument in \cite[Remark 4.2]{FF2026}. The assertion follows from \cite[Theorem 3.4]{FF2026}.
\end{proof}

 Now, we can turn to the double cover map. For a singular OG6-variety $\cK$, denote by $\Sigma$ the singular locus of $\cK$ and by $\Omega$ the singular locus of $\Sigma$. Then $\Sigma\cong B_\cK/\pm 1$ as before and $\Omega$ consists of 256 points of $\Sigma$.
\begin{theorem}\cite[Theorem 4.1, Corollary 4.3]{FF2026}\label{theorem: double cover}
    Let $\cK$ be a projective singular OG6-variety. There is a projective variety $Z_\cK$ of $K3^{[3]}$-type and a degree $2$ generically finite morphism $\phi:Z_\cK\to\cK$ such that $\phi|_{\phi^{-1}(\cK\backslash \Sigma)}$ is an \'etale double cover, $\Delta=\phi^{-1}(\Sigma)$ is isomorphic to $\mathrm{Bl}_\Omega\Sigma$, and $\phi|_\Delta$ is identified with the blow-up map. Moreover, $\phi$ induces a Hodge isometry
    \[
    \phi_*:\rH^2_{\mathrm{tr}}(Z_\cK,\QQ)\to\rH^2_{\mathrm{tr}}(\cK,\QQ)[2].
    \]
    where the ``[2]'' means multiplied by 2. As a result,  $Z_\cK$ is birational to a moduli space of stable sheaves on a K3 surface $S_{\cK}$.
\end{theorem}

Therefore, for the resolution $\widetilde{\cK}$, we only obtain a rational double cover $\widetilde{\phi}:Z_\cK\dashrightarrow \widetilde{\cK}$. We need to resolve this rational map to obtain a regular morphism. This was done in the second paragraph of the proof of \cite[Theorem 5.7]{FF2026}. We exhibit it in the following diagram.
\[\begin{tikzcd}
	{\widetilde{Y}=\mathrm{Bl}_{\Delta'}Y} & \\
	{Y=\mathrm{Bl}_{\Gamma}Z_\cK} \\
	{Z_\cK} & \widetilde{\cK}
	\arrow[from=1-1, to=2-1]
	\arrow["\psi", from=1-1, to=3-2]
	\arrow[from=2-1, to=3-1]
	\arrow["{\widetilde{\phi}}", dashed, from=3-1, to=3-2]
\end{tikzcd}\]
We use the following notation in the diagram:
 \begin{itemize}
     \item $\Gamma\subset\Delta\cong \mathrm{Bl}_{B_\cK[2]}(B_\cK/\pm 1)$ is the union of 256 copies of $\PP^3$ arising as the exceptional divisor of the blow up map $\Delta\to B_\cK/\pm 1$.
     \item $\Delta'$ is the strict transform of $\Delta$ in $Y$, which is again isomorphic to $\mathrm{Bl}_{B_\cK[2]}(B_\cK/\pm 1)$.
     \item $\psi$ is a regular morphism extending $\widetilde\phi$, which is in particular a degree 2 generically finite morphism.
 \end{itemize}

With the help of the above diagram, we can show that the Chow motive of $\widetilde{\cK}$ can be controlled by the Chow motives of $B_\cK$ and $S_\cK$.
\begin{proposition}\label{prop: Chow motives of OG6}
    $\widetilde{\cK}$ is Chow-motivated by $S_\cK$ and $B_\cK$. More specifically, there is an embedding of Chow motives:
    \[
    \frh(\widetilde{\cK})\hookrightarrow \bigoplus_{i=1}^n\frh(S_\cK^{m_i})(t_i)\oplus \frh(B_\cK)(-1) \oplus \bigoplus_{j=1}^r \QQ(-k_j).
    \]
\end{proposition}
\begin{proof}
    Since $\psi:\widetilde{Y}\to\widetilde{\cK}$ is a generically finite morphism of degree 2, the motive $\frh(\widetilde{\cK})$ is a direct summand of $\frh(\widetilde{Y})$. By the blow-up formula, we have
    \[
    \frh(\widetilde{Y})\cong \frh(Y)\oplus \frh(\Delta')(-1).
    \]
    Similarly, 
    \[
    \frh(Y)\cong\frh(Z_\cK)\oplus \frh(\Gamma)(-1)\oplus \frh(\Gamma)(-2).
    \]
    Thus, we have an embedding of motives:
    \[
    \frh(\widetilde{\cK})\hookrightarrow  \frh(Z_\cK)\oplus \frh(\Gamma)(-1)\oplus \frh(\Gamma)(-2)\oplus \frh(\Delta')(-1).
    \]
    We can compute each term on the right. First, since $Z_\cK$ is birational to a moduli space of stable sheaves on the K3 surface $S_\cK$, by \cite{Bu20} there is an embedding:
    \[
    \frh(Z_\cK)\hookrightarrow \bigoplus_{i=1}^n\frh(S_\cK^{m_i})(t_i).
    \]
    Since $\Gamma$ is the union of 256 copies of $\PP^3$, we have
    \[
    \frh(\Gamma)= \big(\bigoplus_{i=0}^3 \QQ(-i)\big)^{\oplus 256}.
    \]
    For $\frh(\Delta')$, we note that $\Delta'=\mathrm{Bl}_{B_\cK[2]}(B_\cK/\pm 1)$ can be realized as $\widetilde{B_\cK}/\pm 1$, where $\widetilde{B_\cK}=\mathrm{Bl}_{B_\cK[2]}B_\cK$. Thus, we have an embedding:
    \[
    \frh(\Delta')\hookrightarrow \frh(\widetilde{B_\cK})=\frh(B_\cK)\oplus \big(\bigoplus_{i=0}^3 \QQ(-i)\big)^{\oplus 256}.
    \]
    Putting everything together, we obtain an embedding:
    \[
    \frh(\widetilde{\cK})\hookrightarrow \bigoplus_{i=1}^n\frh(S_\cK^{m_i})(t_i)\oplus \frh(B_\cK)(-1) \oplus \bigoplus_{j=1}^r \QQ(-k_j).
    \]
\end{proof}

\subsection{The supersingular Tate conjecture}
Now, we can prove the supersingular Tate conjecture for an OG6-type variety. We first show that it can be lifted to an OG6-resolution in characteristic zero.

\begin{proposition}\label{prop : lifting OG6 to OG6-resolution}
    Let $X$ be a supersingular OG6-type variety over $k$ with Artin invariant $\sigma_0\leq 3$. Then there is a finite flat extension $V$ of $W$ with fraction field $F$ and a smooth projective family $\cX\to \Spec V$ such that the special fiber $\cX_0\cong X$ and the geometric generic fiber $\cX_{\bar{K}}$ is birational to an OG6-resolution $\widetilde{K}$. Moreover, we can assume that $\cX/V$ carries a polarization $\boldsymbol{\xi}$ of degree prime to $p$.
\end{proposition}
\begin{proof}
    By \Cref{prop: OG6 small Artin invariant} and \Cref{prop: OG6 another description}, we can always find a $w\in\NS(X)$ such that $w^2=-2$ and the rank one sublattice $\langle w\rangle$ is a direct summand of $\NS(X)$. We can also find an ample line bundle $L$ on $X$ of degree prime to $p$. Let $N\subset \NS(X)$ be the sublattice generated by $w$ and $L$. We can lift $N $ to characteristic zero by \Cref{prop: lifting with line bundles}. Then, we apply the numerical criterion \Cref{prop:numerical criterion of OG6-resoluion} to the geometric generic fiber. This completes the proof.
\end{proof}

Let $B_\cK$ and $S_\cK$ be the abelian fourfold and K3 surface appearing in \Cref{theorem: Abelian fourfold and Kuga-Satake} and \Cref{theorem: double cover}. We may assume that they are both defined over $F$ by replacing $V$ by a further extension. We then show that they both have supersingular good reduction, possibly after replacing $V$ by a further extension again.

\begin{proposition}\label{prop: sueringular reduction of S_K}
    After replacing $F$ by a finite extension, the K3 surface $S_\cK$ has supersingular good reduction.
\end{proposition}
\begin{proof}
    By \cite[Lemma 3.4]{Motive_OG6} and \cite[remark 4.2]{FF2026}, we have isometries induced by algebraic correspondences for any prime $\ell$:
    \[
    \rH^2_{\mathrm{tr}}(S_\cK,\QQ_\ell(1))\cong\rH^2_{\mathrm{tr}}(Z_\cK,\QQ_\ell(1))\cong \rH^2_{\mathrm{tr}}(\widetilde{\cK},\QQ_\ell(1))[2].
    \]
    Here, ``[2]'' means multiplication by 2. Replace $F$ by a finite extension if necessary, we can assume the above isometries are Galois equivariant.

    We claim that for any odd prime $p$, $S_\cK$ carries a line bundle of degree prime to $p$. We can see this by contradiction. If all line bundles on $S$ have gegree divisible by $p$, then the discriminant group of $\NS(S_\cK)\otimes\ZZ_p$ has the form \(\bigoplus_{i=1}^\rho \ZZ/p^{m_i}\ZZ\), where $m_i$'s are positive integers and $\rho$ is the Picard number. Note that $\rho\geq 12$. Since $\NS(S_\cK)$ is a primitive sublattice of the unimodular lattice $\rH^2(S_\cK,\ZZ)$, then the orthogonal complement $\NS(S_\cK)^\perp\otimes\ZZ_p$ also has the discriminant group \(\bigoplus_{i=1}^\rho \ZZ/p^{m_i}\ZZ\). This is a contradiction since $\rank(\NS(S_\cK)^\perp)=22-\rho<\rho$.
    
    Then we can show that $S_\cK$ has supersingular good reduction by the same argument in \Cref{prop: good reduction of the underlying surface,prop: reduction is supersingular}.
\end{proof}

\begin{proposition}\label{prop: reduction of B_K}
    The abelian fourfold $B_\cK$ has supersingular reduction.
\end{proposition}
\begin{proof}
    We know that the Kuga--Satake abelian variety $A$ of $(\widetilde{\cK},\boldsymbol{\xi}_F)$ is isogenous to a power of $B_\cK$ by \Cref{theorem: Abelian fourfold and Kuga-Satake}. So we only need to show that $A$ has supersingular good reduction. 

    Consider the family $\cX\to\Spec V$ in \Cref{prop : lifting OG6 to OG6-resolution}. Since the construction of Kuga--Satake extends to mixed characteristic, the abelian variety $A$ extends to an abelian scheme $\cA/V$. We need to show that the special fiber $\cA_0$ of $\cA$ is supersingular. This is inspired by \cite[Proposition 21]{Ch13}.

    Let $\rP_\cris=\boldsymbol{\xi}_0^\perp\subset\rH^2_\cris(X/K)(1)$ and $C(\rP_\cris)$ be its Clifford algebra. Let $C^+(\rP_\cris)$ be its even part. The Kuga--Satake abelian variety constructed by $C(\rP_\cris)$ is isogenous to the square of the one constructed by $C^+(\rP_\cris)$. So we only consider the latter. In this case, $C^+(\rP_\cris)$ is a central simple algebra. Denote $\rH^1_\cris(\cA_0/K)$ by $\rH$. We have an isomorphism of $K$-vector spaces
    \[
    C^+(\rP_\cris)\cong \rH.
    \]
    However, this isomorphism is not compatible with the $F$-structure. Viewing $\rH$ as a right $C^+(\rP_\cris )$-module, we have an isomorphism of $F$-isocrystals by taking left multiplication
    \[
    C^+(\rP_\cris)\cong \End_{C^+}(\rH).
    \]
   Consider the natural map 
   \begin{align*}
       C^+(\rP_\cris)\otimes C^+(\rP_\cris)^{\mathrm{op}}&\to \End_K(\rH)\\
       a\otimes b&\mapsto L_a\circ R_b
   \end{align*}
   Here, $L_a$ is the left multiplication by $a$ and $R_b$ is the right multiplication. By \cite[\href{https://stacks.math.columbia.edu/tag/0748}{Tag 0748}]{stacks-project}, this map is an isomorphism. It is direct to check by definition that this isomorphism is compatible with the $F$-structure. Since $C(\rP_\cris)$ is isomorphic to the exterior algebra of $\rP_\cris$ as $K$-vector spaces (not as algebras), the left-hand side has all slopes zero, hence so does the right-hand side. But the right hand side is isomorphic to $\End_K(\rH)\cong \rH\otimes \rH^\vee$, whose slopes are the pairwise differences of slopes of $\rH$. So all the slopes of $\rH$ are the same. This means $\cA_0$ is supersingular.
\end{proof}

We can prove the main result of this section.

\begin{theorem}\label{theorem: Tat conj for OG6}
    Let $X$ be a supersingular excellent OG6-type variety with Artin invariant $\sigma_0\leq 3$, then the Chow motive of $X$ is of Tate type.
\end{theorem}
\begin{proof}
    Keep the notation as above. By \Cref{prop: Chow motives of OG6}, we have an embedding of Chow motives
    \[
    \frh(\cX_{\bar{F}})\hookrightarrow \bigoplus_{i=1}^n\frh(S_\cK^{m_i})(t_i)\oplus \frh(B_\cK)(-1) \oplus \bigoplus_{j=1}^r \QQ(-k_j).
    \]
    Since $S_\cK$ and $B_\cK$ both have supersingular good reduction, denote their reduction by $S_0$ and $B_0$. Then by \Cref{lem: specialization of motives}, we have an embedding of Chow motives
    \[
    \frh(X)\hookrightarrow \bigoplus_{i=1}^n\frh(S_0^{m_i})(t_i)\oplus \frh(B_0)(-1) \oplus  \bigoplus_{j=1}^r \QQ(-k_j).
    \]
    This shows that $\frh(X)$ is of supersingular abelian type. Since all odd degree cohomology groups of $X$ vanish, $\frh(X)$ is of Tate type by \cite[Corollary~2.12]{SSIV}.
\end{proof}

\appendix
\section{Torsion in cohomology of irreducible symplectic varieties}
\label{appendix}

In this appendix, we consider the torsion cohomology of irreducible symplectic varieties $X$ of known type.  This question was raised by Huybrechts at the 2025 workshop “K3 surfaces \& friends: Brauer groups and moduli” in Netherlands. Some known relevant results  include:
\begin{itemize}
    \item Markman \cite{Ma07} showed that $\rH^*(S^{[n]}, \ZZ)$ is torsion free for $S$ a Poisson surface (e.g., K3 or abelian).
    \item Totaro \cite{Totaro20} proved that $\rH^*(S^{[n]}, \ZZ)$ is torsion free provided $\rH^*(S, \ZZ)$ is torsion free.
    \item Kapfer and Menet \cite{KM2018} showed that $\rH^*(K_2(A), \ZZ)$ is torsion free.
\end{itemize}
In this appendix, we will partially answer this question, showing that there is no $p$-torsion for  large $p$.

\subsection{Generalized Kummer type}
Let $X$ be the generalized Kummer variety $K_n(A)$. For $n > 2$, the torsion freeness of $\rH^*(X, \ZZ)$ remains open. Instead, we prove the torsion freeness of $\rH^*(X, \ZZ[\tfrac{1}{(n+1)!}])$. The main result is the following:

\begin{theorem}\label{thm:torsion free}
    Let $A$ be an abelian surface, and let $X= K_{n}(A)$ be the $2n$-dimensional generalized Kummer variety. Then
    $\rH^*(X, \ZZ)$ has no $p$-torsion for any prime $p > n+1$.
\end{theorem}

The main ingredient is the decomposition theorem for torsion cohomology:

\begin{theorem}[Juteau--Mautner--Williamson {\cite[Thm. 3.7]{JMW14}}]\label{thm:JMW}
    Let \( f: \widetilde{Y} \to Y \) be a stratified, proper,
surjective, and semi-small morphism from a smooth variety \( \widetilde{Y} \) to a stratified variety \( Y = \bigsqcup_{\lambda} Y_\lambda \). Suppose that all strata are relevant. For each stratum \( Y_\lambda \), let \( F_\lambda \) be the fiber over a point in \( Y_\lambda \), and let \( N_\lambda \) be a normal slice to \( Y_\lambda \) in \( Y \). Suppose that for all relevant \( \lambda \), the intersection form
    \begin{equation}\label{eq:intersect}
         \rH^{\mathrm{BM}}_{2d_\lambda}(F_\lambda) \times \rH^{\mathrm{BM}}_{2d_\lambda}(F_\lambda) \to \ZZ
    \end{equation}
    is non-degenerate modulo \( p \). Then the decomposition theorem holds for \( f \) with coefficients in \( \ZZ/p\ZZ \).
\end{theorem}

\subsubsection*{Proof of \Cref{thm:torsion free}}

 Consider the Hilbert-Chow morphism
    \[
    \pi\colon X \to Z, \quad \text{where } Z = A_0^{n+1} / S_{n+1},
    \]
    and $A_0^{n+1} := \ker(A^{n+1} \to A)$ (see \cite{Motive_Kummer}). There is a stratification $Z = \bigsqcup_{\lambda \dashv ({n+1})} Z_\lambda$ with 
    \begin{equation}\label{eq:stratification}
          Z_\lambda = \left\{ \sum_{i=1}^k \lambda_i [x_i] \,\middle|\, x_i \text{ distinct}, \sum_{i=1}^k \lambda_i x_i = 0 \right\},
    \end{equation}
    where $\lambda = (\lambda_1, \dots, \lambda_k)$ is a partition of $n+1$. Note that $\pi$ is semi-small with all strata being relevant.

By \Cref{thm:JMW} and \Cref{prop:intersection-form}, we have the following decomposition 
 \begin{equation}\label{eq:decomposition}
            \bR\pi_*(\ZZ/p\ZZ[2n]) \cong \bigoplus_{\lambda \dashv n+1} i_{\lambda*} \rI\rC_{\overline{Z_\lambda}}(\ZZ/p\ZZ)
            \quad \text{in } D_c^b(Z, \ZZ/p\ZZ),
        \end{equation}
 where  $ \rI\rC_{\overline{Z_\lambda}}(\ZZ/p\ZZ)$ is the IC-complex associated with the trivial local system on $\overline{Z_\lambda}$ (shifted by dimension).
  Moreover, by the classical decomposition theorem for $\mathbb{Q}$-coefficients, we also have
\begin{equation}\label{eq:decomposition-Q}
    \bR\pi_*(\mathbb{Q}[2n]) \cong \bigoplus_{\lambda \dashv n+1} i_{\lambda*} \rI\rC_{\overline{Z_\lambda}}(\mathbb{Q})
    \quad \text{in } D_c^b(Z, \mathbb{Q}).
\end{equation}

Combining the two decompositions—\eqref{eq:decomposition} and \eqref{eq:decomposition-Q}—we obtain isomorphisms:
\[
\rH^*(X, \ZZ/p\ZZ) \cong \bigoplus_{\lambda \dashv n+1} \rI\rH^*(\overline{Z_\lambda}, \ZZ/p\ZZ),
\]
and
\[
\rH^*(X, \mathbb{Q}) \cong \bigoplus_{\lambda \dashv n+1} \rI\rH^*(\overline{Z_\lambda}, \mathbb{Q}).
\]
Next, it is known that the normalization of $\overline{Z_\lambda}$ satisfies
\[
\overline{Z_\lambda}^\nu \cong A_0^{(\lambda)} := A_0^\lambda / S_\lambda \quad \text{(see \cite{Motive_Kummer})},
\]
where
\[
A_0^\lambda = \left\{ (x_1, \dots, x_{|\lambda|}) \in A^\lambda \,\middle|\, \sum_{i=1}^k \lambda_i x_i = 0 \right\}.
\]
 Here, \( A_0^\lambda \) is non-canonically isomorphic to a disjoint union of abelian varieties of the form \( \prod_{k} A^{r_k} \), where \( \sum r_k = \ell(\lambda)-1 \) and each \( r_k \leq n+1 \). Thus, the integral cohomology of $A_0^\lambda$ is always torsion free. Since \( p > n+1 \), taking the invariant part under $S_\lambda$, we obtain
    \[
    \dim_{\ZZ/p\ZZ} \rI\rH^*(\overline{Z_\lambda}, \ZZ/p\ZZ) = \dim_{\QQ} \rI\rH^*(\overline{Z_\lambda}, \QQ).
    \]

    Therefore,
    \[
    \dim_{\ZZ/p\ZZ} \rH^*(X, \ZZ/p\ZZ) = \dim_{\QQ} \rH^*(X, \QQ) = b_i(X) \quad \text{for all } p > n+1.
    \]
    This implies that \( \rH^*(X, \ZZ) \) has no \( p \)-torsion for any prime \( p > n+1 \). \qed

\begin{proposition}\label{prop:intersection-form}
With the stratification \eqref{eq:stratification} for the morphism $\pi: X \to Z$, the intersection form \eqref{eq:intersect} is given by 
       \[
    (F_\lambda \cdot F_\lambda)_{\tilde{N_\lambda}}= (-1)^{n+1-k} \prod_{i=1}^k \lambda_i,
    \]
which is non-degenerate modulo $p$ if $p>n+1$. 
\end{proposition}  

\begin{proof}
\noindent \textbf{Step 1.} Local geometry of $Z_\lambda$ and its normal slice.
\vspace{.2cm}

Let $z \in Z_\lambda$ be a point, which can be written as $$z = \sum_{i=1}^k \lambda_i [x_i]$$ with $x_i$ distinct points in $A$ satisfying $\sum \lambda_i x_i = 0_A$. The fiber $F_\lambda := \pi^{-1}(z)$ is isomorphic to $\prod_{i=1}^k B_{\lambda_i}$, where $B_{\lambda_i}$ is the Brian\c{c}on variety, i.e., the punctual Hilbert scheme parameterizing ideals of colength $\lambda_i$ in the local ring $\mathcal{O}_{\CC^2, 0}$, which is irreducible of dimension $\lambda_i - 1$.

Let $d_\lambda = \dim F_\lambda = \sum_{i=1}^k (\lambda_i - 1) = n+1 - k$. A \emph{normal slice} $N_\lambda$ to the stratum $Z_\lambda$ at $z$ is a locally closed subvariety of $Z$ of dimension $$\codim_Z (Z_\lambda) = \dim Z - \dim Z_\lambda$$ that intersects $Z_\lambda$ transversally at the point $z$. Its preimage under $\pi$, denoted by $\widetilde{N}_\lambda = \pi^{-1}(N_\lambda)$, is a normal slice to the preimage of $Z_\lambda$ in $X$. 

There are small enough neighborhoods $U_i$ of $x_i$ in $A$, which are pairwise disjoint, such that there is a neighborhood of $z$ in $A^{(n+1)}$ that is isomorphic to $\prod_{i=1}^k U_i^{(\lambda_i)}$. The normal slice $N_\lambda$ is given as follows. Consider the morphism: 

\begin{align*}
    \phi:\prod_{i=1}^k U_i^{(\lambda_i)} &\to \prod_{i=1}^k A\\
    (\sum_{j=1}^{\lambda_1}[y_{1,j}], \cdots,\sum_{j=1}^{\lambda_k}[y_{k,j}])&\to (\sum_{j=1}^{\lambda_1}y_{1,j}, \cdots,\sum_{j=1}^{\lambda_k}y_{k,j})
\end{align*}
where $y_{i,j}\in U_i$ and the latter sum is taken in $A$. Then $N_\lambda:=\phi^{-1}(\phi(z))$ is defined to be the fiber of $\phi$ over $\phi(z)$. Notice that $N_\lambda$ and $\widetilde{N_\lambda}$ are isomorphic to open subsets of $\prod_{i=1}^k A_0^{(\lambda_i)}$ and $\prod_{i=1}^k K_{\lambda_i}(A)$ respectively, hence $N_\lambda$ is normal and $\widetilde{N_\lambda}$ is smooth.

\vspace{.2cm}
\noindent \textbf{Step 2.} Intersection forms on Borel-Moore homology.
\vspace{.2cm}

The chosen normal slice $N_\lambda$ has dimension $d_{N_\lambda} = \codim_Z Z_\lambda$. The fundamental class of the fiber $F_\lambda$ defines an element  $$[F_\lambda]\in \rH^{\mathrm{BM}}_{2d_\lambda}(F_\lambda) \cong \rH^{\mathrm{BM}}_{2d_\lambda}(\widetilde{N}_\lambda|_z).$$ Using the interpretation of Borel-Moore homology as the homology of the one-point compactification, or via the local duality isomorphism
\[
\rH^{\mathrm{BM}}_{2d_\lambda}(F_\lambda) \cong \rH^{2d_{N_\lambda} - 2d_\lambda}(\widetilde{N}_\lambda, \widetilde{N}_\lambda \setminus F_\lambda),
\]
the intersection form \eqref{eq:intersect} can be identified with the pairing given by the cup product and evaluation on the fundamental class of $\widetilde{N_\lambda}$:
\[
\rH^{\mathrm{BM}}_{2d_\lambda}(F_\lambda) \times \rH^{\mathrm{BM}}_{2d_\lambda}(F_\lambda) \to \ZZ, \quad (\alpha, \beta) \mapsto (\alpha \cdot \beta) \frown [\widetilde{N_\lambda}].
\]
When applied to the fundamental class $[F_\lambda] \in \rH^{\mathrm{BM}}_{2d_\lambda}(F_\lambda)$, this pairing computes the self-intersection number $(F_\lambda \cdot F_\lambda)_{\widetilde{N}_\lambda}$ of the fiber within the normal slice.

\vspace{.2cm}
\noindent \textbf{Step 3. } Computation of the intersection number.
\vspace{.2cm}

Now we calculate the intersection number $(F_\lambda \cdot F_\lambda)_{\widetilde{N}_\lambda}$. We will interpret this intersection as a classical intersection in $\prod_{i=1}^k (U_i)^{[\lambda_i]} $. By shrinking $N_\lambda$, we may assume $N_\lambda\cap Z_\lambda=z$. Using the product decomposition of $F_\lambda\cong\prod_{i=1}^k B_{\lambda_i}$, we have  $\widetilde{N}_\lambda\cap \prod_{i=1}^kM_{\lambda_i}=\prod_{i=1}^k B_{\lambda_i}=F_\lambda$, where $M_{\lambda_i}$ is a subvariety of $U_i^{[\lambda_i]}$ whose points correspond to subschemes supported at a single point. By the projection formula, we have
\[
    (F_\lambda \cdot F_\lambda)_{\widetilde{N}_\lambda}=(F_\lambda\cdot\prod_{i=1}^kM_{\lambda_i})_{\prod_{i=1}^k U_i^{[\lambda_i]}}=\prod_{i=1}^k(B_{\lambda_i}\cdot M_{\lambda_i})_{U_i^{[\lambda_i]}}
\]

The calculation of the latter intersection number involving the Brian\c{c}on variety $B_k$ and $M_k$ inside the Hilbert scheme $U^{[k]}$ of $k$ points is a classical result. By \cite[Theorem 1.1]{ES98}, we have $(B_k\cdot M_k)_{U^{[k]}}=(-1)^{k-1}k$. Applying this result to each factor, we get:
\[
(B_{\lambda_i} \cdot M_{\lambda_i})_{U_i^{[\lambda_i]}} = (-1)^{\lambda_i - 1} \lambda_i.
\]
Therefore, the total intersection number is:
\[
(F_\lambda \cdot F_\lambda)_{\widetilde{N}_\lambda}= \prod_{i=1}^k (-1)^{\lambda_i - 1} \lambda_i= (-1)^{n - k} \prod_{i=1}^k \lambda_i.
\]

   Since each $\lambda_i\leq n+1$ , $ \prod_{i=1}^k \lambda_i$ is not divisible by any prime $p > n+1$. Consequently, the intersection number $(F_\lambda \cdot F_\lambda)_{\widetilde{N}_\lambda}$ is a unit modulo $p$, which implies that the intersection form \eqref{eq:intersect} is non-degenerate modulo $p$.
\end{proof}

\subsection{OG6-type} Let $A$ be an abelian surface and $v_0$ be a primitive Mukai vector with $v_0^2=2$. Let $\widetilde{\cM}(A,2v_0)$ be the crepant resolution of the singular moduli space of sheaves $\cM(A,2v_0)$ and $\widetilde{K}$ be its Albanese fiber. We will prove that the integral cohomology of $\widetilde{K}$ has no $p$-torsion for any odd prime number $p$.
\begin{theorem}\label{thm: TorsionOfOG6}
    Let $\widetilde{K}$ be as above and $p$ is an odd prime number. Then the total cohomology $\rH^*(\widetilde{K},\ZZ_{(p)})$ is torsion free.
\end{theorem}
\begin{proof}
    Our proof is based on the  Mongardi--Rapagnetta--Sacc\`a double cover, which we have already used in \Cref{sec:og6}. By \cite[Section 2]{Motive_OG6}, we have the following diagram:
    \[\begin{tikzcd}
	{\widehat{Y}=\mathrm{Bl}_{\overline{\Delta}}(\overline{Y})} \\
	{\overline{Y}=\mathrm{Bl}_{\underline{\Gamma}}(\underline{Y})} & {\widehat{K}=\mathrm{Bl}_{\widetilde{\Omega}}(\widetilde{K})} \\
	{\underline{Y}} & {\widetilde{K}}
	\arrow[from=1-1, to=2-1]
	\arrow["{\hat{\epsilon}}", from=1-1, to=2-2]
	\arrow[from=2-1, to=3-1]
	\arrow[from=2-2, to=3-2]
	\arrow[dashed, from=3-1, to=3-2]
    \end{tikzcd}\]
    All varieties involved are smooth and projective and moreover the variety $\underline{Y}$ is an irreducible symplectic variety of $K3^{[3]}$-type. The morphism $\hat{\epsilon}$ is a degree two cover. All other morphisms are blow-ups with centers:
    \begin{itemize}
        \item $\widetilde{\Omega}$ consists of 256 disjoint smooth 3-dimensional quadrics;
        \item $\underline{\Gamma}$ is a union of 256 disjoint copies of $\PP^3$;
        \item $\overline{\Delta}=\mathrm{Bl}_{B[2]}(B/\pm1)$, where $B$ is the abelian fourfold $A\times A^\vee$.
    \end{itemize}
    The cohomology $\rH^*(\underline{Y},\ZZ)$ is torsion free by \cite{Ma07} and the cohomology $\rH^*(\underline{\Gamma},\ZZ)$ is torsion free since it is a disjoint union of $\PP^3$. Thus, $\rH^*(\overline{Y},\ZZ)$ is torsion free by the blow up formula. By \cite{Spanier1956}, $\rH^*(\overline{ \Delta},\ZZ)$ is torsion free. So again by the blow up formula, the cohomology $\rH^*(\widehat{Y},\ZZ)$ is torsion free. Since $\hat{\epsilon}$ is a degree two cover, we have $\hat{\epsilon}_*\hat{\epsilon}^*$ is multiplication by $2$, which is invertible after base change to $\ZZ_{(p)}$. Thus, $\hat{\epsilon}^*:\rH^*(\widehat{K},\ZZ_{(p)})\to\rH^*(\widehat{Y},\ZZ_{(p)})$ is injective. Therefore, $\rH^*(\widehat{K},\ZZ_{(p)})$ is torsion free, then $\rH^*(\widetilde{K},\ZZ_{(p)})$ is also torsion free by the blow up formula.
\end{proof}

\subsection{OG10-type}
Let $S $ be a K3 surface over $\CC$ and $v_0$ be a primitive Mukai vector with $v_0^2=2$. Choose a $v_0$ generic polarization $H$. Denote the moduli space of semistable sheaves $\cM_H(S,2v_0)$ by $\cM$ and stable locus by $\cM^{\rm st}$. O'Grady constructed a crepant resolution $\widetilde{\cM}$ of $\cM$, which is a projective irreducible symplectic variety of dimension $10$. We will prove that, for any prime number $p>10$, there is no $p$-torsion in the integral cohomology of $\widetilde{\cM}$. To start, we first recall the geometry of $\cM$ and the resolution.

The moduli space $\cM$ admits a filtration
\[
\cM\supset \Sigma\supset \Omega,
\]
where 
\[
\Sigma=\mathrm{Sing}(\cM)=\cM\backslash\cM^{\rm st}\cong\Sym^2\cM_H(S,v_0)
\]
is the singular locus of $\cM$, and
\[
\Omega=\mathrm{Sing}(\Sigma)\cong \cM_H(S,v_0)
\]
is the singular locus of $\Sigma$.

The crepant resolution is given by three steps:

Step 1: We blow up $\cM$ along $\Omega$, resulting a space $\overline{\cM}$ with an exceptional divisor $\overline{\Omega}$. Denote the strict transform of $\Sigma$ by $\overline{\Sigma}$, which is smooth and satisfies
\(
\overline{\Sigma}\cong \mathrm{Hilb}^2(\cM_H(S,v_0)).
\)

Step 2: We blow up $\overline{\cM}$ along $\overline{\Sigma}$, resulting a space $\widehat{\cM}$ with an exceptional divisor $\widehat{\Sigma}$. In fact, $\widehat{\Sigma}$ is a $\PP^1$-bundle over $\overline{\Sigma}$. Denote by $\widehat{\Omega}$ the strict transform of $\overline{\Omega}$. $\widehat{\cM}$ is a compactification of $\cM^{\rm st}$, with boundary
\(
\partial\widehat{\cM}=\widehat{\Sigma}\cup\widehat{\Omega}.
\)

Step 3: Finally, the extremal contraction of $\widehat{\cM}$ contracts $\widehat{\Omega}$ as a $\PP^2$-bundle to $\widetilde{\Omega}$, which is a 3-dimensional quadric bundle over $\Omega$. The resulting space $\widetilde{\cM}$ is the crepant resolution of $\cM$ we want. Moreover, $\widehat{\cM}$ is isomorphic to the blow up $\mathrm{Bl}_{\widetilde{\Omega}}\widetilde{\cM}$.

\begin{lemma}\label{lem: TorsioninLocalOG10}
    The integral cohomology $\rH^*(\widehat{\Sigma},\ZZ)$, $\rH^*(\widehat{\Omega},\ZZ)$, and $\rH^*(\widehat{\Sigma}\cap\widehat{\Omega},\ZZ)$ are all torsion free.
\end{lemma}
\begin{proof}
    Since $\cM_H(S,v_0)$ is an irreducible symplectic variety of $K3^{[2]}$-type, the integral cohomology $\rH^*(\cM_H(S,v_0),\ZZ)$ is torsion free. We have known that the integral cohomology of  $\overline{\Sigma}\cong\mathrm{Hilb}^2(\cM_H(S,v_0))$ is torsion free by \cite{Totaro2016}. Since $\widehat{\Sigma}$ is a $\PP^1$-bundle over $\overline{\Sigma}$, the cohomology $\rH^*(\widehat{\Sigma},\ZZ)$ is also torsion free.

    For $\rH^*(\widehat{\Omega},\ZZ)$, it suffices to prove that $\rH^*(\widetilde{\Omega},\ZZ)$ is torsion free since $\widehat{\Omega}$ is a $\PP^2$-bundle over $\widetilde{\Omega}$. Note that $\widetilde{\Omega}$ itself is a $3$-dimensional quadric bundle over $\Omega\cong \cM_H(S,v_0)$. Denote by $\pi$ the bundle map.  Consider the Leray-Serre spectral sequence with $E_2^{pq}=\rH^p(\Omega,\mathbf{R}^q\pi_*(\ZZ))$. There are two facts about this spectral sequence.
    \begin{itemize}
        \item $\mathbf{R}^q\pi_*(\ZZ)$ is a trivial local system since the base is simply connected;
        \item For a $3$-dimensional smooth quadric $W$, we have $\rH^i(W,\ZZ)=\ZZ$ for $i$ even and $\rH^j(W,\ZZ)=0$ for $j$ odd.
    \end{itemize}
    In particular, every term at the $E_2$-page is a finite free $\ZZ$-module. By Deligne's theorem \cite{Deligne1968}, every differential $\rmd_r$ at the $E_r$-page satisfies $\rmd_r\otimes\QQ =0$ for $r\geq 2$. This implies that $\rmd_2=0$. Using induction, we can show that $\rmd_r=0$ for all $r\geq 2$. Thus, this spectral sequence degenerates at $E_2$-page and the cohomology $\rH^*(\widetilde{\Omega},\ZZ)$ is torsion free.

   The torsion freeness of cohomology of $\widehat{\Sigma}\cap\widehat{\Omega}$ is given by similar argument, because it is a smooth conic bundle over $\widetilde{\Omega}$.
\end{proof}

The main result of this section is the following theorem:
\begin{theorem}\label{thm;pTorsionFreenessOfOG10}
    Let $p\geq11$ be a prime number, then the cohomology $\rH^*(\widetilde{\cM},\ZZ_{(p)})$ is torsion free.
\end{theorem}
\begin{proof}
    Since $\widehat{\cM}\cong\mathrm{Bl}_{\widetilde{\Omega}}\widetilde{\cM}$, it suffices to prove the torsion freeness of $\rH^*(\widehat{\cM},\ZZ_{(p)})$. The proof is inspired by \cite[Proposition~4.4]{Motive_OG10}.

    By \cite[Proposition~3.2]{Motive_OG10}, there is a decomposition of diagonal over the stable locus. That is, there exist finitely many integers $k_i$ and cycles $\gamma_i\in\CH^{e_i}(\cM^{\rm st}\times S^{k_i})\otimes\QQ$, $\delta_i\in\CH^{d_i}(S^{k_i}\times \cM^{\rm st})\otimes\QQ$, such that
    \[
    \Delta_{\cM^{\rm st}}=\sum \delta_i\circ \gamma_i.
    \]
    The proof of \cite[Proposition~3.2]{Motive_OG10} shows that this decomposition is induced by the Chern character of two universal families and the Todd class of $S$. Thus, an easy computation of the coefficients shows that all denominators have no prime factor greater than $10$. Then, we may assume that $\gamma_i\in\CH^{e_i}(\cM^{\rm st}\times S^{k_i})\otimes\ZZ_{(p)}$ and $\delta_i\in\CH^{d_i}(S^{k_i}\times \cM^{\rm st})\otimes\ZZ_{(p)}$.

    Let $\widehat{\gamma}_i\in\CH^{e_i}(\widehat{\cM}\times S^{k_i})\otimes\ZZ_{(p)}$ and $\widehat{\delta}_i\in\CH^{d_i}(S^{k_i}\times \widehat{\cM})\otimes\ZZ_{(p)}$ be closures of cycles representing $\gamma_i$ and $\delta_i$. Then the support of
    \[
    \Delta_{\widehat{\cM}}-\sum\widehat{\delta}_i\circ \widehat{\gamma}_i
    \]
    lies in the boundary $\widehat{\cM}\times\partial\widehat{\cM}\cup \partial\widehat{\cM}\times\widehat{\cM}$, hence we can write
    \[
    \Delta_{\widehat{\cM}}=\sum\widehat{\delta}_i\circ \widehat{\gamma}_i+Y_{\widehat{\Sigma}}+Y_{\widehat{\Omega}}+Z_{\widehat{\Sigma}}+Z_{\widehat{\Omega}}
    \]
    with $Y_{\widehat{\Sigma}}\in\CH^9(\widehat{\cM}\times\widehat{\Sigma})\otimes\ZZ_{(p)}$, $Y_{\widehat{\Omega}}\in\CH^9(\widehat{\cM}\times\widehat{\Omega})\otimes\ZZ_{(p)}$, $Z_{\widehat{\Sigma}}\in\CH^9(\widehat{\Sigma}\times\widehat{\cM})\otimes\ZZ_{(p)}$, and $Z_{\widehat{\Omega}}\in\CH^9(\widehat{\Omega}\times\widehat{\cM})\otimes\ZZ_{(p)}$.

    For each $i$, the cycles $\widehat{\gamma}_i$ and $\widehat{\delta}_i$ induces maps between cohomology
    \[
    \rH^*(\widehat{\cM},\ZZ_{(p)})\xrightarrow{\widehat{\gamma}_{i,*}}\rH^{*+2e_i-20}(S^{k_i},\ZZ_{(p)})\xrightarrow{\widehat{\delta}_{i,*}}\rH^*(\widehat{\cM},\ZZ_{(p)}).
    \]
    Denote by $j_{\widehat{\Sigma}}$ and $j_{\widehat{\Omega}}$ the corresponding closed embeddings. Then we have
    \begin{align*}
    \rH^*(\widehat{\cM},\ZZ_{(p)})\xrightarrow{Y_{\widehat{\Sigma},*}}\rH^{*-2}(\widehat{\Sigma},\ZZ_{(p)})\xrightarrow{j_{\widehat{\Sigma},*}}\rH^*(\widehat{\cM},\ZZ_{(p)}), \\
    \rH^*(\widehat{\cM},\ZZ_{(p)})\xrightarrow{Y_{\widehat{\Omega},*}}\rH^{*-2}(\widehat{\Omega},\ZZ_{(p)})\xrightarrow{j_{\widehat{\Omega},*}}\rH^*(\widehat{\cM},\ZZ_{(p)}),\\ 
    \rH^*(\widehat{\cM},\ZZ_{(p)})\xrightarrow{j_{\widehat{\Sigma}}^*}\rH^{*}(\widehat{\Sigma},\ZZ_{(p)})\xrightarrow{Z_{\widehat{\Sigma},*}}\rH^*(\widehat{\cM},\ZZ_{(p)}), \\
    \rH^*(\widehat{\cM},\ZZ_{(p)})\xrightarrow{j_{\widehat{\Omega}}^*}\rH^{*}(\widehat{\Omega},\ZZ_{(p)})\xrightarrow{Z_{\widehat{\Omega},*}}\rH^*(\widehat{\cM},\ZZ_{(p)}).
    \end{align*}
    Therefore, the sum of the above compositions
    \[\begin{tikzcd}
	{\rH^*(\widehat{\cM},\ZZ_{(p)})} \\
	{\left(\oplus_i \rH^{*+2e_i-20}(S^{k_i},\ZZ_{(p)})\right)\oplus \rH^{*-2}(\widehat{\Sigma},\ZZ_{(p)})\oplus \rH^{*-2}(\widehat{\Omega},\ZZ_{(p)})\oplus \rH^{*}(\widehat{\Sigma},\ZZ_{(p)})\oplus \rH^{*}(\widehat{\Omega},\ZZ_{(p)})} \\
	{\rH^*(\widehat{\cM},\ZZ_{(p)})}
	\arrow[from=1-1, to=2-1]
	\arrow[from=2-1, to=3-1]
    \end{tikzcd}\]
    is induced by $\Delta_{\widehat{\cM}}$, so it is the identity map. Since the middle term is torsion free, $\rH^*(\widehat{\cM},\ZZ_{(p)})$ is also torsion free.    
\end{proof}

\bibliographystyle{abbrv}
\bibliography{main}

\end{document}